\documentclass[reqno,11pt]{amsart}
\usepackage{etex}
\usepackage{amsfonts}
\usepackage{amscd}
\usepackage{amsbsy}
\usepackage{mathtools}
\usepackage{amsxtra}
\usepackage{amssymb}
\usepackage{epsfig}
\usepackage{epic,eepic}
\usepackage{graphicx}
\usepackage[all]{xy}
\usepackage{psfrag}
\usepackage{amsmath}
\usepackage{amsthm}
\usepackage{setspace}
\usepackage{hyperref}
\usepackage{url}
\usepackage{here}
\usepackage{todonotes}
\usepackage[english]{babel}
\newlength{\halfbls}

\usepackage{xparse}
\usepackage{xspace}
\usepackage[normalem]{ulem}
\usepackage{enumerate}

\usepackage{booktabs}
\usepackage{environ}
\usepackage{blindtext}
\usepackage{caption}

\usepackage{tikz-cd}

\usepackage{tikz}
\usetikzlibrary{patterns}
\usetikzlibrary{calc}
\usetikzlibrary{decorations.pathreplacing,decorations.markings}
\usetikzlibrary{arrows}
\usetikzlibrary{decorations.markings, decorations.shapes}
\usetikzlibrary{positioning}
\usetikzlibrary{intersections}
\usetikzlibrary{decorations.pathmorphing}
\usetikzlibrary{arrows.meta,bending}

\usepackage{thmtools}
\usepackage{etoolbox}

\DeclareRobustCommand{\SkipTocEntry}[9]{}
\DeclareMathOperator*\geht\longrightarrow

\DeclareMathAlphabet\mathbfcal{OMS}{cmsy}{b}{n}
\usepackage{adjustbox}

\theoremstyle{plain}
\newtheorem{Defi}{Definition}[section]  \newtheorem{definition}[Defi]{Definition}
   \newtheorem{proposition}[Defi]{Proposition}
\newtheorem{lemma}[Defi]{Lemma}    \newtheorem{corollary}[Defi]{Corollary}
\newtheorem{theorem}[Defi]{Theorem}

\theoremstyle{remark}

\def\={\;=\;}  \def\+{\,+\,}

\DeclareMathAlphabet{\eucal}{U}{eus}{m}{n}
\DeclareMathAlphabet{\newcal}{U}{dutchcal}{m}{n}

\def\be{\begin{equation}}   \def\ee{\end{equation}}     \def\bes{\begin{equation*}}    \def\ees{\end{equation*}}
\def\ba{\be\begin{aligned}} \def\ea{\end{aligned}\ee}   \def\bas{\bes\begin{aligned}}  \def\eas{\end{aligned}\ees}

\newcounter{savedtocdepth}
\newcommand*{\SaveTocDepth}[1]{%
  \addtocontents{toc}{%
    \protect\setcounter{savedtocdepth}{\protect\value{tocdepth}}%
    \protect\setcounter{tocdepth}{#1}%
  }%
}

\DeclareDocumentCommand{\LMS}{ O{\mu} O{g,n}} {\Xi\overline{\mathcal{M}}_{#2}(#1)}
\DeclareDocumentCommand{\kLMS}{ O{\mu} O{g,n}} {\Xi^k\!\overline{\mathcal{M}}_{#2}(#1)}

\newcommand{\hslashslash}{%
\lapbox[\width]{-0.15em}{\raisebox{.05ex}{%
    \scalebox{.7}{%
      \rotatebox[origin=c]{22}{$-$}%
    }%
  }%
}}

\newcommand{\bslash}{%
  {%
   \vphantom{d}%
   \ooalign{\kern.05em\smash{\hslashslash}\hidewidth\cr$b$\cr}%
   \kern.05em
  }%
}

\newcounter{tmp}
\usepackage{breqn}
\allowdisplaybreaks[2]

\makeatletter

\def\chaptermark#1{}

\def\chapter{%
  \if@openright\cleardoublepage\else\clearpage\fi
  \thispagestyle{plain}\global\@topnum\z@
  \@afterindenttrue \secdef\@chapter\@schapter}

\def\@chapter[#1]#2{\refstepcounter{chapter}%
  \ifnum\c@secnumdepth<\z@ \let\@secnumber\@empty
  \else \let\@secnumber\thechapter \fi
  \typeout{\chaptername\space\@secnumber}%
  \def\@toclevel{0}%
  \ifx\chaptername\appendixname \@tocwriteb\tocappendix{chapter}{#2}%
  \else \@tocwriteb\tocchapter{chapter}{#2}\fi
  \chaptermark{#1}%
  \addtocontents{lof}{\protect\addvspace{10\p@}}%
  \addtocontents{lot}{\protect\addvspace{10\p@}}%
  \@makechapterhead{#2}\@afterheading}
\def\@schapter#1{\typeout{#1}%
  \let\@secnumber\@empty
  \def\@toclevel{0}%
  \ifx\chaptername\appendixname \@tocwriteb\tocappendix{chapter}{#1}%
  \else \@tocwriteb\tocchapter{chapter}{#1}\fi
  \chaptermark{#1}%
  \addtocontents{lof}{\protect\addvspace{10\p@}}%
  \addtocontents{lot}{\protect\addvspace{10\p@}}%
  \@makeschapterhead{#1}\@afterheading}
\newcommand\chaptername{Chapter}

\def\@makechapterhead#1{\global\topskip 7.5pc\relax
  \begingroup
  \fontsize{\@xivpt}{18}\bfseries\centering
    \ifnum\c@secnumdepth>\m@ne
      \leavevmode \hskip-\leftskip
      \rlap{\vbox to\z@{\vss
          \centerline{\normalsize\mdseries
              \uppercase\@xp{\chaptername}\enspace\thechapter}
          \vskip 3pc}}\hskip\leftskip\fi
     #1\par \endgroup
  \skip@34\p@ \advance\skip@-\normalbaselineskip
  \vskip\skip@ }
\def\@makeschapterhead#1{\global\topskip 7.5pc\relax
  \begingroup
  \fontsize{\@xivpt}{18}\bfseries\centering
  #1\par \endgroup
  \skip@34\p@ \advance\skip@-\normalbaselineskip
  \vskip\skip@ }
\def\appendix{\par
  \c@chapter\z@ \c@section\z@
  \let\chaptername\appendixname
  \def\thechapter{\@Alph\c@chapter}}

\newcounter{chapter}

\newif\if@openright

\makeatother

\newtheorem*{corollary*}{Corollary}
\newtheorem*{exercise}{Exercise}
\newtheorem*{remark}{Remark}

\title[Groups of piecewise isometric permutation]{Groups of piecewise isometric permutations of lattice points \\
-- or --\\
\emph{finitary rearrangements of tessellations}}

\begin{document}
\author{Robert Bieri}
\email{bieri@math.uni-frankfurt.de\\rbieri@math.binghamton.edu}
\address{Binghamton University and Goethe--Universit\"at Frankfurt}
\author{Heike Sach}
\email{heike\_sach@yahoo.de}
\thanks{The first named author thanks the referee for an important question (leading to Section~ 4) and for his patience to wait for the answer.\\
And he would like to particularly honor the work of his coauthor who took 
over the leading role in Section 5 and 6 to extending the results of her excellent Diploma thesis (written 22 years before) to the remarkable new generality. It could well be the basis of a PhD }

\begin{abstract}
Through the glasses of didactic reduction: We consider a (periodic) tessellation $\Delta$ of either 
Euclidean or hyperbolic $n$-space $M$. By a 
piecewise isometric rearrangement of $\Delta$ we 
mean the process of cutting $M$ along corank-1 
tile-faces into finitely many convex polyhedral pieces, and rearranging the 
pieces to a new tight covering of the 
tessellation $\Delta$. Such a rearrangement defines 
a permutation of the (centers of the) tiles of $
\Delta$, and we are interested in the group $PI(\Delta)$ of all piecewise 
isometric rearrangements of $\Delta$.

In 
this paper we offer: a) An illustration of piecewise isometric rearrangements in the visually attractive 
hyperbolic plane, 

b) an explanation how this is related to Richard Thompson's groups, 

c) a chapter on the structure of the group pei$(\mathbb Z^n)$ of all piecewise Euclidean rearrangements of the standard tessellation of $\mathbb R^n$ by unit-cubes, and 

d) results on the finiteness properties of some subgroups of pei$(\mathbb 
Z^n)$.
\end{abstract}

\maketitle
\tableofcontents
\SaveTocDepth{1} 

\begin{small}
\begin{footnotesize}
finiteness properties of groups, piecewise isometric infinite permutations, rearrangements of tessellations.  
\end{footnotesize}
\end{small} 

\newpage

\chapter{Introduction}
\section{Generalities and main result}
\subsection {The groups.}\label{the groups}

Let $M$ denote either Euclidean or hyperbolic $n$-space, $N\in \mathbb N$, and \mbox{$\Gamma \le$ Isom$(M)$} a discrete group of isometries of $M$
with the property that $\Gamma$ admits a finite sided convex fundamental polyhedron $D$ with finite volume\footnote{In the hyperbolic case this implies
that $D$ is actually a generalized polytope - see Thm. 6.4.8 in [Ra94].}. 
We aim to study certain groups of permutations of 
the orbit $\Omega: =\Gamma p$, for a given point $p 
\in M$. The major part of this paper is concerned 
with the most down-to-earth case when $
\Omega:=\mathbb{Z}^n$, viewed as the set of tile 
centers of the tessellation dual to the standard 
tessellation of Euclidean $\mathbb{R^n}$ by unit 
cubes.

To define the notion of a piecewise $\Gamma$-isometric permutation $\pi: \Omega \rightarrow~\Omega$ requires a notion of $\Gamma$\textit{-polyhedral pieces} of 
$\Omega$ on which $\pi$ should be isometric, and it is reasonable to require that the geometry of these pieces be related to the geometry of $\Gamma$.
Thus, together with the base point $p \in M$ we choose a finite set $\mathcal{H}$ of ``$\Gamma$-\textit{relevant}'' \textit{closed
half-spaces} of $M$, and the resulting groups will - to some extent - depend on this choice: We fix a (finite sided convex) fundamental polyhedron 
$D$ and take $\mathcal{H}$ to be an irredundant finite set of half spaces 
with the property that D is the intersection $D = \bigcap_{H\in \mathcal{H}}H$ and each member of $\mathcal{H}$ has its boundary spanned by a side of $D$.
\par
By a \textit{convex $\Gamma$-polyhedral} subset $P$ of $M$ we mean any finite intersection of 
$\Gamma$-translates H$\gamma$, where $\gamma \in\Gamma$ and $H\in \mathcal{H}$. And a general $\Gamma$-\textit{polyhedral} subset of $M$ is a finite 
union of convex ones. By abuse of language, we call the intersection $S = 
\Omega\cap P$ a \textit{(convex)} $\Gamma$-\textit{polyhedral piece of} 
$\Omega$ whenever $P\subseteq M$ is a (convex) $\Gamma$-polyhedral subset.

\textbf{Definition.}  Let $S\subseteq\Omega$ be a $\Gamma$-polyhedral set, and $\Gamma^\ast\leqslant\Gamma $ a subgroup. A permutation $g: S \to S$ is said to be \emph{piecewise
$\Gamma^\ast$-isometric} if $S$ can be written as a disjoint union of finitely many $\Gamma$-polyhedral pieces $S = S_1 \cup S_2 \cup... \cup S_k$ 
with the property that the restriction of $g$ to each $S_i$  is also the restriction of an isometry $\varphi _i\in \Gamma^\ast$. 

We write $G_{\Gamma^\ast}(S)$ for the group of all piecewise isometric $\Gamma^\ast$-permutations of $S$. The permutations in $G_{\Gamma^\ast}(S)$ 
with finite support form a normal subgroup of $G_{\Gamma^\ast}(S)$ which we denote by ${\rm sym}(S)$; the quotient group $$G_{\Gamma^\ast}(S)/{\rm 
sym}(S)$$ is often particularly interesting.

\textbf{Remark} 1) Particularly nice is the situation when $M$ comes with 
a regular tessellation. In that case we take $\Omega$ as the centers of mass of the tiles, and~ $G(\Omega)$ 
could be viewed (and termed) as the \textit{group of all piecewise isometric tile-rearrangements}: Here, $\Gamma$ is the group of all isometries of $M$ compatible with the tessellation, and  $\mathcal{H}$ as the set of half-spaces bounded by the span of a corank-1 face of a tile-fundamental domain.

2) In a recent preprint \cite{fh20} Farley and Huges present a promising general abstract approach to the finiteness properties of what they call \textit{locally defined groups}. In their terminology our piecewise isometric permutations are locally defined by isometries and hence appear as a 
special case. The authors obtain unified proofs for (the positive direction of) type $F_n$, for several generalized Thompson groups, but they add the remark that our examples \cite{bisa16} appear to pose a more substantial challenge.
\\[1em]
In this paper we consider the group $G_{\Gamma^\ast}(S)$ in two special cases:
\par
a) When $M=\mathbb H^2$ is the hyperbolic plane we 
consider \emph{triangle groups} and their 
orientation preserving subgroups $\Gamma^{\ast}
\leqslant\Gamma$ acting on the tessellation $\Delta$ 
by the $\Gamma$-translates of a hyperbolic triangle $D$. In the special case when $D$ is the \emph{ideal} triangle (all three vertices at infinity) the quotient $G_{\Gamma^\ast}(\Omega)/{\rm sym}(\Omega)$ is Richard Thompson's groups $V$. In the more general case when $D$ has at least one vertex at infinity we can assume that one of these corresponds to the point 
$\infty\in\partial\mathbb H^2$ in the upper half plane model, and that all tile-vertices of $\Delta$ in  $\partial\mathbb H^2$ correspond to rational numbers.

We show that in this situation the group $G_{\Gamma^\ast}(S)$ has a description in terms of the spine $T$ of the tessellation $\Delta$, which is a 
bipartite tree. This description can be used to prove finiteness properties of $G_{\Gamma^\ast}(\Omega)/{\rm sym}(\Omega)$ if and only if $\Gamma$ 
contains no hyperbolic elements with rational fixed points.  We also outline how one could attack the general case.

b) Our main concern then is the case when $M=\mathbb R^n$ is Euclidean $n$-space, $\Gamma$=Isom$(\mathbb Z^n)$, and $\Gamma^\ast$ either equal 
to $\Gamma$ or its translation subgroup $T\leqslant\Gamma$. We call the $\Gamma$-polyhedral pieces $S\subseteq\mathbb Z^n$ the \emph{orthohedral} subsets of $\mathbb Z^n$, and consider the \emph{piecewise Euclidean isometry} groups ${\rm pei}(S) = G_{\Gamma}(S)$ and its subgroup $\textrm{pet}(S)\leqslant{\rm pei}(S)$,
the \emph{piecewise Euclidean translation} groups $ G_T(S)$ of arbitrary orthohedral subsets $S\subseteq \mathbb Z^n$. 

If $S$ is the disjoint union of $h$ copies of $\mathbb{N}$ then the pet-group $pet(S)= G_T(S)$ is 
Houghton's 
group $H_n$ \cite{ho78}. Known for more 
than 38 years was also the pet group pet$(S)$ when 
$S=\bigcup_{1\leq i\leq h}\mathbb N^2$ is a 
disjoint union of $h$ quadrants: This was the topic of the second author's diploma thesis \cite{sa92} 
in
which she proved, among other things, that 
$pet(S)$ 
is of type $F_{h-1}$  (see 
Subsection \ref{history} for more information). 
\par
The fact that our groups have prominent relatives is not our only motivation: In Chapter 3 we make an effort to analyze the structure of pei$(S)$, 
and this culminates at the end of Section  \ref{structure of pei(S)} with 
full information on the normal subgroup lattice of $pei(S)$. And in Chapter 4 we get concrete information on finiteness properties (finite presentability and high finiteness length -- see Subsection \ref{finiteness length}) of  ${\rm pet}(S)$ and ${\rm pei}(S)$. Thus, here is a new playground --  prominently located in a good neighborhood -- to studying the interaction between structure and finiteness properties.
  From the tree-hyperbolic world where the monsters live (like Thompsons' 
group $V$), we have gotten used to seeing many examples which are simple groups of type $F_{\infty}$. What we find in our Euclidean analogue is similar, but interestingly different: Instead of simplicity we find the \emph{Bottleneck Theorem} (Theorem \ref{Bottleneck Theorem}) which excludes hidden normal subgroups; and instead of $F_\infty$ we find, e.g., that $pei(\mathbb Z^n)$ is of type $F_{2^n-1}$ -- for the main results see Subsection \ref{results}.

\subsection{The finiteness length of a group.} \label{finiteness length}
Every group is \emph{of type} $F_0$; every finitely generated group is \emph{of type} $F_1$; every finitely presented group
(equivalently: every fundamental group $\pi_1(X)$ of a finite cell complex $X$) is \emph{of type} $F_2$ ; and $\pi_1(X)$ is \emph{of type} $F_m$
($m\geq2$) if $X$ is a finite cell complex and $\pi_1(X) = 0$, for all $i$ with $2\leq i<m$.  

Ten years after C.T.C. Wall introduced these finiteness properties, Borel 
and Serre \cite{bs73}/\cite{bs76} showed that all semi-simple
S-arithmetic groups have special homological features; in particular they 
are of type $F_{\infty}$ (equivalently, type $F_m$ for all $m \in \mathbb 
N$). And
this was only the first of a number of important infinite families of groups that turned out to be of type $F_{\infty}$ in the following decades; many of them, just like arithmetic
groups, in the center of mainstream group theory: automorphism groups of free groups \cite{cv86}, Richard Thompson's groups \cite{bg84}, etc.
More recent results in this direction are based on Ken Brown's topological discrete Morse theory technique \cite{br87} and its powerful
CAT$(0)$-version of Bestvina-Brady \cite{bb97}.

The insight that many important groups have much further reaching finiteness properties than finite presentability is great progress - but having ``good'' finiteness properties is only one 
side of the concept: The focus on the  \emph{finiteness length} function $fl: \mathbf{Gr}
\rightarrow \mathbb N\cup \{0,\infty\}$, defined on all groups~ $G$ by
\[fl(G) := \{\sup{m~|~G~ \mbox{is of type}~ F_m}\}\]
takes both sides into account. Analogous \emph{algebraic length functions} $afl_A$ are defined for every $G$-module $A$,
to be the supremum of all non-negative integers $m$ with the property that $A$ admits a free resolution which is finitely
generated in  all dimensions $\leq m$. The functions $afl_A$  have the considerable advantage that they extend
immediately to monoids $G$. We write $afl$ for $afl_{\mathbb Z}$, where $\mathbb Z$ stands for the infinite cyclic group with the trivial
$G$-action; by the Hurewicz Theorem we know that $afl$ coincides with $fl$ on all finitely presented groups (i.e., whenever $fl(G) \geq 2)$. An important feature of both $fl$ and $afl_A$ is that they are constant on 
commensurability classes of groups.

In general, the finiteness length of a group is notoriously difficult to compute. Nevertheless, to study and interpret
accessible parts of the pattern that these functions carve into group theory can be very fruitful. A convincing example is
the following: If we fix a finitely generated group $G$, then the function $Hom(G,\mathbb R_{add}) \rightarrow \mathbb N 
\cup\{0,\infty\}$, which associates with each homomorphism $\chi: G \rightarrow \mathbb R_{add}$ the value of $afl_A$ on the 
submonoid $\chi^{-1}([0,\infty)) \subseteq G$, imposes in the finite dimensional $\mathbb R$-vector space $Hom(G,\mathbb
R_{add})$ the pattern exhibited by the homological $\Sigma$-invariants $\Sigma^k(G;A)$ of \cite{br88}. On the other hand, we can 
also evaluate $fl$ and $afl_A$ on the commensurability classes of subgroups containing $G'$, and this yields patterns on
the rational Grassman space of $\mathbb Q$-linear subspaces of $G/G'\otimes\mathbb Q$ (which parametrizes these classes).
The core of the main $\Sigma$-results of \cite{bns86}, \cite{br88}, \cite{r89}, \cite{bge03} consists then of exhibiting
the precise relationship between the two patterns. 
\par
An intriguing point is that in all computable examples the finiteness length patterns have a polyhedral flavor: they
turn out to be expressible in terms of finitely many inequalities. One of 
the few general results here, polyhedrality of 
$\Sigma^0(G;A)$ when $G$ is Abelian, was proved in \cite{bgr84} by methods which were later partly re-detected in tropical
geometry. But polyhedrality questions on $\Sigma^k(G;A$) for non-Abelian $G$ and $k>0$ are wide open.

\subsection{The results.}\label{results} Chapter 2 has two goals: it illustrates piecewise 
isometric permutations in the  visually attractive 
area of two dimensional hyperbolic geometry, and 
it links our piecewise isometric permutations to the group theory revolving around Thompson's groups. 

We mentioned already that for non-cocompact triangle groups $\Gamma$ we can express 
$G_{\Gamma^\ast}(\Omega)$ in terms of the spine $T$ of $\Delta$. This exhibits $G_{\Gamma^\ast}(\Omega)$ as a permutation group on the vertices of 
the tree $T$. We discuss whether this action respects almost all edges and cyclic star-orderings of $T$ (following the terminology of  \cite{leh08}, \cite{ls07}, \cite{bmn13}, and \cite{nst15}, this would be an action by \emph{quasi planar-tree automorphisms}). 

We find: If $\Gamma$ has signature $[\infty,\infty,\infty]$ then $T$ is the dyadic tree and the action of $G_{\Gamma^\ast}(\Omega)$ on it is the one that has always been used to describe the elements of Thompson`s groups in terms of generators (and is obviously quasi isometric). In the general case the action is always by piecewise planar-tree isometries, and by quasi-isometries if and only if the signature is $[p,q,\infty]$, with at least one of $p, q$ infinite or odd.

Chapter 3 and 4 are about the Euclidean case -- more precisely, we restrict attention to the case when $M$ is Euclidean and carries the standard tessellation by unit cubes, i.e., $\Gamma={\rm Isom}(\mathbb Z^n)$, and $\Gamma^\ast$ is either equal to $\Gamma$ or its translation 
subgroup $T\leqslant\Gamma$, and the goal is to 
make first steps towards evaluating the finiteness 
length functions $fl$ and $afl$ on what we like to 
view as the \emph{pei- and pet-clouds} around Isom$
(\mathbb Z^n)$, resp.~ $\mathbb Z^n$: the groups $
\textrm{pei}(S)=G_\Gamma(S)$, resp. 
${\rm pet}(S)=G_T(S)$, as $S$ runs through all orthohedral subsets of some $\mathbb Z^N$.

To state the main results requires the following notation: By an \emph{orthant of rank} $n$ ($n\in\mathbb N$) we mean any subset $L\subseteq\mathbb Z^N$ isometric to the standard rank-$n$ orthant $\mathbb N^n$. Each
orthohedral set $S\subseteq\mathbb Z^N$ is the disjoint union of finitely 
many orthants $S = L_1\cup L_2\cup \ldots\cup L_k.$

 \textbf{Definition (rank and height)} (By the \emph{rank} of $S$, denoted ~$n={\rm rk}S$, we mean the maximum rank of the orthants $L_i$; and the \emph{height} of $S$, denoted~ $h(S)$, is the number of orthants of rank ${\rm rk}S$ among the $L_i$. 
 
 One observes that the orthohedral sets 
with the piecewise Euclidean-isometric maps between 
them (called \textit{pei-maps}) form a category. 
Clearly, the pei-isomorphisms are the bijective pei-isometries; and in Section 3.3 we observe that orthohedral sets are pei-isomorhic if and only if their rank and height agree.

Chapter 3 starts with introducing these basic concepts then turns to analyzing the group theoretic structure of $G:=\emph{pei}(S)$ for an arbitrary orthohedral subset $S\subseteq\mathbb Z^N$. The results are summarized in some detail at the end of Section \ref{subsection4.1}. The key here is a structure at infinity of the orthoedral set $S$ -- analogous to the structure at infinity of the tessellated hyperbolic plane which was used above  to relate groups of piecewise hyperbolic isometries to Thompson's groups. This structure at infinity of $S$ consists of

\begin{enumerate}
\item a rank-graded $G$-set $\Gamma^\ast (S)=
\bigcup_k\Gamma^k(S)$, called the set of \textit{germs of} $S$,
\item 
a family of rank-k cosets $\langle\gamma\rangle\subset\mathbb{Z}^N$ attached at the germs $\gamma\in\Gamma^k(S)$ which we call the \textit{rank-k 
tangent coset of} $\gamma$.  (The product 
$\prod_{\gamma\in\Gamma^k(S)}\langle\gamma\rangle$ can be viewed as a rank-k \textit{tangent space of} $S$ at infinity),
\item an induced action of $G$ on $\Gamma^k(S)$, 
and an induced action of $G$ on each $\langle\gamma\rangle$ 
by isometries $g_\gamma : \langle\gamma
\rangle\rightarrow\langle\gamma g\rangle$.
\end{enumerate} 
The definition of germs is 
in Section \ref{Germs of orthants}, and their tangent cosets crop up first in Section
\ref{Stabilizers of rank-k germs}. Here we mention merely: 
\begin{itemize}
\item Two orthants $L,L'\subset\mathbb{Z}^N$ are \textit{commensurable} if $L, L'$ and $L\cap L'$ have the same rank, and the commensurability classes represented by rank-k orthants $L\subset S$ are the  germs $\gamma\in\Gamma^k(S)$.
\item The tangent coset $\langle\gamma\rangle$ of $\gamma\in\Gamma^k(S)$ is the union of all members of the commensurability class $\gamma$ and thus $\langle\gamma\rangle\cong\mathbb Z^k$. 
\item
For any given pair $(\gamma,g)\in\Gamma^kS)
\times G$, an orthant $L$ representing $\gamma$ can 
be chosen sufficiently far out to make that the 
restriction of $g$ to $L$ is an isometry $g
\lvert_L: L\rightarrow Lg$ (see Lemma 
\ref{lemma3.3}). This restriction defines the action on both $\Gamma^k(S)$, and $T$.
\end{itemize}

Single elements $g\in G$ are supported on a finite 
set of orthants, and the maximum rank of these 
orthants is the \textit{rank of $g$}, denoted rk$
(g)$. From this we infer that $g_\gamma$ is the 
identity if and only if rk$(g)<rk(\gamma)$. 

Now we 
consider the normal series
\begin{equation*}
1=G_{-1} ~\leqslant~ G_0 ~\leqslant ~G_1~ 
\leqslant~ ...~ \leqslant~ G_k~ \leqslant~ ... ~\leqslant~ G_n~ =~ G,
\end{equation*}
where the \textit{rank-k subgroup} $G_k$ consist of all elements $g\in G$ of rank $\leq k$. 

Its factors $G_k/G_{k-1}$ exhibit clear footprints of the structure of Isom$(\mathbb Z^n)$ which is exhibited by the refinement
\begin{equation*}
G_{k-1}\leqslant C^{ord}(\Gamma^k(S))\leqslant 
C(\Gamma^k(S))\leqslant G_k,
\end{equation*}
where $C(\Gamma^k(S))$ consists of all elements $g
\in G$ which fix all rank-k germs of $S$, and 
$C^{ord}(\Gamma^k(S))$ consists of all $g\in G$ with the property 
that, in addition, for all germs $\gamma\in 
\Gamma^k(S)$ the isometry $g_\gamma: \langle\gamma\rangle\rightarrow\langle\gamma g\rangle$ is a translation. 

We prove that the quotient $A_k:=C^{ord}(\Gamma^k(S))/G_{k-1}$ is free-Abelian (of rank $\infty$ for $k<n)$; $C(\Gamma^k(S))$ is the direct product of symmetric groups of degree k, and $G_k/C(\Gamma^k(S))$ is the finitary symmetric group of degree $\lvert\Gamma^k(S)\rvert$ -- for more details see Theorem \ref{classifies G_k(S)/G_k-1(S)}.

In 
particular, $G$ is elementary amenable, and in the $G_k/C^{ord}(\Gamma^k(S))$-module $A_k$ we can track a congruence-
subgroup type property -- see Section \ref{subsection4.7}.

The Bottleneck Theorem (Theorem \ref{Bottleneck Theorem}) in Section \ref{subsection4.9} finally shows that the rank-subgroups $G_k$ are characteristic; and, up to an index 2, all normal subgroups of $G$ can be tracked by the ones in the quotients $G_k/G_{k-1}$.

In Chapter \ref{The Finiteness Length} we use Brown's approach in \cite{br87} to compute the finiteness lengths of pet$(S)$ and a lower bound on the finiteness lengths of pei$(S)$. Just as in Brown's paper each $fl$-result comes together with a parallel $afl$-result, hence our results have the same feature. We found 

\begingroup
\setcounter{tmp}{\value{theorem}}
\setcounter{theorem}{0} 
\renewcommand\thetheorem{\Alph{theorem}}

\begin{theorem}
Rank and height of an orthohedral set $S$ determines the group \emph{pei}$(S)$ up to isomorphism, and we have  $\mbox fl({\rm pei}(S))\geq h(S)-1$;  in particular, $fl\big({\rm pei}(\mathbb Z^n))\geq 2^n - 1$.
\end{theorem}
For a more precise result see Section \ref{theorem5.1}.

The exact value of $fl\big({\rm pei}(\mathbb Z^n))<\infty$ remains a challenging open problem.
In the pet-case we know more: The isomorphism class of ${\rm pei}(S)$ is not determined by rank and height of the orthohedral set $S$. But we do have a precise result for the special case when $S$ is a stack of $h$ parallel orthants of the same rank:

\begin{theorem}
If $S$ is a stack of $h$ rank-n 
orthants then $fl\big({\rm pet}(S)\big) = h(S) - 1.$
\end{theorem}
\endgroup

A generalization to a stack of $k$-skeletons of an orthant is in Theorem \ref{stack of n-skeletons}.

\begin{remark}
Highly complex elementary amenable groups with high finiteness length can 
also be constructed in terms of permutational wreath products $A\wr_X B$. 
Bartholdi, de 
Cornulier, and Kochloukova \cite{bck157} provide the technique to compute 
$fl(A\wr_X B)$ in favorable situations; and Kropholler-Martino \cite{km16} apply this to construct a sequence of wreath powers of Houghton's group 
$H_n$, $P(m):=H_n( \wr_X  H_n)^m$ with constant finiteness length $fl(P(m))=fl(H_n)=n-1$ for all m. Thus they take $fl(H_n$) for granted and 
provide a method to increase the complexity of $H_n$, whereas in the present work we extend Brown's computation of $fl(H_n)$ to new groups which are 
poly-(locally Houghton-by-finite) and analyze their structure.
\end{remark}

\subsection{Outlook.}

Let $\Gamma$ be a discrete group of (Euclidean or hyperbolic) isometries with polyhedral fundamental domain of finite volume. By generalizing the
definition of the group ${\rm pei}(\mathbb Z^n)$ to the groups G$_\Gamma(\Omega)$ of all piecewise $\Gamma$-isometric permutations of the orbit $\Omega =\Gamma p$, we have endowed each such group $\Gamma$ with the $G_\Gamma$-\emph{cloud} of all piecewise $\Gamma$-isometric permutation groups
$G_\Gamma(S)$ where $S$ runs through the $\Gamma$-polyhedral subsets of $\Omega$.  The success with evaluating the finiteness length function on the clouds around
Isom$(\mathbb Z^n)$ and $\mathbb Z^n$, together with the observation that 
the groups around $SL_2(\mathbb Z)$ are closely related to the highly 
respected Thompson groups, indicates that finding more of this might be a difficult but worthwhile program.
\par
Particularly promising projects would be:
\begin{enumerate}[i)]
\item 
Finding the phi$(\Omega)$ when $\Omega$ is given by a regular tessellation of the hyperbolic plane, and the precise relationship between the induced group on the boundary, phi$(\Omega)/{\rm sym}(\Omega)$, and Thompson's 
groups.
First steps in this direction based on (a slight generalization of) Theorem \ref{Refining tessellations by convex pieces} are suggested in Section 
\ref{Connection with Thompson's groups}.
\item There are strong indications that $fl(pei(S))>h(S)-1$. In particular, Thomas Kilcoyne has a proof that $pei(S)/{\rm sym}(S)$ is finitely presented if $S$ is a stack of at least 2 quadrants. Thus, progress in this direction seems accessible -- whether  $pei(S)/{\rm sym}(S)$ is better behaved than  $pei(S)$ itself remains to be seen. For the Houghton groups this is trivially true, and the subtle difference between $QV$ and $\Tilde{QV}$ (see \cite{nst15}) might indicate that this is indeed the case. 
\item Defining and studying a group pal$(\mathbb{Z}^n)$ of \emph{piecewise afine-linear permutations} on $\mathbb{Z}^n$, and find the footprints of the structure of $SL_n(\mathbb{Z}^n)$ in its structure.

\end{enumerate}

\subsection{Remark on the history of this paper.}\label{history}

Houghton originally introduced his groups in~\cite{ho78}. Theorem~B, in the Houghton
group case, i.e., when S is a stack of rays, is due to Brown \cite{br87}, 
and we follow his footsteps.

The inequality $fl\big({\rm pet}(S)\big) \ge h(S) - 1$ in the rank-2 case 
when~$S$ is a stack of quadrants (as well as the equality for a certain ``diagonal subgroup'' of ${\rm pet}(S)$) is due to the second author and appears in her Diploma thesis (Frankfurt 1992 \cite{sa92}), to which the first author contributed little more than the
definition\footnote{Influenced by Peter Greenberg's courage to define $SL_2(\mathbb Z)$-geometry (\cite{gr93}) -- which is similar to but different from ours.} of the group. Her Diploma thesis could have been the starting point of a promising PhD project -- but she preferred starting a true-to-life career in software development. 

Back then, hunting for further generalizations of such groups was not the 
first author's priority either -- they looked artificial and in those days only of use as counterexamples to questions that nobody asked. Therefore the project went dormant for 22 years, until an increasing number of publications on Houghton's groups (\cite{abm14}, \cite{bcmr14}, \cite{st15}, \cite{leh08}, \cite{za15} etc.) suggested that Heike Sach's Diploma thesis~\cite{sa92} should be published, translated, and generalized. We started our collaboration in 2014. 

The (back then surprising) insight that our groups are not only \emph{generalized Houghton groups} but fit in a more interesting general (Euclidean or hyperbolic) geometric framework, which could be described as the groups of \emph{tile-permutations induced by finitary rearrangements of tessellations}, was added when we put the preprint (together with the original Diploma thesis) on the arXiv in June 2016 \cite{bisa16}, \cite{sa92}.

\chapter{On the hyperbolic case}

\section{Planar hyperbolic examples}\label{Planar hyperbolic examples}

\subsection{Piecewise $\Gamma$-hyperbolic triangle groups}
\label{section2.1}
Let $D$ be a hyperbolic triangle of finite area 
in the compactified (Poincaré disk model of the) 
hyperbolic plane $\mathbb{H}^2\cup\partial
\mathbb{H}^2$. We write $v_i$ for the vertices and $e_i$ for edges of ~$D$, $1\leq i\leq 3$, and use the convention that $v_i$ is oposite to $e_i$. We write $H_i\subset\mathbb H^2$ for the half-plane which contains $D$ and is bounded by the line spanned by $e_i$, and we put $\mathcal{H}:=\{H_1,H_2,H_3\}$ to be the irredundant finite set of half-spaces as in the definition in Subsection \ref{the groups}.

Regardless whether some of the $v_i$ 
are on $\partial\mathbb{H}^2$, the triangle $D$ 
has a unique inscribed circle (exhibited in the fundamental triangle in Figure~\ref{(4,4)_inf} and Figure~\ref{Hyperbolic triangle II (inf, inf, inf)}). We will take its hyperbolic center as the base point $p\in D$ and call it the tile-center of $D$. Let $t_i\in e_i$ denote the touching point of the inscribed circle on the edge $e_i$. Elementary geometric arguments show that if $e_i, e_j$ are the two edges emanating from $v_k$ then there is a hyperbolic disk ~$B_k$ centered at $v_k$ 
 which has $t_i$ and $t_j$ on its boundary (if $v_k\in\partial\mathbb{H}^2$ then $B_k$ is understood to be a horodisk centered at $v_k$).

We assume that the angle of $D$ at $v_i$ is $
\frac{\pi}{q_i}$, where $q_i\in \mathbb N\cup\{\infty
\}$ with $\frac{1}{q_1}+\frac{1}{q_2}+\frac{1}{q_3}
<1$. Then the hyperbolic reflections $\sigma_i$ over 
the edges $e_i$ define a particularly nice 
tessellation on $\mathbb{H}^2$ and generate the 
isometry group $\Gamma$. The triple  $[q_1,q_2,q_3]$ is the \emph{signature} of the triangle group $\Gamma$. $\mathbb{H}^2$ is now endowed with a simplicial $\Gamma$-complex $\Delta$. Here are some elementary facts: \\
1. If $P\subset\mathbb{H}^2$ is $\Gamma$-polyhedral, so is the closure of 
its complement. \\
2. Each edge $e$ of $\Delta$ spans a hyperbolic line $h[e]$ in the 1-skeleton $\Delta^1$. \\
3. A hyperbolic line $h\subset\Delta^1$ is tessellated by infinitely many 
finite edges if and only if $h$ is a \emph{hyperbolic axis} (the axis of a hyperbolic element $g\in\Gamma$).

We put $\Omega=\Gamma p$ and are interested 
in the group $G(\Omega)=\mathrm{phi}(\Omega)$ of all 
piecewise $\Gamma$-isometric permutations, and in 
subgroups
$G_{\Gamma^\ast}(\Omega)\leqslant G(\Omega)$, when $\Gamma^{\ast}\leqslant\Gamma$ is a specified subgroup of $\Gamma$. 
$\mathbb H^2$ is now equipped with three $\Gamma$-orbits $\Gamma B_i$ of disks  centered at the tile-vertices $gv_i$; these disks touch each other, and their mutual touching points coincide with the points $gt_i$. The hyperbolic segments connecting the tile centers of edge-neighboring tiles cross vertically at the points $gt_i$ through the edge $ge_i$ of $\Delta$ 
and constitute the edges of the dual tessellation $\Delta^\ast$ of $\Delta$. 
The dual tile with center $gv_i\in\mathbb H^2$ is a convex $2q_i$-gon around the inscribed disk $gB_i$ (in the case when $gv_j\in \partial\mathbb H^2$ this is the area bounded by a doubly infinite sequence of finite dual edges tangent to the horodisk $gB_j$ with center $gv_j$).

The charm of $\Delta^\ast$ is that $\Omega$ stands for the tiles, and the 
edges of $\Delta^\ast$ indicate  how tiles are glued together. This opens 
the possibility that the $\Gamma$-polyhedral pieces can be described in terms of the 1-skeleton of $\Delta^\ast$. If $D$ is compact this remais a challenge to be addressed somewhere else.

\subsection{The case when $\Gamma$ is a non-cocompact triangle group}

The 
situation is simpler when $D$ is not compact, and from now on we assume that at least $q_3=\infty$: The 
point is that we have now a $\Gamma$-equivarant 
retraction of the hyperbolic plane $\mathbb H^2$ along 
the hyperbolic lines emanating out of the horodisk 
centers $gv_j$ for $g\in\Gamma$ and $v_j\in\partial
\mathbb H^2$, and terminating on the boundaries of 
the horodisks $gB_j$. The truncated hyperbolic plane
$$\mathbf T~:=~\mathbb H^2~~-~\Big(\bigcup_{g\in\Gamma,~v_j\in \partial\mathbb H^2}\rm{Int}(gB_j)\Big),$$ 
obtained by excision of all open horodisks 
$gB_j$, closely approximates the finite part $\Delta^\ast_{fin}$ of the complex $\Delta^\ast$. $\mathbf T$ is a tree shaped union of bands meandering between the horodisks $gB_l$ towards  $\partial\mathbb H^2$, and it contains the finite part of the dual tessellation $\Delta^\ast_{fin}\subset\mathbf T$ is the cell complex with vertex set $\Omega=\Gamma p$, all dual edges of length equal to the diameter of the inscribed circle of $D$, and 2-cells semi regular $2q_i$-gons around the finite disks $gB_i$. 

That 
$\mathbf T$ is tree shaped can be seen by referring to either the retraction of $\mathbb H^2$ onto $\mathbf T$, or to the fact that whenever a band enters an area bounded by two horodisks through its tight entrance there is no escape on a different route through another exit.

\emph{ILLUSTRATIONS with signature} $[2,3,\infty],
[3,\infty,\infty]], [\infty,\infty,\infty]$ \emph{are exhibited in Figure~\ref{(4,4)_inf}, Figure~\ref{((6, inf,inf)(1)} and Figure~\ref{Hyperbolic triangle II (inf, inf, inf)}. The coloured part of the pictures exhibits the set 
$\mathbf T$ (the finite disks $B_i$ in green/blue, and the rest of $\mathbf T$ consists of little 
triangular red shapes, each containing exactly one 
point of $\Omega$). As the neighboring ones touch 
each other at a kissing point this red part of $
\mathbf T$ actually contains and outlines the 1-
skeleton of $\Delta^\ast_{fin}$}.

\begin{figure}
\centering
\includegraphics[scale=0.2]{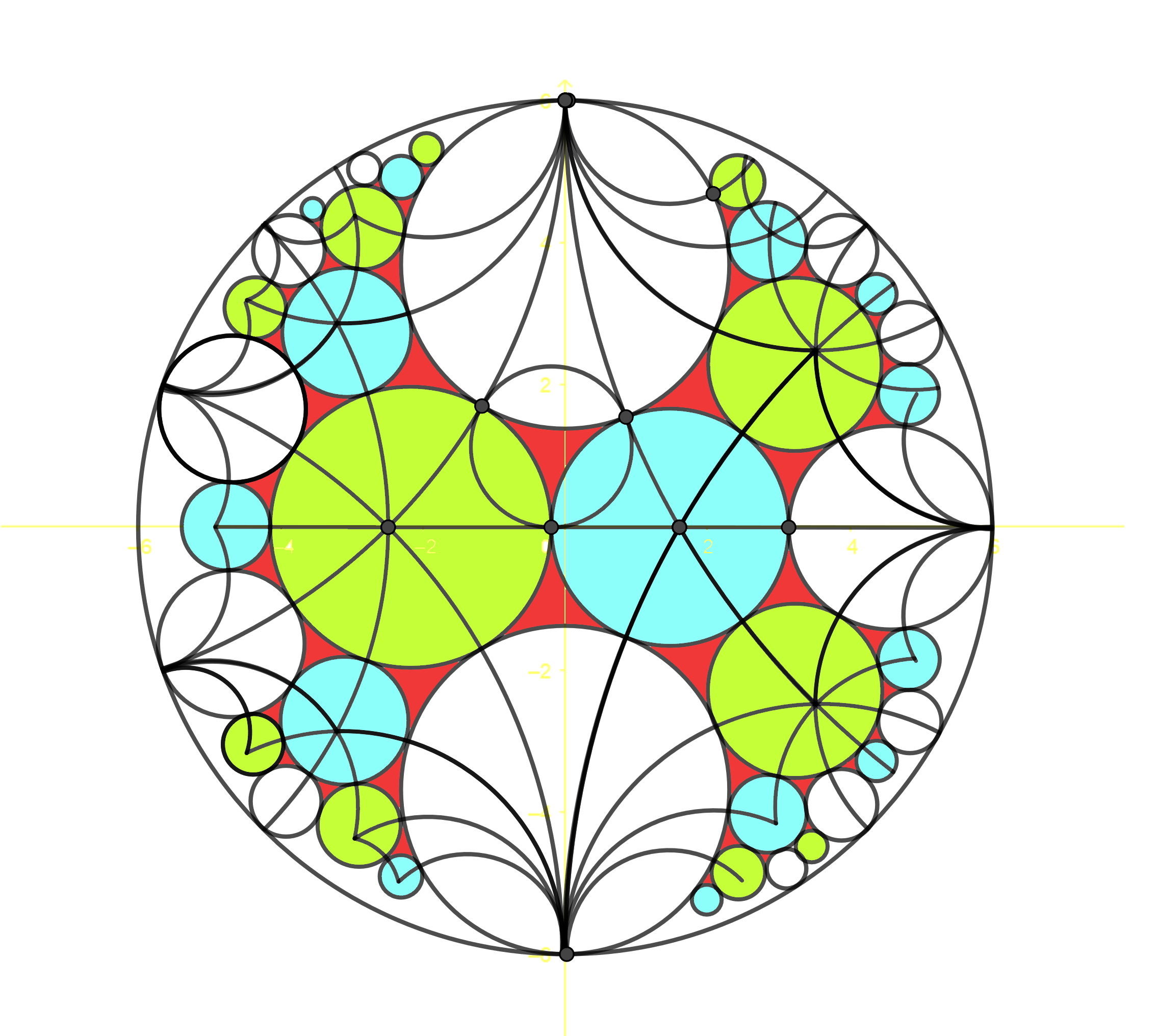}
\caption{\label{(4,4)_inf}The signature $[3, 4, \infty]$ case.}
\end{figure}

\begin{figure}
\centering
\includegraphics[scale=0.2]{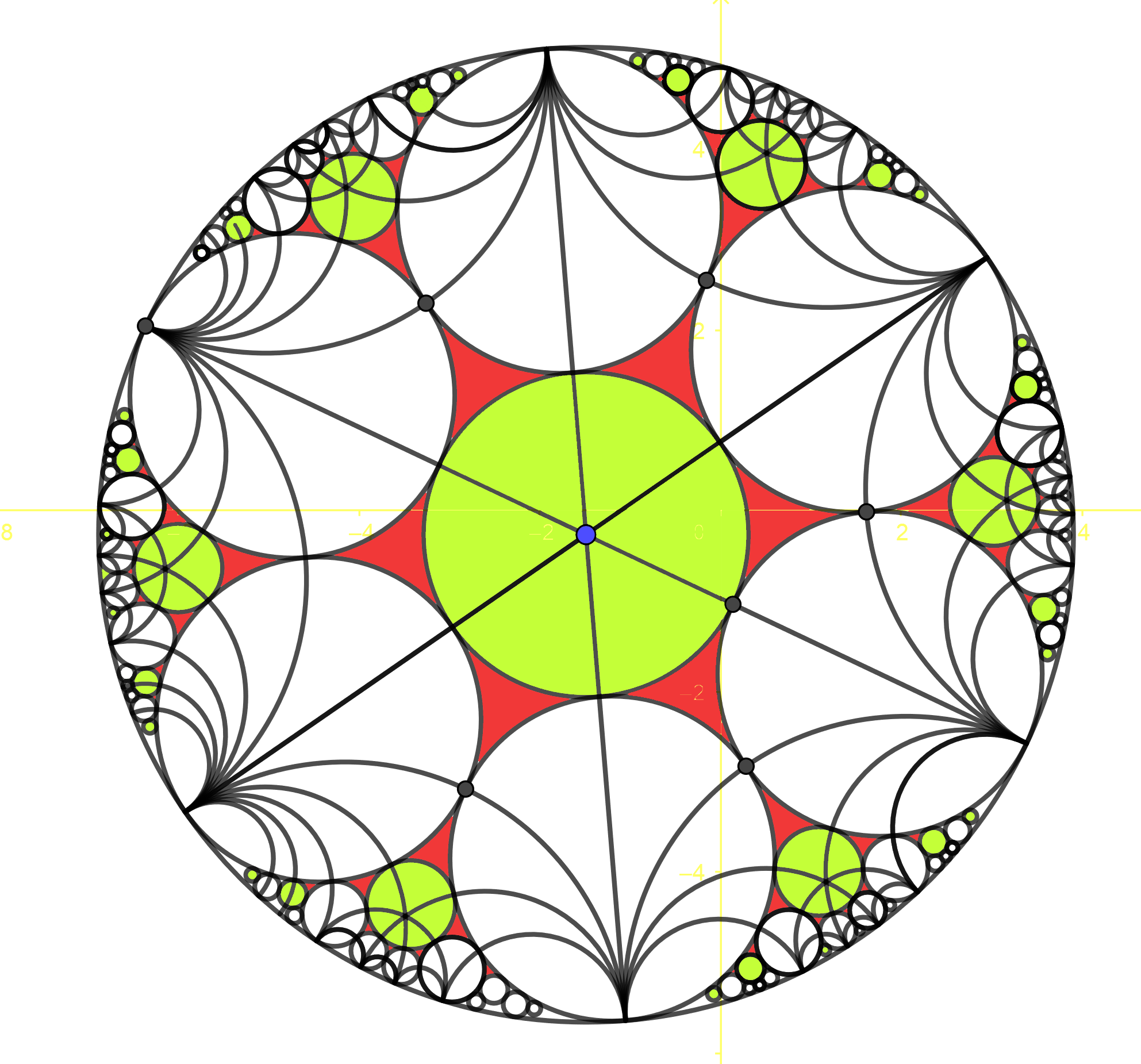}
\caption{\label{((6, inf,inf)(1)}The signature $[3,\infty,\infty]$ case.}
\end{figure}

\begin{figure}
\centering
\includegraphics[scale=0.2]{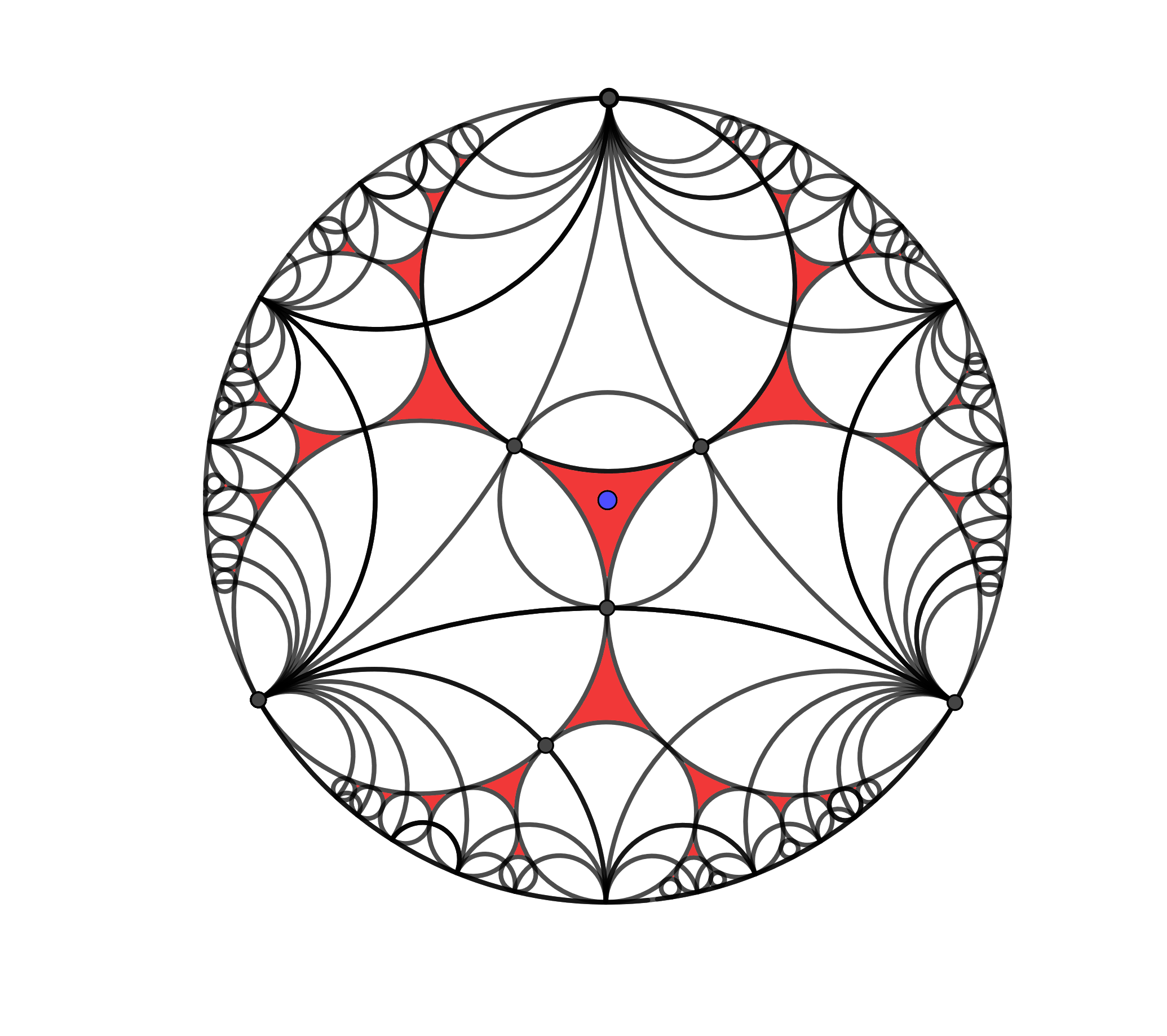}
\caption{\label{Hyperbolic triangle II (inf, inf, inf)}The signature $[\infty, \infty, \infty]$ case.}
\end{figure}

                                                   The retraction of $\mathbb H^2$ onto $\mathbf T$ (or $\Delta^\ast_{fin}$) can be prolongated to 
a retraction $\rho:\mathbb H^2\rightarrow T$  onto a $\Gamma$-invariant tree $T\subseteq\mathbf T$ which we call the \emph{spine} (of $X$). If ~$D$ has a finite edge $e$ this retraction pushes each  triangle $gD$ onto $ge$. Hence ~$T$ is just the finite part of the 1-skeleton of $\Delta$. It 
contains vertices of degrees ~$q_1$ and vertices of degree $q_2$.

                           If $\Gamma$ has signature $[q,\infty,\infty]$ only one vertex $v$ of $D$ is in $\mathbb H^2$ and its opposite edge $e$ is a line. We write $h$ for the hyperbolic segment connecting $v$ with the nearest point $t$ on $e$. Then centered at the endpoints of each $ge$ we have two horodisks $gB_j, gB_k$ which touch each other in the point $gt$. The prolongated retraction sends the whole of $gD$ to the the segment $gh$ which is one half of an edge of $T$ -- the second half is its reflection over the axis $ge$. Hence, as above, the vertex set of $T$ is the set of vertices of $\Delta$ in $\mathbb H^2$, but now the edges of $T$ are  
the geodesic segments connecting vertices with their images under reflection over the opposite sides of their tiles.

If                                                                                                                        $\Gamma$ is of type $[\infty,\infty,\infty]$ then $T=\Delta^\ast_{fin}$ and all vertices are of degree 3.

\subsection{Tessellated sectors and rooted subtrees of the spine.} 

By an (open or closed) \emph{sector} $S\subset\mathbb H^2$ we mean a subset bounded by two rays emanating from $v\in\mathbb H^2$. The two rays are 
the \emph{legs} and their endpoints the \emph{feet} of $S$. Occasionally it is convenient to include \emph{ideal sectors} which have line-legs and 
their tip and feet in $\partial\mathbb H^2$. $S$ is a \emph{tessellated sector} (or a \emph{sector of $\Delta$}) if its tip is a vertex and its legs lie in the 1-skeleton of $\Delta$ and hence are tessellated by edges of $\Delta$.

We call a tessellated sector $S$ \emph{small} if its 
legs are single edges of $\Delta$ and no proper 
subsector with the same tip has the same property. 
One observes that the star of a small sector 
consists of two tiles if the triangle $D$ has 
exactly one vertex at infinity, and of a single tile 
in all other cases. In any case two small sectors 
are $\Gamma$-translates of each other if and only if their tips are in the same $\Gamma$-orbit.

Closely related to sectors of $\Delta$ are the rooted subtrees of the spine $T$. If $v\in{\rm ver}R$ is a vertex of a subtree $R\subset T$ we write
deg$_T(v)$ for the degree of ~$v$ as a vertex in $T$. If deg$_R(v)=1$ then $v$ is a \emph{leaf} of $R$, and if deg$_R(v)={\rm deg}_T(v)$ then $v$ is an \emph{inner vertex} of ~$R$. The subforest spanned by all inner 
vertices of~ $R$ is the \emph{inner part} of $R$, denoted $\mathring R$, and the set of vertices of~$R$ which are not inner is the \emph{boundary} 
of $R$ in $T$, denoted  $\partial R\subset{\rm ver} R$.  A \emph{rooted subtree} $R\subset T$ is a subtree whose boundary in $T$ consists of a single vertex $r$, the \emph{root} of $R$; and if the root is a leaf of $R$ we say that $R$ is a \emph{leaf-rooted} subtree. The inner part of a leaf-rooted subtree is a rooted tree and best described as a \emph{half-tree} 
of $T$, i.e., one of the two connected components obtained by removing the interior of an edge of $T$.

\begin{lemma}\label{rooted subtree} 
Under the assumption that $\Gamma$ is a hyperbolic triangle group with signature $[q_1,q_2,\infty]$ we have a 1-1-correspondence between the small 
tessellated sectors $S$ of $\Delta$ and the leaf-rooted subtrees of the spine $T$. This correspondence associates to $S$ the maximal subtree of $S\cap T$ and to $R$ the minimal subcomplex of $\Delta $ containing the convex closure $\overline R$ of $R$. 

\end{lemma}
\begin{proof} Elementary and left to the reader.
\end{proof}

\subsection{The limit sets of $\Gamma$-polyhedral pieces.} 
Let $P\subset\mathbb H^2$ be an arbitrary $\Gamma$-polyhedral subset.   We consider the boundary of $P$ in the \emph{completed} hyperbolic plane $\mathbb H^2\cup\partial\mathbb H^2$, and denote it by $\partial P=:\partial_{fin}P\cup\partial_{\infty}P$, where $\partial_{fin}P$ consists 
of finitely many edge paths, and $\partial_{\infty}P$ is the \emph{limit-set of $P$} and consists of finitely many segments of $\partial\mathbb H^2$ -- the connected components of $\partial_{\infty}P$ with respect to the $S^1$-topology of the disk model.

We need a slightly stronger connectivity concept: A 
segment $[x,y]\subset\partial_{\infty}P$ is \emph{strongly connected}, if 
none of its inner points is a limit point of the complement $P^c:=\mathbb H^2-P$; and the maximal strongly connected segments are the \emph{strongly connected components} of $\partial_{\infty}P$. We claim that each connected component of $\partial_{\infty}P$ is "tessellated" by (i.e., the union with pairwise disjoint interiors of) its finite set of strongly connected components. Indeed, the only reason why a connected component might not be strongly connected is the possibility that it might contain a tip of the complement $P^c:=\mathbb H^2-P$. As we know that $\Gamma$-polyhedrality of $P$ implies that the closure of $P^c$ is also $\Gamma$-polyhedral (see Subsection \ref{section2.1}), $P^c$ has only finitely many tips. This proves the claim; and at the same time also the

\emph{Observation}: If $P$ is convex $\Gamma$-polyhedral then $\partial_{\infty}P$ is not necessarily connected, but all its connected components are strongly connected.

\begin{lemma}\label{canonical tiles at infinity}
Let  $P$ be a convex $\Gamma$-polyhedral set, $C\subset\partial_{\infty}P$ a connected component of its limit set, and assume that $C$ has two endpoints $x,y\in C$, but at neither of them $\partial P$ continues with a ray in $\Delta$ which is part of a hyperbolic axis. Then there is a canonical finite gallery $G[C]\subset P$, consisting of small sectors and single tiles, which tessellates a neighborhood of $C$ in $P$.
\end{lemma}
\begin{proof} By assumption $\partial P$ continues at $x$ and $y$ with a ray-edge or a line-edge $e_x, e_y$ of $\Delta$.
As $T$ is a tree we have a canonical shortest path in $\Delta^1$ starting 
in $x$, passing through a unique reduced edge path $\omega\subset T$ of length $m\geq 0$ from $e_x\cap T$ to $e_y\cap T$, and ending at $y$. As $T\cap C=\emptyset$ the simple closed path $C\cup\omega$ is the boundary of a topological disk $B$.

For simplicity we first deal with the case when $T$ is contained in the 1-skeleton of $\Delta$, i.e., the triangle $D$ has a unique vertex $v_{\infty}$ in $\partial\mathbb H^2$. In that case $B\cap\mathbb H^2$ is a subcomplex of $\Delta$, and we claim that this is the gallery we are looking for.

Each of the $m$ finite edges of $\omega$ is now the edge of a uniquely defined tile with its opposite vertex in $C$, and we can describe the union 
of these tiles as a finite set of pairwise disjoint \emph{fans} $F_z$ (= 
finite gallery of tiles with the comon vertex $z$), where $z$ runs through a finite subset of $C\cap\Gamma v_{\infty}$. As $C$ is strongly connected each $F_z$ is contained in $P$. The closure of each connected component of the complement  $B-\bigcup_zF_z$ is a sector of $\Delta$ with tip in 
$\omega$ and both legs ray-edges. Such sectors are finite unions of small 
sectors with the same tip. This shows that $G[C]:=B\cap\mathbb H^2$ is the gallery along $C$ as asserted.

In the signature $[q_1,\infty,\infty]$ case the 
vertices of $T$ are in $\Delta$ but not the edges. 
Nevertheless the argument follows the same line: 
Instead of considering tiles with an edge $e\in\omega$ 
we now consider the quadrilaterals $\square$ consisting of 
tile pairs sharing a line-edge of $\Delta$, with the edge $e\in T$ on the 
short diagonal and a vertex of $\square$ in $C$. The signature $[\infty,\infty,\infty]
$ case is even simpler. Here $\omega$ connects the 
centers of the tiles with vertices in the endpoints of 
$C$, and we argue by considering the gallery of tiles 
covering $\omega$.
\end{proof}

We will now consider arbitrary $\Gamma$-polyhedral sets, i.e., finite unions of convex $\Gamma$-polyhedral pieces, $U:=\bigcup_iP_i$, and we can 
assume that the convex polyhedral sets $P_i$ have pairwise disjoint interiors. We find it convenient to express this by saying $U$ is \emph{tessellated} by the \emph{pieces} $P_i$, or that the family $\mathcal{P}:=(P_i)_i$ is a \emph{tessellation} of $U$. But we will avoid calling these pieces "tiles" -- they are infinite unions of the original tiles. 

By a \emph{refinement} of $\mathcal P$ we mean a tessellation $\mathcal P'$ of $U$ with the property that each $P'\in\mathcal P'$ is contained in some $P\in\mathcal P$.

\begin{theorem}\label{Refining tessellations by convex pieces} Each tessellation $\mathcal P$ of $\mathbb H^2$ by a finite set of convex $\Gamma$-polyhedral pieces admits a refinement of the following kind: There is a tessellation $\mathcal P'$ of $\mathbb H^2$ by finitely many single tiles and small sectors, which turns into a refinement of $\mathcal P$ when we take all pieces $P'\in\mathcal P'$ which are small sectors whose middle ray $m$ is a hyperbolic axis, cut them along $m$ in two, and replace $P'$ by the two fragments.  
\end{theorem}

\begin{proof}  The boundary $\partial\mathbb H^2$ is tessellated by the connected components of the 
limit set $\partial_{\infty}P$, with $P$ running 
through $\mathcal P$. The points on $\partial\mathbb 
H^2$ which are endpoints of these connected 
components are finite in numbers, and we consider the 
subset $X$ consisting of those points $x$ which are 
corner points of some $P\in\mathcal P$, in a position 
where a section of $\partial_{\infty}P$ turns into a 
section of $\partial_{fin}P$ which lies on a 
hyperbolic axis $h_x\subset\Delta^1$. In order to 
apply Lemma \ref{canonical tiles at infinity} we have 
to get around the points $x\in X$. Note that $X$ is 
empty unless the signature of $\Gamma$ has two 
finite entries.

Each vertex on $h_x$ is the tip of a unique 
small sector with its middle ray on $h_x$ and with limit point $x$. These 
small sectors are nested and their intersection is the singleton set $\{x\}$. Since the legs of each sector $S$ isolate~ $x$ from the complement of $S$ we find that infinite edge paths of $\Delta^1$ can only reach $x$ through a ray on $h_x$. This shows: 1) $x$ is a limit point of only the two convex polyhedral corner pieces at $x$, $P_x^+ , P_x^-\in\mathcal P$; and 2) if the tip of such a small sector $S$ is sufficiently close to $x$ then $S\subset U=P_x^+\cup P_x^-$. 

Thus, we have a
canonical choice by taking the one small sector $S_x$ with its middle line on the axis $h_x$, its tip in a vertex of $e_x$, and the property that $S_x$ is maximal with respect $S_x\cap(\mathbb H^2-U)=\emptyset$ Now we 
excise the interior of $S_x$ from the corner pieces $P_x^+$ and $P_x^-$, noting that cutting off the corner $x$ along a ray-edge preserves both convexity and $\Gamma$-polyhedrality. Therefore we can replace in $\mathcal 
P$ the pieces $P_x^+,P_x^+$ by the remaining fractions. The result is a tessellation of $\mathbb H^2-S_x$. It avoids the corner $x$ as its boundary turns into the legs of $S_x$ before it meets the remaining finite fragment of the ray $e_x$.

When we have removed all corners $x\in X$ in this way 
we have transformed the tessellation $\mathcal P$ into a tessellation $\mathcal P^\ast$ of $\mathbb 
H^2-\bigcup_{x\in X}S_x$. Then we can apply Lemma  
\ref{canonical tiles at infinity} to all convex $
\Gamma$-polyhedral pieces $P\in\mathcal P^\ast$ to find, along the connected components of their limit sets $\partial_\infty P$, galleries consisting of single tiles and small sectors. 
Together with the small sectors $S_x, x\in X$ they provide a tessellation 
$\mathcal P'$ of a neighborhood of $\mathbb H^2$. As all our tiles have a 
vertex at infinity, such a tessellation must cover all tiles of $\Delta$; 
hence $\mathcal P'$ is the tessellation of $\mathbb H^2$ claimed to exist 
in the assertion of the theorem.
\end{proof}

\begin{remark}\begin{bf}The scaly spider\end{bf}.
Retrospectively, the type of tessellation that we can always achieve is easy to describe:  We pick a finite subtree $T_0$ of the spine $T$ (the spider's body). Attached to the body are 1) all rays which emanate in the 1-skeleton of $\Delta$ out of $T_0$ and are not a single ray-edge (the spider's legs); and 2) all tiles of $\Delta$ that contain a 1-dimensional part of an edge of $T$ (the spider's scales). Then the complement of the scaled body of the spider is the disjoint union of small sectors with the spider's legs on their middle-line.
\end{remark}

Theorem \ref{Refining tessellations by convex pieces} shows that modulo finite permutations each piecewise isometric pemutation of the set $\Omega$ of all tile-centers is given by isometries restricted to small tessellated sectors and halfs of small sectors (split along a hyperbolic axis), which form a finite gallery along the $\partial\mathbb H^2$. Helpful is the simple fact that describing the restriction of an isometry $\varphi$ to 
a small sector ~$S$ requires only the image of the edge emanating at the root together with the information whether  $\varphi$ preserves the orientation. Restricted to the spine, this corresponds to displacing a leaf-rooted subtree $R$ to another position (at a vertex with the same degree in 
$T$). To describe the half-sector moves in the spine is more subtle: They 
are not quasi-autmorphisms of $T$; rather one has to split the rooted subtree along its tree trunk (which lies on the hyperbolic axis) in two, and 
accept a copy of the trunk in both fragments to keep them connected.

\subsection{Piecewise planar tree isometric permutations.} \label{piecewise planar tree isometric permutations}

For simplicity we will from now on restrict the focus to the case when $\Delta^1$ contains no hyperbolic axis, and when the acting group is the orientation preserving subgroup of $\Gamma^\ast\leqslant\Gamma$. Then Theorem \ref{Refining tessellations by convex pieces} asserts that each tessellation $\mathcal P$ of $\mathbb H^2$ by a finite set of convex $\Gamma$-polyhedral pieces admits a refinement consisting of single tiles and small 
sectors. 

Thus, a piecewise $\Gamma^\ast$-isometric permutation $g\in G_{\Gamma^\ast}(\Omega)$ is now given by the restriction of isometries to finitely many tiles and sectors. 
As isometries respect the spine $T$, Lemma 
\ref{rooted subtree} tells us how to translate this 
information to the corresponding tessellation of the 
spine $T$ by a finite set of edges and a subforest $F$ of finitely many leaf-rooted subtrees $R_i$ (the maximal subtrees of $S\cap T$ as $S$ runs through the small sectors of the gallery along $\partial\mathbb H^2$). As 
$T$ and  $R_i$ intersect the boundary of small sectors only in their tips 
the subtrees $R_i$ tessellate $T$ outside a finite subgraph. More precisely: The restriction of $g$ embeds each $R_i$ by a planar-tree isomorphism 
onto the trees of a subforest ~$F'$ such that $T-F$ and $T-F'$ have the same number of vertices. 

In the present framework it is natural to say that $g$ induces on the spine $T$ a \emph{piecewise planar-tree isometric} (ppti-isometric) vertex permutation. Thus we found a homomorphism of $G_{\Gamma^\ast}(\Omega)$ into the group of all a piecewise planar-tree isometric vertex permutation of $T$ which we term ppti$(T)$. 

Conversely: Using Lemma \ref{rooted subtree} one observes that each tessellation of $T$ by finitely many leaf-rooted subtrees and single edges can 
be refined to a tessellation by finitely many single edges and leaf-rooted $R_i$ whose convex closure are (or at least contained in) small sectors 
$S_i$; and each planar-tree isometry on $R_i$ can be represented by a metric-tree isometry on $R_i$ and then extended to an isometry on $S_i$. Hence $G_{\Gamma^\ast} (\Omega)\cong$ppti$(T)$.

By definition piecewise planar-tree isometric vertex permutations respect 
not only the pairs of end points of almost all edges of $T$ but also the cyclic ordering of the stars at almost all vertices of $T$. Hence $g$ can 
also be referred to as a quasi-(planar-tree) automorphism of the spine $T$ -- see (\cite{leh08}, \cite{ls07}, \cite{bmn13}, and \cite{nst15}).

Hence we can also summarize:
\begin{corollary}\label{quasi-automorph} Let $\Gamma^\ast\leqslant\Gamma$ 
be the orientation preserving subgroup of a triangle group with signature 
$[q_1,q_2,\infty]$. If at least one of $q_1 $or $q_2$ is odd or infinite, 
then the group of all piecewise $\Gamma^\ast$-isometric permutations of the tile centers, $G_{\Gamma^\ast}(\Omega)$ coincides with the quasi-(planar-tree) automorphism group {\rm  ppti(ver}(T)) of the spine $T$ of the tessellation $\Delta$.\hfill$\square$
\end{corollary}

\subsection{Connection with Thompson's groups}\label{Connection with Thompson's groups}
Interpreting the statement of Corollary \ref{quasi-automorph} for the triangle group $\Gamma$ with signature $[\infty,\infty,\infty]$ yields the connection with Thompson's groups: in this case the spine $T$ is the the infinite binary (planar) tree $T_2$ that has always been around when Thompson's groups (and their generalizations) have been investigated in terms of generators or as groups of piecewise linear homeomorphisms of the Cantor set on the real line or on $\partial\mathbb H^2$. 

Thus, we infer that for triangle groups with signature $[q_1,q_2,\infty]$, and at least one of $q_1, q_2$ odd or infinite, the quotients $G_{\Gamma^\ast}(\Omega)/{\rm sym}(\Omega)$ are straigtforward generalizations of Thompson's group $V$. Hence roofs in the literature (e.g., \cite{wz14}), showing that these generalized Thompson groups are of type $F_\infty$, also apply in our situation. 

However, to extend the results on $G_{\Gamma^\ast}
(\Omega)/{\rm sym}(\Omega)$ for more 
general $\Gamma$, or on the groups 
$G_{\Gamma^\ast}(\Omega)$ themselves, it seems more rewarding to 
skip the detour to the spine $T$ but rather try 
to modify tools that were successful for Thompson's 
groups: Instead of tree-parameters and the 
partially ordered set of rooted subtrees of $T$ one 
might be able to use hyperbolic plane parameters 
and the partially ordered set of small sectors (or 
halfs of small sectors) of the tessellated 
hyperbolic plane to find some understanding of $G_
\Gamma(\Omega)$ when the 1-skeleton of $\Delta$ contains a hyperbolic axis -- with a bit of luck even in the case when $\Gamma$ is a cocompact triangle group.

Instead of trying to do this one-handedly it would be interesting to know 
how much of that can already be covered (or promoted) by the cloning systems of \cite{wz14} or the abstraction of \cite{fh20}.

\chapter{The Euclidean case I: The structure of pei(S)}

\section{Orthohedral sets}

\subsection{Integral orthants in \texorpdfstring{$\mathbf{\mathbb Z^N}$}{Z^N}.}\label{Integral orthants}~ In the standard $N$-dimensional Euclidean integral lattice 
$\mathbb Z^N $ we consider affine-orthogonal transformations
$$\begin{aligned}
\mbox{$\tau_{a,A}:  \mathbb Z^N \to
\mathbb Z^N$}\\
\tau_{a,A}(x) =  a+Ax,\end{aligned}$$
where $A\in O(N,\mathbb Z)$ is an
integral orthogonal matrix and $a \in \mathbb Z^N $.
Inside $\mathbb Z^N$ we have the \emph{standard orthant of rank} $N, \mathbb N^N
\subseteq\mathbb Z^N$, and all images of its $k$-dimensional faces, $0\le 
k\le
N$, under affine-orthogonal transformations.
More precisely: the subsets $L = \tau_{a,A} \langle Y \rangle \subseteq\mathbb Z^N $, where $\langle Y \rangle $ stands for the monoid generated 
by the $k$-element set $Y$ of canonical basis vectors.
We call $L$ an \emph{integral orthant} (of rank $k$, and based at $a\in L$) of
$\mathbb Z^N $ or just a \emph{rank-$k$} orthant. 

We write $\Omega^k$ for the set of all rank-$k$ orthants of $\mathbb Z^N$ 
and $\Omega^\ast$ for the union $\bigcup_k\Omega^k$.
$\Omega^\ast$ is partially ordered by inclusion, with $\Omega^0=\mathbb 
Z^N$. The subset of all orthants based at the origin 0
will be denoted by $\Omega_0^\ast \subseteq \Omega^\ast$; it retracts the 
order preserving projection $\tau:\Omega^\ast
\rightarrow \Omega_0^\ast$ which associates to each orthant $L\in \Omega^\ast$ based at $a\in\mathbb Z^N $ its unique
parallel translate $\tau(L)= -a + L\in\Omega_0^\ast$. $\tau(L)$ is characterized by its canonical basis $Y = \{y \in
\pm X~|~ a+\mathbb Ny \subseteq L\}$ which indicates the directions of $L$; hence we call $\tau(L)$ the indicator of $L$.
Note that $Y$ is given by the function $f$: $X$ $ \rightarrow \{0, 1, -1\}$ with  $f(x) = \epsilon\in \{1, -1\}$ if  
$\epsilon x\in Y$, and $f(x) = 0$ if  $\{x, -x\}\cap Y = \emptyset$; hence $|\Omega_0^\ast| = 3^N$. 

We call a subset $S\subseteq\mathbb Z^N$ \emph{orthohedral} if it is the union of a finite set of orthants - without losing
generality we can assume that the union is disjoint. The \emph{rank} of $S$, denoted ${\rm rk}S$, is the maximum rank of
an orthant contained in $S$. If $S$ is isometric to $\mathbb N^k \times \{1, 2, \ldots , h\}$, we call it a \emph{stack of orthants of height} $h$. The terminology agrees with the \emph{height} $h(S)$
of an arbitrary orthohedral set $S\subseteq\mathbb Z^N$, defined as the number of orthants of maximal rank, ${\rm rk}S$, which
participate in a pairwise disjoint finite decomposition of $S=L_1\cup L_2\cup \ldots \cup L_m$ -- see Section \ref{the groups} in the
introduction.

\begin{lemma}\label{closed under complements}
Orthohedrality of subsets $S\subseteq\mathbb Z^N$ is closed under the set 
theoretic operations of taking intersections, 
unions, and complements.
\end{lemma}
\begin{proof} The main observation here is that the intersection of a finite set of half-spaces of $\mathbb Z^N$ (each defined by an upper or lower bound on one coordinate) is a disjoint union of finitely many orthants. 
We prove this by induction. If two of these half-spaces, $H, H'$, are bounded by parallel hyperplanes then either one of them is redundant or their intersection $ H\cap H'$ is a (possibly empty) finite union of lower dimensional subspaces. In both cases we are reduced to the intersection of fewer half-spaces. If no pair of the half spaces have parallel boundary there are only $k\leq N$ of them, and their intersection is isometric to $\mathbb Z^{N-k} \oplus\mathbb N^k $ and hence is a finite union of $2^{N-k}$ rank-N orthants. The assertion of the lemma follows now by set-theoretic tautologies.
\end{proof}
\begin{remark}
As a consequence we note that the orthohedral subsets of $\mathbb Z^N $ are precisely the $\mathbb Z^N $-polyhedral subsets of the lattice $\mathbb Z^N $ as defined in Section~ \ref{Integral orthants}.
\end{remark}

We write $\Omega^k(S)=\{L\in\Omega^k | L\subseteq S\}$ for the set of all rank-$k$ orthants of $S$, $\Omega^\ast(S)$ for the disjoint union over 
k, and $\Omega_0^\ast(S)$ for the
set of all orthants of S based at the origin 0. We consider the restriction of the indicator map $\tau:\Omega^\ast(S)
\rightarrow \Omega_0^\ast$. We write $S_\tau\subseteq \mathbb Z^N$ for the union of all orthants in $\tau(\Omega^\ast(S))$
and call this the indicator image of S. Note that $\tau(\Omega^\ast(S))=\Omega_0^\ast(S_\tau)$, and we can view the
indicator map as a rank preserving surjection \mbox{$\tau:\Omega^\ast(S) \twoheadrightarrow \Omega_0^\ast(S_\tau)$}.

\subsection{Germs of orthants.}\label{Germs of orthants}
Two orthants $L, L'$  in $\Omega^\ast$ are said to be
\emph{commensurable} if ${\rm rk}L = {\rm rk}(L\cap L') = {\rm rk}L'.$ We write $\gamma (L)$ for the commensurability class of $L$ and call 
it the \emph{germ} of $L.$ The union of all members of $\gamma (L)$ is a coset of a subgroup of $\mathbb Z^N$; we denote it by $\langle L\rangle$ $\subseteq\mathbb Z^N$ and call it the tangent coset of S at $\gamma$. The germs inherit from their representing orthants $L$ the \emph{rank}, 
relations like \emph{parallelism and orthogonality}, and also a \emph{partial ordering} defined as follows: given two
germs $\gamma, \gamma'$ we put $\gamma\leq\gamma'$ if they can be represented by orthants $L, L' \in\Omega^\ast$ with
$L\subseteq L'$. Note that if $L$ and $L'$ are arbitrary orthants representing $\gamma$ and $\gamma'$, respectively, then
$\gamma\leq\gamma'$ if and only if (1) $L'$ contains an orthant parallel to $L$ and (2) $L\subseteq\langle L'\rangle$.

We write $\Gamma^\ast(S)= \bigcup_k\Gamma^k(S)$ for the set of all germs of orthants in $S$ and $\Gamma_0^\ast(S)$ for
the set of all germs represented by an orthant of $S$ based at the origin 
0. $\Gamma^\ast(\mathbb Z^N)$ and
$\Gamma_0^\ast(\mathbb Z^N)$ are abbreviated by  $\Gamma^\ast$ and  $\Gamma_0^\ast$, respectively. As $\Gamma_0^\ast$
and $\Omega_0^\ast$ are canonically bijective, we will identify them when 
this is convenient. Note that $\Gamma^\ast(S)$
is a convex subset of $\Gamma^\ast$ in the sense that if $\gamma\in\Gamma^\ast(S)$ then $\{\gamma'\in\Gamma^\ast~|~
\gamma'\leq\gamma\}\subseteq\Gamma^\ast(S)$. We can interpret the indicator map as an order and rank preserving surjection
\mbox{$\tau:\Gamma^\ast(S) \rightarrow \Gamma_0^\ast$} with $\tau(\Gamma^\ast(S))= \Gamma_0^\ast(S_\tau)$. By
max$\Gamma^\ast(S)$ we mean the set of all maximal germs of $S$.

\begin{exercise}Observe that $\tau({\rm max}\Gamma^\ast(S)) \supseteq {\rm max}\Gamma_0^\ast(S_\tau)$, but this is not, in general, an equality.
\end{exercise}

\begin{lemma}\label{lemma3.2} \emph{max}$\Gamma^\ast(S)$ is finite for each orthohedral set $S$. The set of all germs of
rank $n = {\emph rk}S$ is a subset of \emph{max}$\Gamma^\ast(S)$, whose 
cardinality coincides with the height $h(S)$.
Hence $h(S)$ is independent of the particular decomposition of $S$.
\end{lemma}
\begin{proof} Let $S = \bigcup_j L_j$ be an arbitrary decomposition of $S$ as a finite pairwise disjoint union of
orthants $L_j$. Each orthant $L\subseteq$ is the disjoint union of the orthants $M_j = L\cap L_j$, and
exactly one of them is commensurable to $L$. Hence $\gamma(L)  =  \gamma(M_j)\subseteq\gamma(L_j)$. This shows that each
germ $\gamma\in\Gamma^\ast(S)$ is smaller than or equal to one of the $\gamma(L_j)$. In particular, max$\Gamma^\ast(S)$ is
contained $\{\gamma(L_j)~|~j\}$ and hence finite. The orthants $L_j$  of rank $n$ form a complete set of representatives
of all orthants of rank $n$.
\end{proof}

\begin{remark}
We leave it to the reader to deduce that $h(S\cup S') = h(S)+h(S')$, if 
$S$ and $S'$ are orthohedral sets with ${{\rm rk}}(S) = {{\rm rk}}(S') > {{\rm rk}}(S\cap S')$. 
\end{remark}

\subsection{Piecewise isometric maps.}
Let $S\subseteq\mathbb Z^N$ be an orthohedral subset. We call a map $f: S\to\mathbb Z^N$ \emph{piecewise-Euclidean-isometric} (abbreviated as \emph{pei-map}), if $S$ is covered by a finite set $\Lambda$ of pairwise disjoint orthants, with the property that the restriction of $f$ to each orthant $L\in\Lambda$ is an isometric embedding $f_{\mid{L}}:L\rightarrow S$. 
The \emph{support} of $f\in G(S)$,~ ${\rm supp}(f)=\{a\in S\mid ag~\neq~a\}$, is orthohedral, and we refer to its rank also as \emph{the rank of 
$f$}, denoted $rk(f)$.
\par
Correspondingly, we call $f$  \emph{piecewise Euclidean-translation map} (abbreviated as \emph{pet-map}), if  $S$ is a
finite disjoint union of orthants with the property, that the restriction 
of $f$ to each of them is given by the
restriction of a translation.
\par
If a bijection $f: S\to S'$ is a pei-map (respectively, a pet-map), so is 
$f^{-1}$ and we say that $S$ and $S'$ are
\emph{pei-isomorphic} (respectively, a \emph{pet-isomorphic}). 

By the argument used in the proof of Lemma \ref{lemma3.2} above one shows 
that if $f$ is a pei-map, then each orthant
$L\subseteq S$ contains a commensurable suborthant on which $f$ restricts 
to an isometric embedding. In fact, we leave it
to the reader to observe the following:
\begin{lemma}\label{lemma3.3} Let \mbox{$f:S\to \mathbb Z^N$} be an injective map on an orthohedral set
$S\subseteq\mathbb Z^N$. Then $f$ is a pei(resp. pet)-injection if and only if every orthant $L$ of $S$
contains a commensurable suborthant $L'\subseteq L$ on which $f$ is given 
by an isometry (respectively, a translation) onto
$f(L')\subseteq\mathbb Z^N$.
\end{lemma}

It follows that every injective pei-map \mbox{$f: S \to \mathbb Z^N$} induces a rank preserving injection \mbox{$f_\ast:
\Gamma^\ast(S) \to \Gamma^\ast(f(S))$}. $f_\ast$ does not preserve the ordering -- not even if $f$ is a pet-map. But
since it is rank-preserving, it does induce a bijection between the germs 
of maximal rank of $\Gamma^\ast(S)$ and
$\Gamma^\ast(f(S))$, whence $h(f(S)) = h(S)$.
The following observations can be left as an exercise:
\begin{lemma}\label{lemma1.4} If \mbox{$f: S \to \mathbb Z^N$} is a pet-map, then $f_\ast(\gamma)$ is parallel to $\gamma$
for each $\gamma\in \Gamma^\ast(S)$. Hence $\tau(f_\ast(\gamma)) = \tau(\gamma)$, and  $S_\tau = f(S)_\tau$. In other
words we have the commutative diagram
$$
\begin{xy}
  \xymatrix{
      \Gamma^\ast(S)  \ar[r]^{f_\ast} \ar[d]_\tau &   \Gamma^\ast(f(S)) \ar[d]_\tau \\
      \Gamma_0^\ast(S) \ar@{=}[r] & \Gamma_0^\ast(f(S)_\tau)   
  }
\end{xy}
$$
 \end{lemma}

\subsection{Normal forms.}
Consider the disjoint union of orthants
$$S =  L_1\cup L_2\cup \ldots \cup L_m$$
in $\mathbb Z^N$. Assuming that
${\rm rk}S<N$ we have enough space to parallel translate each $L_i$ to an 
orthant $L'_i$ in such a way that the $L'_{i}$ are
still pairwise disjoint, but that each (oriented) parallelism class of the orthants $L'_i$ is assembled to a stack.
This describes a pet-bijection $S\to S' = \bigcup_jS_j$ , where the $S_j$
stand for pairwise disjoint and non-parallel stacks of orthants. We can go one step further by observing that when the maximal orthants of a stack 
$S_i$ are parallel to suborthants of the stack $S_j$, then there is a pet-bijection   $S_i\cup S_j \to S_j$ which feeds $S_i$ into $S_j$. Hence we 
can delete all stacks $S_i$ of orthants that are parallel to a suborthant 
of some other $S_j$ and find

\begin{proposition}\label{proposition3.5}(pet-normal form
Each orthohedral set $S$ is pet-isomor\-phic to a disjoint union of stacks of
orthants $S' =\bigcup_ j S_j$, with the property that no maximal orthant of any
$S_j$ is parallel to a suborthant in some  $S_k$, if $k\not= j.$
\end{proposition}
\begin{corollary}\label{corollary3.6}(pei-normal form)
Each orthohedral set $S$ is pei-isomorphic to a stack of orthants.
\end{corollary}

As observed in Section \ref{Integral orthants}, rank ${\rm rk}S$ and height $h(S)$ are pei-invariant; hence they can be read off from the pei-normal
form; and the pair $({\rm rk}S, h(S))$ characterizes $S$ up to pei-isomorphism. For the corresponding pet-result we consider the
height function
\begin{equation}
h_S: \Gamma_0^\ast \longrightarrow \mathbb N\cup\{0\},
\end{equation}\label{eq2}
which assigns to each 0-based orthant  $L\in \Omega_0^\ast =\Gamma_0^\ast$ the number of maximal germs
$\gamma\in{\rm max}\Gamma^\ast(S)$ with  $\tau(\gamma) = L$, which is finite by Lemma \ref{lemma3.2}. The support
${\rm supp}(h_S) \subseteq \Gamma_0^\ast$ is the set of all 0-based orthants $L$ with $h_S(L) > 0$. From the Exercise in Section
3.2 we infer that max$\Gamma_0^\ast(S_\tau)\subseteq {\rm supp}(h_S)$, and that this is not, in general an equality. One
observes easily that the equality 
\begin{equation}\label{eq3}
\tau(\max\Gamma^{\ast}(S)) = \max\Gamma_0^\ast(S_\tau) \text{ or equivalently: } \max\Gamma_0^\ast(S_\tau) = {\rm supp}(h_S)
\end{equation}
is a necessary condition for $S$ to be in pet-normal form. Thus we call $S$ \emph{quasi-normal} if the equation (\ref{eq3})
holds. Of course, a quasi-normal orthohedral set is not necessarily in pet-normal form. But as quasi-normality implies
that $\tau$ restricts to a surjection \mbox{$\tau : {\rm max}\Gamma^\ast(S) \twoheadrightarrow
{\rm max}\Gamma_0^\ast(S_\tau)$}, max$\Gamma^\ast(S)$ is the pairwise disjoint union of the fibers $f^{-1}(\gamma)$, which consist of  $h_S(\gamma)$  germs
parallel to $\gamma$. This can be viewed as a weak germ-version of the pet-normal form.

\begin{lemma}\label{lemma1.7}
 If \mbox{$f:S \to \mathbb Z^N$} is a pet-injection of a quasi-normal orthohedral set $S\subseteq\mathbb Z^N$,
 then $f_\ast(\emph{max}\Gamma^\ast(S)) \subseteq \emph{max}\Gamma^\ast(f(S))$.
\end{lemma}
\begin{proof}
By Lemma \ref{lemma3.3} $f$ induces a rank preserving bijection
$$f_\ast: \Gamma^\ast(S) \to
\Gamma^\ast(f(S)),$$
and by Lemma \ref{lemma1.4} $f(S)_\tau = S_\tau.$ Let $\gamma\in{\rm max}\Gamma^\ast(S)$. Then
we know that $\tau(\gamma)$ is maximal in $\Gamma_0^\ast(S_\tau)$. Since $f$ is a pet map, we also know that
$\tau(f_\ast(\gamma)) = \tau(\gamma)$; hence $\tau(f_\ast(\gamma))$ is maximal in  $\Gamma_0^\ast(S_\tau) =
\Gamma_0^\ast(f(S)_\tau)$. We claim that $f_\ast(\gamma)$ is maximal in $\Gamma^\ast(f(S)_\tau)$. Indeed, if $f_\ast(\gamma)$ is
not in max$\Gamma^\ast(f(S)_\tau)$, then $\tau(f_\ast(\gamma))$ cannot be 
maximal in $\Gamma_0^\ast(f(S)_\tau)$. This
shows that  $f_\ast({\rm max}\Gamma^\ast(S)) \subseteq {\rm max}\Gamma^\ast(f(S))$, as asserted.
\end{proof}

\begin{corollary}\label{corollary3.8}
If \mbox{$f:S \to S'$} is a pet-isomorphism between quasi-normal orthohedral sets, then $f_\ast(\emph{max}\Gamma^\ast(S))
= \emph{max}\Gamma^\ast(S')$ and $h_S = h_{S'}$.
\end{corollary}

This shows, in particular, that the stack heights in a pet-normal form are uniquely determined and characterize S up to
pet-isomorphism.

\section{Permutation groups supported on orthohedral sets}\label{structure of pei(S)}
\subsection{pei- and pet-permutation groups.}\label{subsection4.1}
Let $G = {\rm pei} (\mathbb Z^N)$ denote the group of all pei-permutations of $\mathbb Z^N$. From now on it will be convenient to follow the permutation-group tradition to have the permutation group G act on its set $\Omega = \mathbb Z^N $ from the \emph{right} and interpret the product $ gf $ of elements $ g,f\in G$ as $g$ \emph{followed by} $f$.
\par
The \emph{support} of an element $g\in G$ is defined as the union of all orthants on which $g$ restricts to a non-trivial isometry: $${\rm supp}(g): = \bigcup \{ L\in\Omega^\star\mid g_{\mid L} \text{ is an isometric embedding} \neq id_L\}.$$
To say that a given subset $S\subseteq\mathbb{Z}^N$ \emph{supports g} merely means that that ${\rm supp}(g)\subseteq S$. 
\begin{exercise}
Prove that ${\rm supp}(g)$ is the minimal orthohedreal subset containing the set $\{a\in\mathbb{Z}^N\mid ag\neq a\}$.
\end{exercise}
As we know that the support of an element $g\in G$ is orthohedral it makes sense to put $rk(g):= rk({\rm supp}(g))$, and call this the \emph{rank} of the element $g$.\par
The \emph{support} of a subgroup $H\leqslant G$ is the union of the supports of its elements.
\par
The product of a finite number of elements $g_i\in G$ is \emph{disjoint} if ${\rm supp}(g_i)\cap {\rm supp}(g_j)=\emptyset$ for all i$\neq$j.

If 
$S\subseteq\mathbb Z^N$ is orthohedral we write $G(S):= \{g\in G
\mid {\rm supp}(g)\subseteq S\}$ for the subgroup of $G$ supported on $S$. 
As we know, by Lemma \ref{closed under complements}, that the complement of $S$ is also 
orthohedral each pei-bijection of $S$ extends to an element of $G$; 
hence the subgroup $G(S)\leqslant G$ is also the pei-automorphism 
group of $S$. We write also ${\rm pei}(S)$ for $G(S)$ when this is 
convenient.
\par
As an immediate consequence of Corollary \ref{corollary3.6} we have 
\begin{corollary}\label{corollary4.1}  If $S\subseteq\mathbb Z^N$ is an
orthohedral subset, then ${pei}(S)$ is isomorphic to ${pei}(S')$, where $S'$ is a stack of orthants of rank
$\emph{rk}S$ and height $h(S)$.
\end{corollary}
The set of all pet-permutations on the orthohedral set $S$ is the pet-subgroup $pet(S) \leqslant G(S)$. As an immediate consequence of  Proposition~ \ref{proposition3.5} and Corollary  \ref{corollary3.8} we find 
\begin{corollary} \label{corollary4.2} If $S\subseteq\mathbb Z^N$ is an
orthohedral subset and $S' = \bigcup_j S_j$  its \emph{pet}-normal form, then $\emph{pet(S)}$ is isomorphic to $\emph{pet}(S')$.
\end{corollary}
\begin{definition}[The \emph{rank groups~$G_k$}]
As conjugation in $G = pei(\mathbb{Z^N})$ preserves the rank of the elements, putting $G_{-1}:=1$, and for $k\geq0$
$$G_k := \{g\in G~\mid~ rk(g)\leq k~\}, $$
yields the normal series $1=G_{-1}\leqslant G_0 
\leqslant ... \leqslant G_k \leqslant ...\leqslant G_N = G$ 
which plays the key role to understanding the structure of G. 
For each orthohedral subset $S\subseteq\mathbb Z^N, G_k(S):= G_k\cap G(S)$ yields the corresponding normal sequence for $G(S)$.
\end{definition}
Note that by the pei-normal form we have $$\emph{pei}(S)~=~G_{rkS}(S)~\cong~ G_{rkS}(\bigcup_{1\leq i\leq h(S)}\mathbb{N}^{rkS}),$$ for 
every orthohedral set $S\subseteq \mathbb{Z}^N.$

In this section we are aiming for insight into the group theoretic structure of pei$S$, are now in a position to outline its main results in a nutshell: 
\begin{theorem}\label{theorem4.4} \label{preview} If $G=pei(S)$, with $S$ an orthohedral set of rank $rkS=n$ then the following holds:
\begin{enumerate}[i)]
\item $G_k/G_{k-1}$ is an extension of a free-Abelian normal subgroup (of 
infinite rank when $1\leq k<rk(S)$) with a locally finite factor group. In particular $G$ is elementary amenable.  
\item The rank-groups $G_k$ are characteristic in $G$. Each normal subgroup $N\leqslant G$ is contained in some $G_k$ and intersects $G_{k-1}$ in a subgroup of index $\leq 2$. Consequently every Abelian-by-locally-finite section of $G$ is a section of $G_k/{\rm alt}G_{k-1}$ for some $k\leq n$, and n+1 is the minimum length of normal series of $G$ with Abelian-by-locally-finite factors. 
\item $G$ satisfies the maximal condition for 
normal subgroups.
\end{enumerate}
\end{theorem}
For the proofs see Theorem \ref{classifies G_k(S)/G_k-1(S)}, Theorem \ref{Bottleneck Theorem}, and Corollary \ref{corollary4.19}.

\subsection{The action of G on the germs.}
By Lemma \ref{lemma3.3} we know that given a pei-permutation $g\in G$, each germ $\gamma \in \Gamma^\ast$ represented by an orthant $L$ contains a 
suborthant $L'$ commensurable with $L$ on which g restricts to an isometric embedding of $L'$ into $Lg$. Hence putting $\gamma g := \gamma(L'g)$ 
 $\in \Gamma^\ast$ well defines a rank preserving action of $G$ on $\Gamma^\ast$. For each orthohedral subset $S\subseteq\mathbb Z^N$ this action restricts to an action of $G(S)$ on $\Gamma^\ast(S)$.

\begin{lemma}\label{transitive action} $G_k$ acts on the set $\Gamma^k$ of rank-k germs by finite permutations; and for each orthohedral set S, the restricted action of $G_k(S)$ on $\Gamma^k(S)$ is highly transitve in the sense that each bijection $f: F \rightarrow F' $ between finite subsets of $\Gamma^k(S)$ is induced by the action of some $g\in G_k(S)$.
\end{lemma}
\begin{proof}
An element $g\in G_k$ can only dislocate the rank-k germs in
$$\Gamma^k({\rm supp}(g))\subseteq\Gamma^k,$$
and these are finite in number.

We claim that the bijection $f$ 
extends to a permutation $\pi$ of $F\cup F'$. To 
see this consider the graph $\mathfrak{G}$ with 
vertex set ver$\mathfrak{G}=F\cup F'$, and the 
oriented edge set 
edg$\mathfrak{G} = \{(a,f(a))\mid a\in F\}$. 
Then one observes that $\mathfrak{G}$ can be completed to a permutation graph since $\rvert F-F\cup F'\lvert=\rvert F'-F\cup F'\lvert$.

\par
Now we represent the elements of $F\cup F'$ by a set of pairwise disjoint 
orthants $\{L_\gamma\mid\gamma\in F\cup F'\}$, and we lift the graph $\mathfrak{G}$ as follows: we choose for each $\gamma\in F\cup F'$ an isometry $\tilde{\pi}_\gamma: L_\gamma\rightarrow L_{f(\gamma)}$, but make sure that along each simple closed path the product of the chosen isometries is the identity. The union of thes isometries is an element of $G_k(S)$, and induces the map $f$.
\end{proof}

\subsection{Stabilizers of rank-k germs} \label{Stabilizers of rank-k germs} Next we consider the stabilizer
\begin{equation}\label{eq4}
 C(\gamma): = \{g\in G_k\mid \gamma g = \gamma \}, ~~~\gamma\in\Gamma^ k.
\end{equation}
We attach to $\gamma$ the union $\langle \gamma \rangle \subseteq \mathbb 
Z^{N}$ of all orthants of $\mathbb Z^{N}$ representing $\gamma$. $\langle 
\gamma \rangle$
is a coset of a coordinate supgroup of $\mathbb Z^{N}$ and isometric to $\mathbb Z^k$; we call it the \emph{tangent coset of S at}
 $\gamma$. The stabilizer $C(\gamma)$ acts canonically on $\langle \gamma 
\rangle$: Indeed, given $g\in C(\gamma)$, we find an orthant $L'$ representing $ \gamma$ with the property that $g$ maps $L'$ isometrically to $L'g$ ~which is commensurable to $L$, and that isometry extends canonically to an isometry of $\langle \gamma \rangle$ onto itself. This yelds a homomorphism
 \begin{equation}\label{eq5}
\varphi_\gamma: ~ C(\gamma )\to
\emph{Isom}\langle \gamma \rangle. 
\end{equation}
$\langle \gamma \rangle$ carries additional $C(\gamma)$-invariant structure: As commensurable orthants are canonically linked by a unique parallel 
translation we can endow the canonical basis of the orthants representing 
$\gamma$ with compatible orderings. Hence $\langle \gamma \rangle$ comes endowed with a canonical $C(\gamma)$-invariant set $X(\gamma)$ of k pairwise orthogonal coordinate directions. $C(\gamma)$ acts k-transitively on $X(\gamma)$; and the homomorphism \eqref{eq2} factors, modulo translations, through an epimorhism onto the symmetric permutation group on $X(\gamma)$,
\begin{equation}\label{eq6}     
\overline{\varphi}_\gamma: ~ C(\gamma )\twoheadrightarrow {\rm sym}_k(X(\gamma)). 
\end{equation}
By an \emph{ordered germ} we mean a germ $\gamma $ together with an ordering on the canonical basis-directions of $\langle \gamma \rangle$; and we 
write $C^{ord}(\gamma)$ for the stabilizer of the ordered germ $\gamma$. By choosing an ordering of the canonical monoid basis of $\mathbb{Z^N}$ we can impose germ orderings simultaneously on all germs of $\mathbb{Z^N}$; these orderings are preserved by all maps induced by inclusions and parallel translations, and in this situation we say that the germs are endowed with \emph{compatible orderings}.
It is easy to observe  
\begin{lemma}\label{germ stabilizer}
\begin{itemize}
\item[i)] $G_{k-1} \leqslant C^{ord}(\gamma) \leqslant C(\gamma)$,
\item[ii)] $\ker(\overline{\varphi}_\gamma) = C^{ord}(\gamma)$,
\item[iii)] $\ker(\varphi_{\gamma}) = \{g\in G\mid g$ fixes an orthant representing $\gamma$ pointwise\}.
\end{itemize}
\end{lemma}
\begin{exercise}
The following conditions are equivalent for an orthhedral set S:\\
1) S has a germ $\gamma$ with ${\rm coker} (\varphi_{\gamma})$ non-zero,\\
2) S has a germ $\gamma$ with ${\rm coker} (\varphi_{\gamma}) = \mathbb{Z}$,\\ 
3) there is a number $k\in\mathbb{N}$ with $\Gamma^k(S)$ a singleton set,\\
4) S is pei-isomorphic to $\mathbb{N}^k$ with $k \geq 1$.
\end{exercise}
{
By Lemma \ref{transitive action} the stabilizers of all rank-k germs are conjugates of one another; hence their intersection $C(\Gamma^k(S))$ := 
$~\bigcap_{\gamma \in \Gamma^k(S)}C(\gamma)$ is a normal subgroup of G(S), and so is $C^{ord}(\Gamma^k(S))$ := $~\bigcap_{\gamma \in \Gamma^k(S)}C^{ord}(\gamma)$. We claim that this yields the following refinement of the normal series based on ranks:
\begin{equation}\label{normal series} 
G_{k-1}(S)~ \leq~C^{ord}(\Gamma^k(S)) ~\leq ~C(\Gamma^k(S))~ \leq~ G_k(S),
\end{equation}
for all $k\leq rkS$.\\ 
Indeed, the first two inclusions immediate from the first part of Lemma~ \ref{germ stabilizer}, while the remaining inclusion is the following observation: Given $g\in G(S)$, any rank-(k+1) germ $\gamma$ is represented by an orthant $L$ on which $g$ restricts to an isometric embedding $ f = 
g\mid_L: L\rightarrow S$, and $f$ maps each rank-k face $F$ of $L$ to a face $Fg$ of $Lg$. Assuming $g\in C(\Gamma^k(S))$ implies that each $Fg$ is commensurable to $F$ hence $f$ parallel shifts each $F$ to $Fg$, and these shifts can be interpreted in the tangent coset $\langle \gamma \rangle$. Since $f$ is an isometry it follows that $f$ can only be the identity 
of $L$, whence $rk(g) \leq k$.

\subsection{Dynamics of the action of $G_k$ on \texorpdfstring{$\mathbf{\Gamma^{k-1}}$}{TEXT}}
If~$S$ is an orthohedral set of rank n then $\Gamma^n(S)$ is finite. $
\Gamma^{n-1}(S)$ is infinite and comes with a rank-1 orthohedral 
structure: Each rank-n orthant $L\subset S$ contributes $n-1$ maximal 
parallelism classes of rank-$(n-1)$ germs represented by parallel 
cross-sections of $L$; we call these the rays of  $\Gamma^{n-1}(S)$. As $S$ is orthohedral we find that the union of $(n-1)h(S)$ such rays is cofinite in $\Gamma^{n-1}(S)$. Thus $\Gamma^{n-1}(S)$ has a 1-dimensional piecewise isometric structure and one observes readily the induced action of $G(S)$ is piecewise isometric. 

Let 
$g\in C(\gamma)$ for some $\gamma\in\Gamma^k$. Let $L$ be a rank-k orthant representing $\gamma$, with the 
property that $g$ restricted to $L$ is isometric. Then $Lg$ is a rank-k orthant commensurable to $L$; we put $fl_\gamma(g):= h(L-Lg)-h(Lg-L)
$ and note that this is an integer which does not depend on the particular choice of $L$. Thus, $fl_\gamma:C(\gamma)\rightarrow\mathbb{Z}$ is a well defined homomorphism for all $\gamma\in\Gamma^k$. It measures the balance of trading rank-$(k-1)$ germs towards and away from $\Gamma$, and we call it the \emph{corank-1 germ flow of $g\in C(\gamma)$ at $\gamma\in\Gamma^k$}. As the action of $g\mid_L$ can be monitored in the tangent coset 
$\langle\gamma\rangle$ via the homomorphism \eqref{eq5} we have $fl_\gamma(g)= fl_\gamma(\varphi_\gamma(g))$.
\par
If $fl_\gamma(g)$ is positive $\gamma$ is a \emph{sink} of $g$; if $fl_\gamma(g)$ is negative it is a \emph{source} of $g$. Clearly, $fl_\gamma(g)$ vanishes when $\gamma\notin\Gamma^k({\rm supp}(g))$. This shows that $g\in G$ has -- if any -- only finitely many sources and sinks in $\Gamma^{rk(g)}$. Hence collecting the flow maps $fl_\gamma$ as $\gamma$ runs through $\Gamma^k$ yields the \emph{global flow homomorphism}
\begin{equation}\label{flow map}  
fl:C(\Gamma^k)\rightarrow \oplus_{\gamma\in\Gamma^k}\mathbb{Z}.
\end{equation}
We say that the elements of its kernel are \emph{stagnant}, call $\ker fl$ the \emph{stagnant} subgroup of $G_k$, denoted $ST_k(S)$, and note that 
$G_{k-1}\leqslant ST_k\leqslant G_k.$
Next 
we claim that the total flow-sum function vanishes on $C(\Gamma^k(S))$, i.e., we have for each k, 
\begin{equation}\label{total flow sum}  
\sum_{\gamma\in\Gamma^k(S)}fl_\gamma(g)=0.
\end{equation}

\begin{proof}
We choose, for a given $g\in C(\Gamma^k(S))$, a finite set $\Lambda$ of pairwise disjoint orthants representing the rank-k germs of ${\rm supp}(g)$, and with the property that $g$ restricted to each $L\in\Lambda$ is an isometric embedding $L\rightarrow S$. As $g$ fixes all rank-k germs $Lg$ is commensurable to $L$ for each $L\in\Lambda$.
\par
We claim that without loss of generality we can assume that as $L$ runs through $\Lambda$, the sets $L\cup Lg$ are pairwise disjoint. Indeed, the intersections  $L'g\cap L$ are necessarily of rank $<$ k when $L'$ and $L$ are different members of $\Lambda$; hence we find in $L$ a commensurable suborthant $K$ that avoids intersecting any of the $L'g$ with $L\neq L'$. Replacing $L$ by such a suborthant $K$ for all $L\in\Lambda$ justifies 
the claim.
\par
Let $T:=\bigcup_{L\in \Lambda}L$.
The complements of both $T$ and $Tg$ in ${\rm supp}(g)$ are of rank $\leq 
k-1$ and since g yields a pet-isomorphism between them we have $h({\rm supp}(g)-T)=h(\rm{supp}(g)-Tg)$. On the other hand, the two complements have decompositions into disjoint unions
$$ {\rm supp}(g)-T~ = ~({\rm supp}(g)-T\cup Tg)~ \cup~(Tg-T),$$
$${\rm supp}(g)-Tg~ =~ ({\rm supp}(g)-T\cup Tg) ~\cup~(T-Tg).$$
From which we infer that $h(T-Tg)=h(Tg-T)$, This establishing (\ref{total flow sum}).
\end{proof}

\subsection{Generation in $G_k$.}\label{Generation in G_k}
We start by introducing special elements $g\in G=G(\mathbb{Z}^N)$.
\begin{itemize}
\item We call $g$ a \emph{single-orthant-isometry} if ${\rm supp}(g)$ is a single orthant $L$ on which $g$ restricts to an isometry of $L$. 

\item We call $g$ an \emph{orthant-n-cycle} if we are given a set of 
pairwise disjoint orthants, cyclically connected 
by a sequence of n isometries $$L_1\overset{f_1}\rightarrow L_2
\overset{f_2}\rightarrow ... \overset{f_{n-1}}\rightarrow L_n
\overset{f_n}\rightarrow L_1,~~ with~~~ f_1f_2...f_n=Id_{L_1},$$ 
and $g$ is the union $g=\bigcup_{1\leq i\leq n}f_i$. An orthant-2-cycle 
is also called an \emph{orthant-transposition}.

\item We call $g$ a \emph{pei-translation from $L$ to $L'$ (or between $L$ and $L'$)}  if ${\rm supp}(g)= L\cup L'$ is the union of two disjoint 
orthants containing commensurable suborthants $K\subseteq L,~ K'\subseteq 
L'$ such that $g$ restricted to $K$ is the parallel shift that sends $K$ to $L$ and $g$ restricted to $L'$ is the parallel shift that sends $L'$ to $K'$. This implies that ${\rm supp}(g)$ has exactly two rank-k germs (a 
source and a sink), fixes them, and restricts to a pei-isomorphism $g\mid_{L-K}: (L-K)\rightarrow (L'-K')$, whence $h(L-K) = h(L'-K')$, as is also seen from the vanishing of the total flow function, cf. (\ref{flow map}). 

\item We call $g$ a \emph{an endotranslation} if it is supported on an orthant $L$ and parallel shifts a commensurable suborthant $K\subseteq L$ to a commensurable suborthant $Kg\subseteq L$. This implies that $h(L-K)=h(L-Kg)$. Note that this includes all elements $g\in G_{k-1}$ as the special case when $G\mid_K= id_K$.
\par
One observes easily that endotranslations are stagnant, cf. \eqref{flow map}, and that products of  endotranslations with commensurable supports are again endotranslations. Hence the set of all endotranslations supported on orthants with one and the same germ $\gamma\in\Gamma^k(S)$ forms a subgroup $E_k(\gamma)$ with 
$G_{k-1}\leqslant E_k(\gamma)\leqslant ST_k(S)$, for all $\gamma$.
\par
We write $E_k(S)$ for the group generated by all $E_k(\gamma)$ with $\gamma\in\Gamma^k(S)$.
\begin{exercise} 1. Prove that $g$ is a rank-k endotranslation if and only if $g\in C$ is supported on a rank-k orthant $L$, $\Gamma^k(L)$ is the and singleton set $\{\gamma(L)\}$, and $g$ fixes the ordering of its boundary directions.\\ 
2. Prove that $E_k(L)$ is a normal subgroup of $pei(L)$, and $pei(L)$ is the semi direct product of  is the semi direct product of $E_k(L)$ with the subgroup $Isom(L)\leqslant pei(L)$.
\end{exercise}

\item Special pei-translations $g:L\rightarrow L'$ are those when $L-K$ and $L'-K'$  are different corank-1 faces $F, F'$ of $L, L'$ and the restriction of $g$ to $F$ is an order preserving isometry. We call these the unit-pei-translations from $L$ into $L'$ and note that they are uniquely determined by the face pair $(F,F')$ and a given ordering on the canonical 
basis of $\mathbb N^N$. For simplicity we will often use "(unit)-tanslation" for "(unit)-pei-translation" when this is unambiguous.
\par
Similarly, we consider the special 
endotranslations $g:L\rightarrow L$ with the property that $K$ and $K$ and $Kg$ are the complements of two different corank-1 faces $F=L-K$ and $F'=L-Kg$. We call these the unit-endotranslations, noting that they are uniquely given by the pair $(F,F')$ of different faces and the isometry 
 $g\lvert_{F\cup F'}$.   
There are two 
possible canonical requirements that we can ask 
 $g\lvert_{F\cup F'}$ to fulfill: 
1. $g\lvert_{F\cup F'}$ is the orthant transposition given by the restriction of a 
reflection on $L$, or 2.  $g\lvert_{F\cup F'}$ is 
given by the the uniquely defined order 
preserving isometry $f:F\rightarrow 
F'$. We will always use the first option unless 
making the statement to the contrary. Thus the 
unit-endotranslations on $L$ are uniquely 
determined by their face pair $(F,F')$.
\end{itemize} 

For later reference we collect some elementary facts on the arithmetics of these special elements.
\newpage
\begin{lemma}\label{lemma4.7} \textbf{ A) Orthant-transpositions and pei-translations}
\begin{enumerate}[i)]
\item[i)] If $\tau:L\rightarrow L'$ is an orthant-transposition and $\alpha:L\rightarrow L$ a single-orthant-isometry then $\alpha=\alpha\tau\cdot\tau$ exhibits $\alpha$ as the product of two orthant-transpositions. \\
If $\sigma, \sigma'$ are single orthant reflections of $L. L'$, respectively then there is an orthant-transposition $\tau: L\rightarrow L'$ with  $\sigma\sigma'=[\sigma,\tau]$.
 
\item[ii)] Every unit-pei-translation $\lambda$ of rank k is the product of two orthant-transpositions $\lambda =\tau\tau'$ of rank k. Related to this is the observation that $\lambda^2=\tau\tau^\lambda=[\tau,\lambda]$. The translation $\lambda$ itself is not necessarily a commutator (cf. Theorem \ref{alternation} below). However, if $\rvert\Gamma^k\rvert\geq 3$ then there is an orthant-transposition $\tau$ and a unit-pei-translation $\mu$, both of rank $k$, with $\lambda=[\mu,\tau]$.

\item[iii)] Assume we are given two disjoint rank-k orthants $L, L',$ together with rank-k suborthants $K\subseteq L$ and $K'\subseteq L'$. If $h(L-K)=h(L'-K')$. Then there is a pei-translation $\lambda$ from $L$ to $L'$ which parallel shifts $K$ to $L$ and $L'$ to $K'$; and  $\lambda$ can 
be chosen as a product of unit-pei-translations.
\par
Moreover, each pei-translation of rank $k$ is equal, modulo $G_{k-1}$, to 
a product of unit-pei-translations of rank $k$.

\textbf{B) 
Single-orthant-reflections and endotranslations} 

\item[iv)] Assume 
we are given a rank-k orthant $L$ 
with two rank-k suborthants $K, K'\subseteq L$. 
If $h(L-K)=h(L-K')$ then there is an 
endotranslation  $\eta$ on $L$ which parallel 
shifts $K$ to $K'$, and $\eta$ can be chosen as a 
product of unit-endotranslations.\\ 
Moreover, each endotranslation of rank $k$ is equal, modulo $G_{k-1}$, to 
a product of unit-endotranslations of rank $k$.

\item[v)] Let 
$\sigma_{xy}:L\rightarrow L$ be the reflection of the orthant $L$ interchanging the canonical axes $x,y$ (orthogonal to faces $F_x, F_y)$ and fixing the remaining ones. Let $t_y$ denote the parallel shift of $L$ in direction $y$ by one unit into itself, and $\sigma_{xy}^{t_y}$ the corresponding reflection of $Lt_y$. Putting $\eta_{xy}:=\sigma_{xy}\sigma_{xy}^{t_y}:L\rightarrow L$ yields an explicit description of the unit-endotranslation of $L$ defined on the face pair $(F_x, F_y)$ by the restriction of $\sigma_{xy}$: We have $L=Lt_x\cup F_x$; on $Lt_x$, $\eta_{xy}$ is the diagonal shift by one (diagonal) unit in direction $y-x$ onto $Lt_x=L-F_y$, and on $F_x$ it is the restriction $\sigma_{xy}\lvert_{F_x}$ which maps $F_x$ onto $F_y$. Then we have
$$\eta_{xy}^{\sigma_{xy}}~ = ~\eta_{yx}~ =~ (\eta_{xy})^{-1},~~ [\sigma_{xy},\eta_{xy}]~=~\eta_{xy}^2,~~ and~~
\eta_{xy}\eta_{yx}^{t_y}~ = ~\sigma_{xy}\lvert_{F_x\cup F_y}.$$
We observe that if $t_y$ is induced by a unit-pei-translation $\lambda$, then $\eta_{xy}$ is the commutator $
\eta_{xy}=[\sigma_{xy},\lambda ]$, and ~ $\sigma_{xy}
\lvert_{F_x\cup F_y}=[\eta_{xy},\lambda]$. \\
Moreover, if $x,y,z$ are three pairwise different canonical basis elements of $L$ then
\begin{equation}\label{eq10}  
\eta_{xy}\eta_{yz}\eta_{zx}~=~\sigma_{yz}\mid_{F_x}, ~~~~~ and  ~~~~~~
\eta_{xz}\eta_{yx}\eta_{zy}~=~\sigma_{yz}\mid_{F_x\cup F_y},
\end{equation}
and note that $\sigma_{yz}$ is a reflection of the face $F_x$ of $L$ and $\sigma_{yz}\mid_{F_x\cup F_y}$ is a canonical orthant-transposition of the form $(F_x-F_x\cap F_y~,F_y-F_x\cap F_y)$.
\end{enumerate}
\end{lemma}

\begin{proof}
Assertion \textit{i)} is obvious.

\textit{ii)} 
Let $\lambda$ be a unit-translation from $K$ to $L$ which 
maps the face $F$ of $K$ isometrically onto the face $F
\lambda$ of $L$. Then the isometry $\lambda\lvert_F:F
\rightarrow F\lambda$ extends uniquely to an isometry $K
\rightarrow L$ which fixes $F\lambda$ pointwise, and 
thus defines an orthant-transposition $\tau=(K,L)$. Correspondingly, the restriction of $\lambda^2\lvert_F: F
\rightarrow F\lambda^2$ extends uniquely to an isometry $K
\rightarrow L-F\lambda$ and hence defines a pei-transposition $\tau'=(K,L-F\lambda)$. Both $\tau\tau'=(L,K)(K,L-F\lambda)=\lambda$ and the formula $\lambda^2=\tau\tau^\lambda=[\tau,\lambda]$ are easily seen by 
inspection.\\
If there is  a third rank-k orthant $M$ disjoint to both $K$ and $L$ we consider a new pei-transposition $\tau:=(L,M)$. Then $\lambda^{\tau}$ is 

a unit-translation from $K$ to $M$. $\lambda^{-1}
\lambda^\tau=[\lambda,\tau]$ is a unit-pei-
translation from $M$ to $L$. One checks that $
[\lambda,\tau]$ is conjugate, by an appropriate 
choice of an orthant-3-cycle of the form
$\pi=(M,L,K)$, to $\lambda=[\lambda,\tau]^\pi$. 
This shows that $\lambda=[\mu,(K,L)]$ with $\mu=
\lambda^{(K.L.M)}$.

\textit{iii)} 
This assertion is easy to accept by viewing unit-translations from $L$ to 
$L'$ as 
the process of cutting a rank-(k-1) orthant off 
from a face of $L$ and pushing it down onto a 
face of $L'$. By repeating this process with 
changing face pairs one constructs an orthant-
translation $\lambda$ from $L$ to $L'$ which 
parallel shifts $K$ to $L$, and since arbitrary 
face pairings are possible we can achieve that $
\lambda$ parallel shifts $L'$ onto an arbitrary 
given rank-k suborthant $K'$ with $h(L'-K') = 
h(L-K)$.\\
If $\lambda$ is is an arbitrary pei-translation  
of rank-k the procedure above constructs a 
product $\pi$ of unit-pei-translations of rank $k
$ that coincides with $\lambda$ on the one rank-
$k$ orthant on which $\lambda$ is a non-zero 
isometry. $\pi$ depends on the special procedure, 
but ${\rm supp}(\lambda\pi^{-1}$ is always of rank 
$k-1$. This shows that modulo $G_{k-1}$, $\lambda$ is equal to $\pi$. 

\textit{iv)}
The argument for \textit{iv)} is similar to the one in \textit{iii)} above: instead of moving $K'$ to $L$ by  sequence of parallel shifts along coordinate axes we have to move $K'$ directly to $L'$ by a sequence of pushing/pulling pairs along two axis - details left to the reader.

\textit{v)}
All formulae are proved by inspection which can be left to the reader as an exercise (in the case of formulae \eqref{eq10} start by showing that the restriction of $\eta_{xy}\eta_{yz}\eta_{zx}$ to
$(1,0,0)+L$ is the identity, and so is the restriction of $\eta_{xz}\eta_{yx}\eta_{zy}$ to $(1,0,0)+L$.
\end{proof}
The following classifies $G_k(S)/G_{k-1}(S)$ for an arbitrary orthohedral 
set $S$ up to extensions. 

\begin{theorem}\label{classifies G_k(S)/G_k-1(S)}
\begin{enumerate}[i)]
\item Each transposition $(\gamma,\gamma')$ of germs in $\Gamma^k(S)$ lifts to an orthant-transposition of representing orthants in $S$, and the action of $G_k(S)$ on $\Gamma^k(S)$ induces an isomorphism onto the finitary symmetric group, $$G_k(S)/C(\Gamma^k(S)) ~ \cong  ~sym(\Gamma^k(S)). $$

\item The action of $C(\Gamma^k(S)$ on the canonical coordinate directions $X(\gamma)$ in each $\langle\gamma\rangle$, $\gamma\in\Gamma^k(S)$ defines an isomorphism $C(\Gamma^k(S)$ onto the finitary direct product of the symmetric permutation groups of degree k,

$$C(\Gamma^k(S))/C^{ord}(\Gamma^k(S))~ \cong 
\bigoplus_{\gamma\in\Gamma^k(S)}sym_kX(\gamma)$$.
 
\item The homomorphism $\varphi =\oplus_\gamma \varphi_\gamma$ 
restricted to $C^{ord}(\Gamma^k(S))$ induces a short exact sequence
$$0\rightarrow C^{ord}(\Gamma^k(S))/G_{k-1}(S)\xrightarrow{\varphi} 
\bigoplus_{\gamma\in\Gamma^k(S)}Trans\langle\gamma\rangle \xrightarrow{\oplus_\gamma fl_\gamma}\mathbb{Z} \rightarrow 0, $$ 
In particular, $A^k(S):=C^{ord}(\Gamma^k(S))/G_{k-1}(S)$ is free-Abelian of rank $\rvert\Gamma^k(S)\lvert-1$ (which is infinite for $k< rk(S)$. Each $a\in A^k(S)$ can be represented by an element $g \in C^{ord}(\Gamma^k(S))$ with the property that~${\rm supp}(g)$ is the union of a finite set of pairwise disjoint orthants that represent the non-trivial components of $\varphi(a)$; and $A^k(S)$ is generated by unit-translations and unit-endotranslations.
\end {enumerate}
\end{theorem} 

\begin{proof}
\textit{i)} As $\gamma, \gamma'$ are different germs they can be represented by a pair of disjoint orthants, and any orthant-transposition between 
those lifts the transposition of the germs. The rest of assertion \textit{i)} is immediate from Lemmas \ref{transitive action} and \ref{germ stabilizer}.

\textit{ii)} 
Combining the homomorphisms \eqref{eq6} with $\gamma$ running through $\Gamma^k(S)$ yields a homomorphism $\overline{\varphi} = \prod_\gamma\overline{\varphi}_                                                           
\gamma: C(\Gamma^k(S))\rightarrow \prod_{\gamma}sym(X(\gamma))$, and by Lemma \ref{germ stabilizer} its kernel is $\bigcap_\gamma C^{ord}(\gamma) = C^{ord}(\Gamma^k(S))$. If $\gamma$ is an arbitrary rank-k germ with $\overline{\varphi}_\gamma (g)\neq Id$ then any orthant representing $\gamma$ is commensurable to an orthant in~${\rm supp}(g)$. This shows that $\overline{\varphi}_\gamma (g)$ has only finitely many non-vanishing components; hence we can infer that the image of $\overline{\varphi}$ is in the 
finitary product. As each permutation in sym$(X(\gamma))$ can be lifted by a single orthant isometry, the image of $\overline{\varphi}$ is, in fact, the full finitary product.

\textit{iii)} 
Combining the homomorphisms \eqref{eq5} with $\gamma$ running through $\Gamma^k(S)$ yields a homomorphism $\varphi$ of $C(\Gamma^k(S))$ into the product $\prod_\gamma$Isom$(\langle\gamma\rangle)$. As above in ii) one argues that for $g$ of rank k, $\varphi (g)$ has only finitely many non-vanishing components; hence we can infer that the restriction of $\varphi$ to the ordered germs yields a homomorphism into the direct sum
\begin{equation}\label{eq11}  
\varphi: C^{ord}(\Gamma^k(S)) \rightarrow \bigoplus_{\gamma \in \Gamma^k(S)}\emph{Trans}\langle\gamma\rangle. 
\end{equation}
By Lemma \ref{germ stabilizer} the kernel of $\varphi$ consists of the elements $g$ that pointwise fix in each rank-k orthant a commensurable suborthant; that means ${\rm supp}(g)$ contains no rank-k orthant. Hence $\ker(\varphi)= G_{k-1}(S)$ and $\varphi$ induces an embedding of $C^{ord}(\Gamma^k(S))/G_{k-1}(S)$ into the Abelian group $\bigoplus_\gamma$Trans$(\langle\gamma\rangle)$. 

We choose, for a given $g\in C^{ord}(\Gamma^k(S))$, a finite set $\Lambda$ of pairwise disjoint orthants representing the rank-k germs of ${\rm supp}(g)$, and with the property that $g$ restricted to each $L\in\Lambda$ is an isometric embedding $L\rightarrow S$. As $g$ fixes all rank-k germs 
$Lg$ is commensurable to $L$ for each $L\in\Lambda$. 

As we saw in the proof of \eqref{total flow sum} we can assume, without loss of generality, that as $L$ runs through $\Lambda$, the sets $L\cup Lg$ are pairwise disjoint, and as in that proof we put 
$T:=\bigcup_{L\in \Lambda}L$. Then we observe that $h(T-Tg)-h(Tg-T)$ 
is the total flow $fl$ and deduce from \eqref{total flow sum} that $T-Tg$ 
 and $Tg-T$ are pei-isometric.

Thus, by the pei-normal form, there is an pei-bijection $\beta: Tg-
T\rightarrow T-Tg$. Let $\alpha: T\rightarrow Tg$ denote the 
restriction of $g$ to $T$ and put $K:=\alpha^{-1}(T\cap Tg)$. 
The composition of $\alpha$ with the union $\emph{id}_{T\cap Tg}\cup\beta: Tg\rightarrow T$ now yields a pei-permutation $\mu:T \rightarrow T$ which coincides on $K$ with the restriction of $g$. Thus, the composition $g\mu^{-1} $                                                                
                          is supported in~${\rm supp}(g)$ and fixes $K$ pointwise, whence $rk(g\mu^{-1})<k$. This shows that $g=\mu$ modulo $G_{k-1}(S)$. Hence every element of $C^{ord}(\Gamma^k(S))/G_{k-1}(S)$ can be 
represented by an element $g \in C^{ord}(\Gamma^k(S))$ with the property that g is supported on a set of pairwise disjoint orthants $L_\gamma$ each of which represents its index $\gamma \in\Gamma^k({\rm supp}(g))$.

From here it is easy to prove that $C^{ord}(\Gamma^k(S))/G_{k-1}(S)$ 
is represented by a product of translations. We use induction on the number $\rvert\{\gamma\mid \varphi_\gamma(g)\neq 0\}\lvert$:  Pick a pair of 
germs $\gamma, \gamma'$, both with $\varphi_\gamma(g)\neq 0$. Then 
multiply g with a sequence of translation $\mu_i$ from $L_\gamma$ 
to $L_\gamma'$ in coordinate directions such that the product $
\mu = \prod_i\mu_i$ reverses the restriction of $g$ to a 
commensurable suborthant $K\subseteq L_\gamma$ that $g$ parallel 
shifts within $L_\gamma$. Then $g\mu$ fixes $K$ pointwise, 
hence modulo $G_{k-1}(S),~ g\mu$ is equal to a an element of $C^{ord}
(\Gamma^k(S))$ with smaller number $\rvert\{\gamma\mid \varphi_\gamma(g\mu)\neq 0\}\lvert$. The procedure ends when the $\varphi_\gamma(g) = 0$ except for one germ $\gamma$, and then $g$ is mod $G_{k-1}(S)$ is an endotranslation. In view of parts v) and vi) of Lemma \ref{lemma4.7} this proves iii).
\end{proof}
As a consequence of Theorem \ref{classifies G_k(S)/G_k-1(S)} we obtain economical generation properties. The obvious crucial fact that if $S$ is an orthohedral set of rank $rkS=n$ then $\rvert\Gamma^n(S)\lvert\in\mathbb N$, while $\rvert\Gamma^k(S)\lvert=\infty$ when $k<n$. The \emph{exceptional case} when $\rvert\Gamma^n(S)\lvert=1$ -- equivalently: $G_n$ contains no rank-n orthant-transpositions --  requires special treatment: 
In that case all rank-n elements of $G_n$ are rank-n stagnant, and this is a serious restriction on the rank-(n-1) elements that are products of rank-n orthant transpositions. E.g., non-trivial pei-translations cannot be products of single-orthant-reflections.

\begin{corollary}\label{generation properties} Let $S$ be  a stack of $h(S)$ rank-n orthants, $G_k:=G_k(S)\leqslant pei(S)$, and $\Gamma^k = \Gamma^k(S)$, with $0\leq k\leq n\in\mathbb N$. 
\begin{enumerate}[i)]
\item If $\rvert\Gamma^k\lvert\geq2$ then $G_k$ is generated by its orthant-transpositions of rank k. \\
If $\rvert\Gamma^k\lvert=1$ then $k=n$ and $G_n=pei(\mathbb N^n)$ is the normal subgroup generated by all single-orthant-reflections of rank 
n.

\item If $\rvert\Gamma^k\lvert\geq2$ then $C^{ord}(\Gamma^k)$ is generated by its pei-translations of rank-k.\\If $\rvert\Gamma^k\lvert=1$ then $C^{ord}(\Gamma^k)$ is the normal subgroup generated by the the endotranslations of rank-k.

\item If $\rvert\Gamma^k\lvert\geq 5$ then every product of two orthant-transpositions $g=\tau\tau'$, where~$\tau, \tau'\in  G_k$, can be written as a product $g=v_1v_2v_3$, where each $v_i$ is either trivial or a product $v_i=\tau_i\tau_i'$ of two \emph{disjoint} orthant-transpositions (i.e., ${\rm supp}(\tau_i)\cap {\rm supp}(\tau_i')=\emptyset$, for each $1\leq i\leq3$).
\end {enumerate}
\end{corollary}

\begin{proof}
\textit{i)} We start by proving that the claim holds true modulo $G_{k-1}$. By Theorem \ref{classifies G_k(S)/G_k-1(S)} this amounts to lift generators of the three sections $Q_1:=G_k/C(\Gamma^k)$, $Q_2:=C(\Gamma^k)/C^{ord}(\Gamma^k)$, and $Q_3:=C^{ord}(\Gamma^k)/G_{k-1}$. Now, $Q_1$ is generated by germ transpositions, and those lift to orthant transpositions. $Q_2$ is generated by transpositions of face directions, and those lift to single orthant reflections. $Q_3$ is generated by unit-pei-translations and unit-endotranslations. By Lemma \ref{lemma4.7}~iii) and~iv) we can thus infer that $G_k/G_{k-1}$ is generated by orthant-transpositions, 
single-orthant-reflections, unit-pei-translations, and unit-endotranslations.

In the \emph{exceptional case} where $G_k=pei(\mathbb N^n)$ contains neither orthant-trans\-posi\-tions nor pei-tranlations of rank~$k$, $G_k/G_{k-1}$ is thus generated by single-orthant-reflec\-tions and unit-endotranslations. Moreover, we know from Lem\-ma \ref{lemma4.7}~v) that unit-endotranslations are products of two single-orthant-reflec\-tions and that single-orthant-reflections actually suffice to generate $G_k/G_{k-1}$ in that case.

The case when $\rvert\Gamma^k\lvert\geq2$ is similar: here the existence of a rank-k orthant-transposition $\tau\in G_k$ allows to apply Lemma \ref{lemma4.7}~i) showing that all single-orthant-reflections of rank k can now be replaced by of products of two orthant-transpositions. Hence $G_k/G_{k-1}$ is generated by its rank-k orthant-transpositions in this case.

Next we prove that if $\rvert\Gamma^k\lvert\geq2$ then every rank-(k-1) orthant transposition $(F,F')
$ is a product of rank-k orthant transpositions. 
This is easy when  $F$ and $F'$ are contained in 
disjoint rank-k orthants, for then they are, in 
fact, faces of disjoint rank-k orthants $(L,L')$, 
and $(F,F')=(L,L')(L-F,L'-F')$. And if $F$ and $F'$ 
are contained in the same rank-k orthant $L$, we 
find a rank-(k-1) orthant $F''$ supported in a 
rank-k orthant disjoint to $L$, and therefore $
(F',F'')(F,F'')(F',F'')=(F,F')$. The corresponding 
weaker result in the exceptional case $\rvert
\Gamma^k\lvert=1$ is obvious: If $(F,F')$ is an 
arbitrary rank-(k-1) orthant-transposition then we 
find a rank-k orthant $L$ disjoint to $F\cup F'$,  
and by Lemma \ref{lemma4.7}v) a face-transpositions 
$(F_x,F_y)$ which is a product single orthant 
reflections of rank k. As $\rvert\Gamma^{k-1}
\lvert=\infty$ any two rank-[k-1) orthant 
transpositions are conjugate in $pei(S)$ hence the 
normal subgroup generated by $(F_x,F_y)$ contains $
(F,F')$.

Now assertion \textit{i)} follows by induction on k: Let $H\leqslant G_k$ 
be the subgroup generated by all rank-k orthant-transpositions (resp. the 
normal subgroup generated by all rank-k reflections). In the case $k=0$, we have $H=G_0$ because  $G_0$ is the finitary countable symmetric group and hence generated by its transpositions. If $k>1$ we have seen that 
the subgroup  generated by rank-k transpositions (resp. the normal subgroup generated by all rank-k reflections) contains all rank-(k-1) orthant-transpositions, and by induction those generate $G_{k-1}$. Thus the rank-k 
orthant-transpositions   
(resp. single orthant-reflections generate both $G_k/G_{k-1}$ and $G_{k-1}$, and hence $G_k$.

\textit{ii)} 
The proof along the lines of assertion \textit{i)} and can be left for the reader.

\textit{iii)}
Let  $\tau =(K,L),~\tau'=(M,N)$. If $\rvert\Gamma^k(S)
\lvert =\infty$ one finds rank-k orthants $X,Y$ 
such that $K,L,X,Y$ and $M,N,X,Y$ are pairwise disjoint 
quadruples, and $(K.L)(M,N)=(K,L)(X.Y)(X;Y)(M,N)$ as 
needed. As $S$ is a stack of orthants we infer that if $k$ is finite $\geq 5$ then $k=rkS$ and any two rank-n orthants are either disjoint or commensurable. If 
both $K$ and $L$ are commensurable to $M$ or $N$ we find 
two rank-k orthants $X,Y$ as above, the argument above 
applies. In the remaining case we may assume that $K$ is 
commensurable to $M$ but $L\cap N=\emptyset$; then we find 
two rank-n orthants $X,Y$ such that both $K,L,N,X,Y$ and 
$M,L,N,X,Y$ are pairwise disjoint quintuples, and  $(K,L)
(M,N)=(K,L)(X,N)(X,N)(Y,L)(Y,L)(M,N)$.
\end{proof}

\begin{exercise}
Prove that a) The stagnant subgroup  of $G_k$, 
$$ST_k=\ker\Big(fl:C(\Gamma^k)\rightarrow \bigoplus_{\gamma\in\Gamma^k(S)}\mathbb{Z}\Big),$$ 
is generated by $G_{k-1}$ together with all single orthant isometries of rank $k$, and also equal to the normal subgroup generated by all single orthant reflections.

b) The stagnant subgroup of
$$C^{ord}(\Gamma^k(S)) \text{ , i.e., } E_k(S):=ST_k\cap C^{ord}(\Gamma^k(S)),$$ is the normal subgroup of $G_k$ generated by all endotranslations of rank $k$. As a group it is generated by $G_{k-1}$ together with all 
endotranslations of rank $k$.
\end{exercise}

\subsection{Conjugation, Abelianization and alternation.}\label{Conjugation, Abelianization, and alternation.}
As before $S$ is 
a stack of orthants of rank n and $G:=pei(S)$, and we recall that $\rvert\Gamma^k\lvert$ is finite if and only if $k=n$. We 
say an element $g\in G_k$, $k\leq n$, is 
\emph{even} if $g$ is equal to the product of 
an even number of rank-k 
orthant-transpositions -- by Lemma~\ref{lemma4.7} this includes all single-orthant-isometries of rank $k$; and corank-1 faces of a rank-k orthants 
can be written as products of two such. We write ${\rm alt}G_k\leqslant G_k$ for the subgroup consisting of all even elements and observe that alt$G_k$ is the kernel of the \emph{rank-k parity homomorphism}, $par_{\Gamma^k}: G_k\rightarrow\mathbb Z_2$, which sends $g\in G_k$ to the parity of 
the permutation that $g$ induces on the rank-k germs $\Gamma^k$. Hence ${\rm alt}G_k$ is of index 2 in $G_k$. Moreover, by Corollary \ref{generation properties} iii) alt$(G_k)$ is generated by products of pairs of disjoint orthant-transpositions if $\rvert\Gamma^k(S)\lvert\geq5$.

We will also need a refinement of the action of $G_k$ on $\Gamma^k(S)$. When $g\in G_k$ sends the tangent coset $\langle\gamma\rangle$ to $\langle\gamma g\rangle$ then it also induces a map $g: X(\gamma)\rightarrow X(\gamma g)$ between the canonical axes directions of $\langle\gamma\rangle$ and $\langle\gamma g\rangle$.
Given an orthant $L$ representing $\gamma$ on which $g$ is isometric, and 
a canonical axis-direction $x\in X(\gamma)$, we have $x$ orthogonal to a unique corank-1 face $F$ of $L$ and $xg$ is the canonical axis direction orthogonal to $Fg$. Thus $G_k$ acts on the disjoint union $Y:=\bigcup_{\gamma\in\Gamma^k}X(\gamma)$ by finite permutations; and we have a corresponding parity homomorphism  $par_{Y^k}: G_k\rightarrow\mathbb Z_2$.\\ 
Restricted to $C(\Gamma^k)$ the parity map $par_{Y^k}$ is easy to compute: On $C^{ord}(\Gamma^k)$ even the action on $Y$ is trivial, hence we need 
only consider it on
$$C(\gamma^k)/C^{ord}(\Gamma^k)=ST_k/E_k.$$
$ST_k$ is generated by all single-orthant-reflections, and as those are the transpositions of the symmetric groups $sym(X(\gamma))$ an element of $ST_k$ has parity 0 (or is even) if and only if it is the product of an even number of single orthant reflections. This is a subgroup of index 2 in $ST_k$, we call it the alternating subgroup ${\rm alt}ST_k\leqslant ST_k$, and have the normal series
$$G_{k-1}\leqslant E_k\leqslant [ST_k,ST_k]\leqslant{\rm alt}ST_k\leqslant ST_k\leqslant C(\Gamma^k)\leqslant {\rm alt}G_k\leqslant G_k.$$

\begin{theorem}\label{alternation} 
\begin{enumerate}[i)]
\item If $\rvert\Gamma^k(S)\lvert\geq 3$ then the following holds:
	\begin{enumerate}[a)] 
	\item In $G_k$ all unit-pei-translations of rank~$k$ are conjugate, together they generate $C^{ord}(\Gamma^k(S))$, and $(G_k)_{ab}\cong\mathbb Z_2\oplus\mathbb{Z}_2$ (generated by an or\-thant-transposition and a single-orthant-reflection).
	
	\item In $G_k$ all products $\sigma_1\sigma_2$ of pairs of disjoint single-orthant-reflec\-tions of rank k are conjugate,\\
	together they generate ${\rm alt}ST_k\leqslant ST_k$,\\ 
	and
	$(C(\Gamma^k(S))_{ab}\cong (ST_k)_{ab}\cong\bigoplus_{\Gamma^k(S)}\mathbb Z_2$ (generated by single orthant reflections).
	\end{enumerate}
	
\item If $\rvert\Gamma^k(S)\lvert\geq4$ then all orthant-3-cycles $p=(L_1,L_2,L_3)$ of rank~$k$ are conjugate (more generally: If $\rvert\Gamma^k(S)\lvert\geq m$ then any two orthant-$(m-1)$-cycles are conjugate) in $G_k$ and together they generate ${\rm alt} G_k$.	
	
\item If $\rvert\Gamma^k(S)\lvert\geq5$ then all products $\tau_1\tau_2$ of pairs of disjoint rank-k orth\-ant-transpositions are conjugate and together they generate ${\rm alt}G_k$.	
\item If $\rvert\Gamma^k(S)\lvert=2$ then $(G_k)_{ab}\cong \mathbb Z_2\oplus\mathbb Z_2 \oplus \mathbb Z_2$ (generated by an orth\-ant-transposition, a single-orthant-reflection, and a unit-pei-translation).
\item If $\rvert\Gamma^k(S)\lvert=1$ (hence $G_k\cong
pei(\mathbb N^k$)) then $(G_k)_{ab}\cong \mathbb Z_2\oplus\mathbb Z_2$ (generat\-ed by a single-orthant-reflection, and a unit-endotranslation).
\end{enumerate}
\end{theorem}

\begin{proof}
\textit{ia)} Let $\lambda$, $\lambda'$ be two rank-k unit-pei-translations from $K$ to $L$, and $K$ to $L'$ respectively. If $L$ and $L'$ are disjoint then $\lambda$ and  $\lambda'$ are conjugate by an orthant transposition $(L,L')$. 
If $L$ and $L'$ are nested (and hence commensurable) an 
auxiliary rank-k orthant is available to construct a translation that sends $L$ to $L'$ or vice versa, and thus a 
conjugation between $\lambda$ and $\lambda'$. In the general 
case one finds inside $L$ a rank-k suborthant which is either 
disjoint to or contained in $L'$ and obtains the required 
conjugation in two steps. The general conjugation assertion is now obvious, and that the unit-pei-translations generate all of $C^{ord}(\Gamma^k(S))$ was established in Lemma \ref{lemma4.7}. \\
The action of the two parity homomorphisms yields a epimorphism $par_{Y^k}\times par_{\Gamma^k}:G_k\rightarrow\mathbb Z_2\oplus\mathbb Z_2$, whose 
kernel is generated by all translations together with all products of two 
single orthant reflections. By Lemma \ref{lemma4.7}i) and ii) both are commutators.

\textit{ib)} 
The proof is analogue and easier than the one of ia).

\textit{ii)}  To prove this we start by observing that if $p=(L_1,L_2,L_3)$ is an orthant-3-cycle of rank k, given by given by the pair of isometries $L_1\xrightarrow{\varphi_1}
L_2\xrightarrow{\varphi_2}L_3$, then the fact that there is an auxiliary rank-k orthant $K$ disjoint to ${\rm supp}(p)$ provides the existence of translations $\vartheta_i\in Isom(K\cup L_i)$, with the property that $\vartheta_i(K)$ is an arbitrary given commensurable suborthant of $L_i$. We 
can put them together to an element $\vartheta\in Isom(K\cup\bigcup_iL_i)$. Hence we find that $p$ is conjugate to orthant 3-cycles $p'=(L'_1,L'_2,L'_3)$, where the $L_i'$ are arbitrary given commensurable suborthants 
of $L_i$.\\
If $q=(M_1,M_2,M_3)$ is an arbitrary second orthant-3-cycle we can choose the suborthants $L_i'\subseteq L_i$ to be either contained in $M_i$ (if $L_i$ and $M_i$ are commensurable) or disjoint to $M_i$ (if $L_i$ and $M_i$ are disjoint). Thus, in order to prove that $p$ and $q$ are commensurable we can now assume, without loss of generality, that each $L_i$ is either contained in or disjoint to $M_i$. 

Now 
we complete the proof of \textit{ii)} in two steps: First we choose, for all indices $i$ with $L_i\cap M_i=\emptyset$, an arbitrary orthant transposition $\tau_i=(M_i,L_i)$. Conjugation with these $\tau_i$ shows that we find a conjugate of $p$ which replaces $L_i$ by $M_i$ whenever $L_i$ 
is not contained in $M_i$. In other words we are now reduced to a case when $L_i\subseteq M_i$ for all $i$. Repeating the first step completes the 
proof. It is clear that the argument proves, in fact, the general statement for orthant-3-cycles generate all pairs of orthant-transpositions.

\textit{iii)}
The argument is exactly like that of ia). Let $\tau_1\tau_2 =(K,L)(M,N)$. We show first that $\tau_1\tau_2$ is conjugate to $(K,L)(M,N')$ for each choice of $N'$ disjoint to $K,L,M$. This is done by the same case distinction as in ia). Then one can repeat the argument with $K$, $L$, and $M$. The generation assertion is covered by Corollary \ref{generation properties}iii).

\textit{iv)}
If $\rvert\Gamma^k(S)\lvert= 2$ then $S$ is the disjoint union of two rank-k orthants $K,L$, and we consider in $G$ an orthant-transposition $\tau=(K,L)$, a unit-pei-translation $\lambda$ from $K$ to $L$, and a single orthant reflection $\sigma$ of $K$. We have $C(\Gamma^k(S))=C^{ord}(\Gamma^k(S))ST_k$, and 
since the stagnant normal subgroup $ST_k$ contains all of 
$E_k$ but no non-trivial pei-translations 
$C(\Gamma^k(S))=gp(ST_k)$ is the semi-direct product of 
the normal $ST_k$ with the infinite cyclic group 
$gp(\lambda)$. As $G_k/C(\Gamma^k(S)$ is cyclic of order 
2 generated by $\tau$ it follows that $G_k/ST_k$ is 
isomorphic to the infinite dihedral group $gp(\lambda,
\tau)$, and its Abelianisation is the Klein-4-group 
generated by $\lambda$ and $\tau$. As $G_k/{\rm alt}
G_k$ is the Klein-4-group generated by $s$ and $\tau$, this shows that all three elements $\lambda$, $\tau, \sigma$ are needed to generate $(G_k)_{ab}$; and as $\lambda^2=[\tau,\lambda]$ we find the asserted result.

\textit{v)}
$\rvert\Gamma^k(S)\lvert= 1$. In that case $G_k=pei(L)$, 
for a single rank-k orthant $L$, and this is easily seen to 
be the semi-direct product $E_k\rtimes Isom(L)$. The 
symmetric group Isom(L) acts transitively on the k axes and 
hence on the unit-endotranslations supported on $L$: Using the notation of Lemma~ \ref{lemma4.7}v) we
have, e.g., $\eta_{xy}^{\sigma_{xz}}=\eta_{zy}$, and can 
infer that $[\sigma_{xz},\eta_{xy}^{-1}]=\eta_{zy}\eta_{xy}
^{-1}$. It follows that all endotranslation of $L$ coincide in the Abelianization, and by Lemma~ \ref{lemma4.7}v) their square is a commutator.

Lemma~ \ref{lemma4.7}v) also shows that $G_{k}'$ contains an orthant-transposition of rank equal to (k-1). As $\rvert\Gamma^{k-1}(S)\lvert=\infty$ all rank-(k-1) orthant-transpositions are conjugate and hence by Corollary \ref{generation properties} all of $G_{k-1}$ is contained in $G_k'$.
This shows that $(G_k)_{ab}$ is the Klein-4 group generated  by $\sigma_{xy}$ and $\eta_{xy}$.
\end{proof}

\subsection{The \texorpdfstring{$\mathbf{G_k/C^{ord}(\Gamma^k)}$}{TEXT}-module structure of \texorpdfstring{$\mathbf{C^{ord}(\Gamma^k(S))/G_{k-1}}$}{TEXT}}\label{subsection4.7}
By using Theorem~\ref{classifies G_k(S)/G_k-1(S)} iii) we can consider $A^k(S):=C^{ord}(\Gamma^k(S))/G_{k-1}$ as the kernel of the corank-1 germ-flow homomorphism sum in the direct sum
$$\bigoplus\limits_{\gamma\in\Gamma^k(S)}{\rm Trans}\langle\gamma\rangle.$$
Thus, each $a\in A^k(S)$ is given by a finitely supported family of translations indexed by the (in general infinite) rank-k germs, $a=(t_\gamma)_{\gamma\in\Gamma^k(S)}$, and each $t_\gamma$ is uniquely determined by its translation vector in $\mathbb{Z}^k$ with respect to the canonical basis of the tangent coset $\langle\gamma\rangle$. Hence we can write $t_\gamma$ as a row-vector $(a_{(\gamma,1)},...,a_{(\gamma,k)})\in\mathbb{Z}^k$, and note that the sum of its entries is the flow value $fl_\gamma(a)$.

Thus, 
in this section we organize the elements of $A^k(S)$  as the additive group of integral $(\Gamma^k(S)\times k)$-matrices $\mathbf{a}=(a_{(\gamma,i)})$ with only finitely many non-zero entries that add up to 0. The row 
indices are the rank-k germs $\gamma\in\Gamma^k(S)$ for a fixed number $k$, and they are endowed with a compatible ordering of the canonical bases 
of  $\langle\gamma\rangle$. The column index $1\leq i\leq k$ stands for the $i^{th}$ canonical basis element in this ordering. 

The 
quotient group $Q_k(S):= G_k/C^{ord}(\Gamma^k)$ is the 
(finitary permutational) wreath product $S_k(\beta)\wr 
{\rm sym}(\Gamma^k(S))$, where $\beta\in\Gamma^k(S)$ is a chosen base germ and $S_k(\beta)$ the symmetric group on the 
canonical basis of $\langle\beta\rangle$. ${\rm sym}(\Gamma^k(S))$ we interpret as the the permutation group on the entries which stabilizes all columns and acts diagonally by the symmetric group the set of all rank-k germs.\\
The flow $f_\gamma(\mathbf a)\in\mathbb Z$ of a matrix $\mathbf 
a\in A^k(S)$ at $\gamma$ is the sum of the entries in the $
\gamma$-row. The total flow $fl(\mathbf a)$ is the sum of all 
entries of $\mathbf a$ and by (2.6) we have $fl(\mathbf a)=0$. 

\emph{Row-subgroups}. 
A general $\gamma$-row-matrix 
represents an element of $A^k(S)$ if and only if its row-sum is zero, and 
then it is represented by an endotranslation on any $\gamma$-representing 
orthant. We write $E_\gamma\leqslant A_k(S)$ for the subgroup of all $\gamma$-row-matrices. The unit-endotranslation at $\gamma$ represent the $\gamma$-row-matrices consisting of a \emph{lone pair of entries} $(1, -1)$, 
by which we mean that all other entries are zero. 
Let $E\leqslant A^k(S)$ denote the subgroup of $A^k(S)$ generated by all row-subgroups, and observe that $E=\ker(fl: A^k(S)\rightarrow \oplus_{\gamma\in\Gamma^k(S)}\mathbb{Z})$. In particular, $E$ is a $Q_k(S)$-submodule of $A^k(S)$, and as any finite set of germs can be represented by pairwise disjoint orthants we have $E=\bigoplus_{\gamma\in\Gamma^k(S)}E_\gamma$. 

\emph{Column-subgroups}. Lei $i$ be a natural number $\leq k
$. A general finite $i^{th}$-column-matrix $(n_\gamma)_{\gamma
\in\Gamma^k}$ is given by a finitely supported map $f_i:
\Gamma^k(S)\rightarrow\mathbb{Z}$, $n_i:=f_i(\gamma)$ and 
defines the parallel shift of a finite set of pairwise disjoint  orthants 
$L_\gamma$ of $\langle\gamma\rangle$ in the direction of their $i^{th}$ axis onto $n_iL_\gamma$. 
This defines an element of $A^k(S)$ if and only if the column sum $\sum_{\gamma\in\Gamma^k}n_\gamma$ is zero, and we write $C_i\leqslant A^k(S), ~i = 1, 2, ..., k$,  for the subgroup of all $i^{th}$-column matrices. Note that every matrix $\mathbf{a}$ with a lone pair of unit entries $1, -1$ in different rows $(\gamma,\gamma')$ is represented by a unit-pei-translation $\lambda$; and $\mathbf{a}$ is a column matrix if and only if the 
two unit shifts of $\lambda$ are anti-parallel.

\emph{The diagonal subgroup $D\leqslant A^k(S)$}. The
column subgroups $C_i$ are invariant under the order preserving action of 
$sym(\gamma^k(S))$ but not, of course, under all of $Q_k(S)$ -- nor is the direct sum $\bigoplus_{1\leq i\leq k}C_i$. Only the diagonally embedded 
copy of $C_1$ into $A^k(S)$, i.e., the group $D\leqslant A^k(S)$ of all matrices with constant rows and zero column sums is actually a $Q_k(S)$-submodule of $A^k(S)$. 

\begin{lemma}\label{lemma4.11} 
For every $Q_k(S)$-submodule $M\leqslant A^k(S)$ we have:
\begin{enumerate}[i)]
\item If $\mathbf m$ is a matrix in $M$, so is the matrix $
(fl_1(\mathbf m),~...~,fl_k(\mathbf m))\in D$ which has in each of its columns the flow-column of $\mathbf m$. 

\item If $\Gamma ^k(S)$ is infinite then we have: for every choice of a 
pair $(\gamma,\gamma')$ of different germs in $\Gamma^k(S)$, $M
$ is generated by $E_\gamma\cap M$ together with the lone-pair-of-rows matrices supported on the  $(\gamma,\gamma')$-rows.
(In other words: $M$ is generated by its endotranslations of $            
                                              \gamma$-orthants and its pei-translations between $\gamma$ and $\gamma'$).  

\item Either $M\leqslant D$ or there is a unique minimal 
natural number $q$ with $qE\leqslant M.$
\item Either $M\leqslant E$ or there is a unique minimal natural number $p$ with $pD\leqslant M$.
\end{enumerate} 
\end{lemma}

\begin{proof}
\textit{i)} As $\mathbf m$ has only finitely 
many non-zero rows we find a n element $\vartheta\in Q_k(S)$ 
which permutes its columns cyclically. It follows that $
\mathbf m + \mathbf m\vartheta+ ... + \mathbf m\vartheta^{k-1}
$ is contained in $M$ and has the required form.

\textit{ii)} 
As every matrix $0\neq \mathbf m\in M$ has only 
finitely many non-zero rows, $\mathbf m$ contains both a 
non-zero row and a zero row, and $Q^k(S)$contains a transposition that 
interchanges the two. Thus $M$ contains a lone-row-pair matrix 
of the form $\mathbf m(1-\tau)=\binom{-\alpha}
{\alpha}$ for every row $\alpha$ of $\mathbf m$. As $Q^k(S)$ acts 2-transitively on the rows we can assume that here the entry $-\alpha$ stands in 
a pre-chosen $\gamma$-row while $\alpha$ has its original position. If we 
subtract all these lone-pair-of-rows matrices for all non-zero rows $\neq\gamma$ from $\mathbf m$ we find the $\gamma$-row matrix whose entries are the column sums of $\mathbf m$, i.e., $\mathbf m':=\big(\sum_{\gamma\in\Gamma^k}n_{\gamma 1},~.~.~.~ ,~ \sum_{\gamma\in\Gamma^k}n_{\gamma k}\big)\in M$. The flow of $\mathbf m'$ is the total flow of $\mathbf m$ and hence zero. This shows that $\mathbf m'\in E_\gamma$. Thus $M$ is generated by $E_\gamma\cap M$ together with all lone-pair-of-rows matrices in $M$. As $Q_k(S)$ acts k-transitively, each lone-pair of rows matrix of $M$ is conjugate to a lone $(\gamma,\gamma')$-pair-of-rows matrix of $M$.

\textit{iii)} 
We assume $M$ is not contained in $D$, i.e., it contains a matrix $\mathbf a$ which contains a $\gamma$-row 
with two different entries $x\neq y\in\mathbb Z$. Let $\tau\in S_k(\gamma)$ be the transposition that interchanges those two entries. Then $\mathbf a(1-\tau)$ is a 
matrix in $E_\gamma\cap M$ with a \emph{lone-pair of 
entries} of the form $(z,-z)$. We consider the smallest 
natural number $q$ with the property that $E_\gamma\cap M$ contains a matrix with a lone-pair of entries of the form $(q,-q)$, and call this a \emph{minimal lone-pair-of-entries matrix }$E_\gamma\cap M$. Since $Q_k(S)$ acts highly transitively on the rows this applies to each rank-k germ $\gamma$. With familiar 
arguments one observes that each row matrix of $M$ with a 
lone pair of entries is a multiple of a minimal lone-pair-of-entries row matrix; hence the latter generate in each $E_\gamma$ the subgroup $q(E_\gamma\cap M)\leqslant M)$ and in $E\cap M$ the $Q_k(S)$-submodule $qE\leqslant M$.

\textit{iv)} 
Now we assume that $M$ is not contained in $E$. Then there is a matrix $\mathbf{a}\in M$ with non-zero flow-column $\mathbf{fl(a)}:=(fl_{\gamma} 
(\mathbf{a}))_{\gamma\in
\Gamma}$. As the sum of the entries of $\mathbf{fl(a)}$ is 
zero, its entries cannot be constant; hence $\mathbf{fl(a)}$ 
contains a pair of non-equal entries. Therefore $\mathbf{a}$ 
contains a pair of non-equal rows, and we find a row-transposition $\tau$ 
which interchanges these two rows. $
\mathbf{a}(1-\tau)$ is then a \emph{lone-pair-of-rows} matrix in $M$ of the form $\binom{\alpha}{-\alpha}$ which has a lone pair of non-zero entries in its flow column $\mathbf {fl(a}(1-\tau))$. It follows that there is 
a smallest natural number $p$ with the property that $M$ contains a lone-pair-of-rows matrix $\mathbf{m}=\binom{\alpha}{-\alpha}$  with a 
lone-pair-of-entries flow-column of the form $\binom{\pm p}
{\mp p}$. We call $\mathbf{m}$ a \emph{flow-minimal} lone-pair-of-rows matrix of $M$.\\
By part \textit{i)} it follows that the constant lone-pair-of-rows matrix
$$\pm p\binom{~1, ~.~.~.~, ~1}{-1,~.~.~.~, -1}$$
is also a flow-minimal lone-pair-of-rows matrix in $M$, whence $pD\leqslant M$.
\end{proof}

We are now in a position to describe all $Q_k(S)$-submodules of $A^k(S)$.

\begin{theorem}\label{theorem4.12} The $Q_k(S)$-submodules $0\neq M\leqslant A^k(S)$ are of the following types:
\begin{itemize}
\item $pD\leqslant M\leqslant D$, for some $p\in\mathbb N$, $D/pD$ contains only finitely many $Q_k(S)$-submodules, and if $p$ is minimal and $\Gamma^k(S)$ is infinite then $M=pD$.
\item $qE\leqslant M\leqslant E$, for some $q\in\mathbb N$, $E/qE$ contains only finitely many $Q_k(S)$-submodules, and if $q$ is minimal and $\Gamma^k(S)$ is finite then $M=\mathbb{Z}Q_k(S)(E_\gamma\cap M)$.
\item $pD+qE\leqslant M\leqslant A^k(S)$, for some $p,q\in\mathbb N$, and 
$A^k(S)/pD+qE$ contains only finitely many $Q_k(S)$-submodules.
\end{itemize}
\end{theorem}
\begin{proof}
Note that $D$ is free-Abelian with a countable basis $X$, and $Q:=Q^k(S)$ acts on $D=\mathbb{Z}[X]$ via the symmetric group $sym(X)$. If $M$ is contained in $D$ it cannot be contained in $E$ hence by Lemma \ref{lemma4.11} vi) there is a unique minimal $p\in\mathbb{N}$ with $pD\leqslant M$. Thus, $D/pD\cong\mathbb{Z}_p[X]$. If $X$ is finite so is $D/pD$. If $X$ is infinite and $x,y\in X$ are two different elements then (a special case of) Lemma  \ref{lemma4.11}ii) implies that every $Q$-submodule of $M/pD\leqslant\mathbb{Z}_p[X]$ is generated by lone pairs of row-matrices in 
$\{tx-ty\mid t\in\mathbb{Z}_p\}$, which is a finite subset of $D/pD$ independent of $M$. Thus, in both cases we find that $D/pD$ contains only finitely many $Q$-submodules.\\
The case when $M\leqslant E$ is similar: $M$ cannot be in $D$; hence Lemma \ref{lemma4.11}iii) applies and asserts that there 
is a unique minimal $q\in \mathbb N$ with $qE\leq M$. If 
$\Gamma^k(S)$ is finite, so is $E/qE$. If $\Gamma^k(S)$ is 
infinite then (a special case of) Lemma ~ \ref{lemma4.11}ii) 
implies that for any chosen $\gamma\in\Gamma^k(S)$ all $Q$-submodules $M\leqslant E$ are generated by $E_\gamma\cap M$. 
Hence every submodule of $E/qE$ is generated by elements in 
the finite set $E_\gamma /qE_\gamma$. Thus, in both cases we find that $D/pD$ contains only finitely many $Q$-submodules.

If $M$ is neither contained in $D$ nor in $E$ then $pD+qE
\leqslant M$, and therefore $t(D+E)\leqslant M$ for 
$t:=gcd(p,q)$. In this situation we fix a germ $\gamma\in
\Gamma^k(S)$ and consider the $\gamma$-row matrices $
\mathbf{a}:=(t,t,...,t),  ~\mathbf{b}:=(t,-t,0,...,0)$, both elements 
of $D+E\leqslant M$, and the 
element $\vartheta\in Q$ which cyclically permutes the entries 
of the $\gamma$-rows. By observing that 
$$\mathbf{a}-\mathbf{b}(1+\vartheta +2\vartheta^2+3\vartheta^3+...+k\vartheta^k) = (kt,0,...,0)\in t(D+E)\leqslant M$$
we can infer that $ktA^k(S)\leqslant D+E\leqslant M$. Hence it remains to 
show that $A^k(S)/ktA^k(S)$ contains only 
finitely many submodules. This is again obvious when $
\Gamma^k(S)$ is finite, for in that case $A^k(S)/ktA^k(S)$ is 
a finite Abelian group. If $\Gamma^k(S)$ is infinite Lemma 
 \ref{lemma4.11} ii) asserts that for any two different germs $
\gamma, \gamma'\in\Gamma^k(S)$ we know that all submodules $M$ 
of $\Gamma^k(S)$ are generated by $E_\gamma\cap M$ together 
with all lone pairs of $(\gamma,\gamma')$-rows in $M$. Modulo 
$kt$, this shows that all submodules of $A^k(S)/ktA^k(S)$ are 
generated by a subset of a finite set which depends only on $D
$ and $E$. Hence $A^k(S)/ktA^k(S)$ contains only finitely many 
submodules.
\end{proof}

The fact that $A^k(S)$ is generated by two 
$Q_k(S)$-orbits -- unit-pei-transla\-tions and 
unit-endotranslations -- shows that $A^k(S)$ is a 
finitely generated $Q_k(S)$-module. From Theorem 
\ref{theorem4.12} we infer that all submodules of 
$A^k(S)$ are finitely generated; in other words:

\begin{corollary}\label{Corollary4.13} 
The $Q^k(S)$-module $A^k(S)$ is Noetherian.\hfill$\square$
\end{corollary}

\subsection{The finite subgroups.}
Our next results will be used in Section 4.9 to classify all normal subgroups of $pei(S)$. On the side it also yields all finite subgroups.

\begin{lemma}\label{lemma4.14}
Let $S$ be orthohedral, $L\subset S$ an orthant of rank k, and $g\in G_k$. If $g$ has the property that its image in $G_k/G_{k-1}$ of finite order $m$ then we can find a commensurable suborthant $K\subset L$ with the property that the sequence
\begin{equation}\label{eq12}
K\xrightarrow{g}Kg\xrightarrow{g}Kg^2\xrightarrow{g} ... \xrightarrow{g}Kg^{\rvert\Lambda\lvert -1}\xrightarrow{g}Kg^{\rvert\Lambda\lvert}=K
\end{equation}
goes through a set $\Lambda$ of pairwise disjoint orthants representing the germs $\gamma (L)g^j$, $j\geq 0$, and ends with an isometry $\alpha:=g^{\rvert\Lambda\lvert}\rvert_K: K\rightarrow K$.
\end{lemma}
We call (\ref{eq12}) the the \emph{covering orthant-orbit} of the germ-orbit $\gamma(L)gp(g)$.
\begin{remark}
This applies also in the case when $\gamma g=\gamma$ (and even when $rk(g)<k$): Then the assertion is $\gamma$ is represented by a rank-k orthant $K$ pointwise fixed by $g$.
\end{remark}

\begin{proof}
By Lemma \ref{lemma3.3} we can assune that the restriction of $g$ to $L$ is an 
isometric embedding. If $Lg$ is not commensurable 
to $L$ then the intersection $L\cap Lg^j$ is of 
smaller rank for all $1\leq j <t$=length of the length of the $g$-orbit 
of the germ of $L\in\Gamma^k$. Hence we find a commensurable suborthant $K\subset L$ with the property that $Kg^j\cap K=\emptyset$ for all $1\leq j <t$. This implies that we have a sequence like the one asserted to exist in the lemma, except that it ends with an isometry $\alpha: K\rightarrow Kg^t$ onto a commensurable orthant $K	
g^t$. But by assumption $\alpha'$ cannot be of infinite order, hence further powers $g^J$ will, after finitely many steps, come back to $K$. Taking their intersections yields the claimed assertion.
\end{proof}

The following Theorem extends the Lemma from a single element $g$ to a finitely generated subgroup $H$.  

\begin{theorem}\label{theorem4.15} Let $H\leqslant G_k$ be 
a finitely generated subgroup whose rank-$(k-1)$ subgroup 
$N:=H\cap G_{k-1}$ is of finite index in $H$. Then we find   a set $\Lambda$ of pairwise disjoint orthants represening 
the germs in $\Gamma^k({\rm supp}(H))$ with the property that $H
$ acts on $S':=\bigcup_{L\in \Lambda} L$ by isometries  (on and between 
the members of $\Lambda$) -- with the 
understanding this includes the assertion that $N$ fixes $S'$ pointwise.
\end{theorem}

\begin{proof}
The set of rank-k germs $\Gamma^k({\rm supp}{H})$ is finite and permuted by $H$, and we can represent the germs $\gamma\in\Gamma^k(S)$ by pairwise 
disjoint rank-k orthants $L_\gamma$. 
Moreover, by passing, if necessary, to commensurable suborthants we may, for 
a given finite set $P\subseteq H$, assume that the 
restrictions $f\mid_{L_\gamma}$ are isometric for all $f\in P$ and all $\gamma\in\Gamma^k(S)$. The proof of Lemma \ref{lemma4.14} shows that this is true for a single element $g$, and the general case
follows by induction on $\rvert P\lvert$: For the inductive step we can assume that the orthants $L\gamma$ we start with already satisfy the conclusion for a proper subset of $P$ and then go through arguments in the proof of Lemma \ref{lemma4.14} for an additional element. Note also that each $n\in N\cap P$ will have the feature to act trivially on $S'$. 

Now we take advantage of this by applying it to a set $P$ which we choose 
as follows: First we pick a transversal $T\subset H$ of $H/N$ which contains the unit element of $H$ and put $X:=T^{\pm 1}$; then we consider all triple products $XXX\subset H$ and and pick a finite subset $Y\subset N$ which contains the set $N\cap XXX$ and generates $N$ as a monoid; finally we put $P:=X\cup Y$ and note that $P$ generates $H$ as a monoid. Now 
we observe:
\begin{itemize}
\item   As $Y\subset P$ we have ${\rm supp}(N)\cap 
L_\gamma =\emptyset$ (in particular $L_\gamma y=L_\gamma$) 
for all $\gamma\in\Gamma^k(S)$. 
\item As $T^{\pm 1}\subseteq P$ the translates $L_\gamma t$ are 
orthants commensurable to $L_{\gamma t}$ for all $t\in T$, and 
the $T_{\gamma t}t^{-1}$ are orthants commensurable to $L_\gamma$. 
\item For each $\gamma\in\Gamma^k(S)$ we now consider the intersection
$$K_\gamma:=\bigcap_{t\in T}L_{\gamma t}t^{-1}.$$
This is a finite intersection of orthants commensurable to $L_\gamma$ and 
hence is a suborthant contained in and commensurable to $L_\gamma$. Thus $\Lambda :=\{K_\gamma\mid \gamma\in\Gamma^k(S)\}$ is a pairwise disjoint set of representatives of the germs in $\Gamma^k(S)$ on which all restrictions of elements in $P$ are isometric injections.                      
 
\end{itemize} 
We aim to prove that the elements of $P$ act on and/or  permute the members of $\Lambda$ by isometries. For the elements in $Y$ we know this already. To prove it for $x\in X=T^{\pm 1}$ we note that for each pair $(t,x)\in T\times X$ there is a unique $s\in T$ with $n:=t^{-1}xs\in N$. From here we find, on the one hand, $L_{\gamma t}t^{-1}x=L_{\gamma t}(t^{-1}x)=L_{\gamma t}ns^{-1}=L_{\gamma t}s^{-1}$ since $n=\in P$; and on the other hand, $\gamma t=\gamma xsn^{-1}=xsn^{-1}(xs)^{-1}xs=\gamma 
xs$ since $xsn^{-1}(xs)^{-1}\in N$ which acts trivially on 
the rank-k germs. Hence $L_{\gamma t}t^{-1}x=L_{\gamma xs}s^{-1}$, and we find 
$$K_\gamma x=(\bigcap_{t\in T}L_{\gamma t}t^{-1})x=\bigcap_{t\in T}L_{\gamma t}(t^{-1}x)=\bigcap_{t\in T}L_{\gamma xs}s^{-1}=\bigcap_{s\in 
T}L_{\gamma xs}s^{-1}=K_{\gamma x}.$$
This shows that the monoid generators of $H$ and hence $H$ itself acts on 
the union $S'':=\bigcup_{L\in\Lambda}L$ as asserted. Thus
$S$ has the $H$-invariant decomposition of $S$ as the union of $S''$ and its complement $S':=S-S''$; and as $S''$ covers all rank-k germs of $S$ 
its complement is of rank at most $(k-1)$.
\end{proof}

\begin{corollary}\label{corollary2.16}. A subgroup $H\leqslant G_k$ is finite if and only if ${\rm supp}(H)$ is the union of a finite set $\Lambda$ of pairwise disjoint orthants $\bigcup_{L\in\Lambda}L$ on which $H$ acts faithfully by means of isometries on and between the members of $\Lambda$. \hfill$\square$.\end{corollary}

\subsection{The normal subgroups of pei(S)}\label{subsection4.9}  Throughout this section $S$ is an orthohedral set of rank $rkS=n$, germs $\Gamma^k:=\Gamma^k(S)$, and height $h(S)=\rvert\Gamma^n\lvert$; and $G:=pei(S)$. The most important normal subgroups of $G$ we have met so far are the \emph{rank subgroups}, and between them the 
(ordered and unordered) germ stabilizers $C^{ord}(\Gamma^k)\leq C(\Gamma^k)$. But in addition to those we found also the stagnant subgroups $ST_k\leqslant C(\Gamma^k)$ (with $C^{ord}(\Gamma^k)ST_k = C(\Gamma^k)$, see (\ref{eq5})), the endotranslation subgroup $E_k:=ST_k\cap C^{ord}(\Gamma^k)$ (see \emph{Exercise} at the end of Section \ref{Generation in G_k}), and the alternating subgroups alt$G_k$ (see Theorem \ref{alternation}). 
The lattice of these normal subgroups is exhibited in the diagram

\begin{eqnarray}
& \leqslant C^{ord}(\Gamma^k)\leqslant &  \nonumber\\
G_{k-1}\leqslant E_k & & C(\Gamma^k) \leqslant {\rm alt}G_k\leqslant G_k. 
\nonumber\\
 & \leqslant{\rm alt}ST_k\leqslant ST_k\leqslant & 
\end{eqnarray}

The 
$G$-module structure of $A^k(S)=G_k/G_{k-1}$ exhibited in Section \ref{subsection4.7} provides detailed information on the normal subgroups between $G_{k-1}$ and $C^{ord}$. The goal is now to show that what we have seen so far covers essentially all normal subgroups of $G$.\\
To prove this requires detailed information on the normal 
subgroups $gp_G(g)$ of $G$ generated by specific 
elements $g\in G_k$, $k\leq n$, and we start by investigating the special 
case when the canonical image 
of $g$ in $G_k/G_{k-1}$ is of positive finite order $m$. In 
this case Lemma \ref{lemma4.14} describes the 
covering orthant-orbit $\Lambda$ of a given $g$-orbit of $\Gamma^k$, and how $g$ acts on the 
union $\bigcup\Lambda$ by isometries on and 
between its members $L\in\Lambda$.

\begin{lemma}\label{lemma4.17} Let $g\in G_k$, with $k\leq n=rkS$, be an element whose image in $G_k/G_{k-1}$ is of finite order. Then the following holds (but note that if $k=n$ then $\rvert\Gamma^k\lvert$ is finite, and if $k<n$ then $\rvert\Gamma^k\lvert=\infty$ and only assertion \textit{v)} is relevant).
\begin{enumerate}[i)]
\item $rk(g)=k$ alone implies already that $gp_{G}(g)$ contains  ${\rm alt}G_{k-1}$.
\item If $g$ acts non-trivially on $\Gamma^k$ 
then $gp_G(g)$ contains in addition to ${\rm alt} G_{k-1}$ \emph{certain} 
products of pairs of disjoint rank-k elements of the form $\varphi\varphi^g$, where $\varphi\in G_k$ is a single-orthant-reflection or an endotranslation of rank~$k$.
\item 
If $g$ acts non-trivially on $\Gamma^k$ and $\rvert\Gamma^k\lvert\geq3$ then $gp_G(g)$ contains ${\rm alt}ST_k$ (which includes $G_{k-1}$ and $E_k$). In addition we have:\\ 
 $gp_G(g)$ contains either an orthant-3-cycles of rank k, or the product of a pair of disjoint orthant-transpositions. (In the second case we have 
$\rvert\Gamma^k\lvert\geq4$ and further consequences below apply).

\item 
If $g$ acts non-trivially on $\Gamma^k$ and $\rvert\Gamma^k\lvert\geq4$ then $gp_G(g)$ contains
$$C^{ord}(\Gamma^k)^2{\rm alt}ST_k$$
which is a subgroup of finite index in $G$. If g acts on $\Gamma^k$ by a 3-cycle or by a single transposition then $gp_G(g)$ contains the commutator subgroup $G'$.
 
\item
If $g$ acts non-trivially on $\Gamma^k$ and $\rvert\Gamma^k\lvert\geq5$ then $gp_G(g)$ contains  ${\rm alt}G_k$. 
\end{enumerate}
\end{lemma}

\begin{proof}
\textit{i)}
Let $L\in\Lambda$ be an orthant contained in 
${\rm supp}(g)$. Then -- regardless whether $Lg=L$ or 
$Lg\neq L$ -- there is a rank-(k-1) face $F$ 
of $L$ with $F\neq Fg$. We claim that one can 
choose an orthant-transposition $\tau=(K,K')$ of 
rank $k-1$ supported on suborthants of $L$ 
parallel to $F$ and of distancce 1 to each 
other, with the feature that $\tau$ and $\tau g$ 
are disjoint. Indeed, if $Lg\neq L$ then taking 
$K:=F$ and $K'$ its parallel neighbor of 
distance 1 will do; and if $Lg=L$ then we can 
take $K$ to be the orthant obtained by shifting 
$F$ diagonally into itself by two diagonal 
units, and $K'$ its parallel neighbor of 
distance 1 (which has distance $>1$ from all 
other faces). Due to Lemma \ref{closed under complements} $\tau\in G_k$. Then the commutator $[\tau,g]=\tau\tau^g$ is the product of two disjoint rank-$(k-1)$ orthant-transpositions and contained in $gp_{G_k}(g)$. As $\rvert\Gamma^{k-1}\lvert =\infty$ we know from Theorem \ref{alternation} iii) that the conjugates of $[\tau,g]$ generate  ${\rm alt}G_{k-1}$.

\textit{ii)} 
By assumption $\Gamma^k$ has a $g$-orbit of length $\geq 2$, and we consider the corresponding covering orthant-orbit $\Lambda$. Let $L\in\Lambda$, $K\subseteq L$ an arbitrary rank-k suborthant, and $\sigma$ a single-orthant-reflection of $K$. Due to Lemma~ \ref{germ stabilizer} $\sigma\in G_k$, hence  $[g,\sigma]=\sigma^g\sigma\in gp_{G_k}(g)$, and $K\cap Kg=\emptyset$. Thus, $gp_{G_k}(g)$ contains a product of two disjoint $g$-conjugate single-orthant-reflections of rank k.\\
We can apply this to the suborthant $Kt_y\subset K$ and recall from Lemma 
\ref{lemma4.7}~v) that 
$\eta=\sigma\sigma^{t_y}$ is a unit-endotranslation. As both $\sigma^g\sigma$ and ${\sigma^{t_y}}^g
\sigma^{t_y}$ are in $gp_{G_k}(g)$, so is their 
product $(\sigma^g\sigma)({\sigma^{t_y}}^g
\sigma^{t_y})=\eta^g\eta$. This shows that $gp_G(g)$ also contains products of pairs of disjoint $g$-conjugate (unit)-endotransla\-tions.

\textit{iii)}
We assume that $g$ acts non-trivially on $\Gamma^k$ and $\rvert\Gamma^k\lvert\geq 3$. Three cases occur:\\
Case 1: $\Gamma^k$ contains a $g$-orbit of length $\geq 3$. Let $\Lambda$ 
be the corresponding covering orthant-orbit, and $K\subseteq L$ an arbitrary rank-k suborthant of some $L\in\Lambda$. Then $K, Kg, Kg^2$ are pairwise disjoint. We put $\tau:=(K,Kg)$ to be the orthant-transposition defined by the restriction $g\rvert_K:K\rightarrow Kg$, and observe that the 
commutator $[g,\tau]=\tau^g\tau=g^{-1}g^\tau$ is contained in $gp_G(g)$ and is the orthant-3-cycle $(K\xrightarrow{g}Kg\xrightarrow{g}Kg^2\xrightarrow{g^{-2}}K)$.\\
Case 2: $\Gamma^k$ contains a $g$-orbit of length 2 (with covering orthant-orbit $L
\xrightarrow{g}Lg\xrightarrow{g}L$), and disjoint to it is a $g$-invariant rank-k orthant $L'$ (with covering orthant-orbit $L'\xrightarrow{g}L'$). We choose arbitrary rank-k suborthants $K\subseteq L$, $K'\subseteq L'$, pick an 
orthant-transposition $\tau:=(K',K)$, and observe that the commutator $[g,\tau]=\tau^g\tau=g^{-1}g^\tau$ is contained in $gp_G(g)$ and is the orthant-3-cycle $(K\xrightarrow{\tau} 
K'\xrightarrow{\tau^g}Kg\xrightarrow{\tau^g\tau}K)$. Appropriate products 
of two orthant-3-cycles are products of pairs of disjoint orthant-transpositions.  By Theorem \ref{alternation} ib) all products of pairs of disjoint orthant transpositions are conjugate and generate ${\rm alt}ST_k$. Hence $gp_G(g)$ contains the unique subgroup of index 2 in $C(\Gamma^k)$ and all alternating finite orthant-permutations; together these generate the commutator subgroup of $G_k$. \\
Case 3: $\Gamma^k$ contains two $g$-orbit of length 2 (with corresponding 
covering orthant-orbits $L_i\xrightarrow{g}L_ig\xrightarrow{g}L_i$, i=1, 2). We choose arbitrary rank-k 
suborthants $K_1\subseteq L_1$, $K_2\subseteq L_2$, pick an orthant-transposition $\tau:=(K_1,K_2)$, and observe that the commutator $[g,\tau]=\tau^g\tau=g^{-1}g^\tau$ is contained in $gp_G(g)$ and is the product $\tau_1\tau_2$ of two disjoint $g$-conjugate orthant-transpositions. By Theorem \ref{alternation} ib) all products of pairs of disjoint orthant transpositions are conjugate and generate ${\rm alt}ST_k$.

\textit{iv)}
Now we assume $\rvert\Gamma^k\lvert\geq 4$. By Theorem \ref{alternation} ii) we know 
that all orthant-3-cycles are conjugate. 
Thus, if $gp_G(g)$ contains an orthant-3-
cycle of rank k then $\rvert\Gamma^k
\lvert\geq 4$ implies that it
contains all of them. We claim that this 
implies $C^{ord}(\Gamma^k)\leqslant 
gp_G(g)$.\\
 To prove this recall that every unit-pei-translation $\lambda$ is the products of two orthant-transpositions $\lambda=\tau\tau'$ of the form $(K,L)(K,L-F)$ (see the proof of  Lemma \ref{lemma4.7}ii)). As an additional rank-k orthant $M$ disjoint to $K\cup L$ is available we can write $\tau\tau'$ as the product $\lambda=(K,L)(K,L-F)=(K,L)(K,M)(K,M)(K,L-F)=(K,L,M)(K,M,L-F)$ of two orthant-3-cycles, whence $\lambda\in gp_G(g)$. Our claim follows since the unit-pei-translations generate $C^{ord}(\Gamma^k)$.\\
By \textit{iii)} we know that $gp_G(g)$ contains also ${\rm alt}ST_k$, and together with $C^{ord}(\Gamma^k)$ this yields the unique subgroup of index 2 in $C(\Gamma^k)$. Moreover the orthant-3-cycles generate, in the symmetric group $G/C^{ord}(\Gamma^k)$, the alternating subgroup of index 2. 
Hence $gp_G(g)$ contains a normal subgroup of index 4;  this can only be the commutator subgroup $G'$.\\
It remains to consider the case $gp_G(g)$ contains no orthant-3-cycle -- this happens only when  $\rvert\Gamma^k\lvert=4$ (which implies that $k=n=rk(S))$ and $\Gamma^k$ is the union of two $g$-orbits of length 2. 
In that case assertion iii) still tells us that $gp_G(g)$ contains all of 
$E_k$ and products $\tau_1\tau_2$ of pairs of disjoint orthant-transpositions of rank k.\\ 
From pairs of disjoint orthant-transpositions we can obtain the result for pairs of disjoint translations: We use our disjoint orthant-transpositions $\tau_i=(K_i,L_i)$ to construct the disjoint unit-pei-translations $\lambda_j=(K_i,L_i)(K_i,L_i-F_i)$, where $F_i$ stands for a rank-$(k-1)$-face of $L_i$. Then we have $\lambda_1\lambda_2=(K_1,L_1)(K_2,L_2)(K_1,L_1-F_1)(K_2,L_2-F_2)$ which shows that $\lambda_1\lambda_2\in gp_G(g)$. It follows that all products of pairs of disjoint unit-pei-translations are in $gp_G(g)$. By multiplying two appropriate such pairs we find that $gp_G(g)$ contains all squares of unit-pei-translations and therefore 
all translations of even length, i.e., $C^{ord}(\Gamma^k)^2\leqslant gp_G(g)$. Hence $gp_G(g)$ contains $C^{ord}(\Gamma^k)^2{\rm alt}ST_k$ and all 
products of disjoint pairs of orthant-transpositions. We leave it to the reader to prove that this is a finte index subgroup of $G_k$ is a subgroup of finite index in $G_k$.  

\textit{v)}
We assume that $g$ acts non-trivially on $\Gamma^k$ and $\rvert\Gamma^k\lvert\geq 5$. We know by 
assertion iii) that $gp_G(g)$ contains either an 
orthant-3-cycle or a product of two disjoint 
orthant-transpositions. As $(K,L,M)(L,M,N)=(K,M)
(L,N)$ we have products of disjoint orthant-
transpositions in either case and we know, by 
Theorem \ref{alternation}, that they generate  $
{\rm alt}G_k\leqslant G_k$ as a normal subgroup.
\end{proof}

We can now prove that the index-2 pairs ${\rm alt}
G_k<G_k$ are "bottlenecks'' for the normal 
subgroups $N$ of $G$, i.e., either $N\leqslant G_k$ 
or ${\rm alt}G_k\leqslant N$. Or, equivalently: 
    
\begin{theorem}\label{Bottleneck Theorem}\emph{(Bottleneck Theorem)} For every normal subgroup 
 $N\leqslant G$ of rank $rk(N)=k$ we have  ${\rm alt}G_{k-1}\leqslant N\leqslant G_k;$ (recall that $G_{-1}:=1$).
\end{theorem} 

\begin{proof}
The key here is proving that the assertion i) of Lemma \ref{lemma4.17}, i.e., $rk(g)=k$ alone implies ${\rm alt}G_{k-1}\leqslant gp_G(g)$, holds  without the assumption that the image of $g\in G_k$ in $G_k/G_{k-1}$ be of finite order. To prove this we can now assume that $g$ is of infinite orde. 

As $\Gamma^k(supp(g))$ is finite and $G_k/C^{ord}(\Gamma^k)$ is a torsion 
group, some power $g^p$ is a non-trivial pei-isometry of rank $k$ which fixes the germ $\gamma(L)$ of some rank-k orthant $L$. Hence $g^p$ parallel shifts t $L$ to an orthant $Lg^p\neq L$ commensurable to $L$. Then one finds inside $L$ rank-$(k-1)$ othants $K\subset L$ with the property that 
$K, Kg^p, Kg^{2p}, Kg^{3p} ... $ 
are sequences of $length \geq 3$ of pairwise 
disjoint parallel orthants. As in the proof of 
Lemma ~\ref{lemma4.17} iii), Case 1. We find in $gp_G(g)$ an orthant-3-cycle of rank $k-1$. Since $\rvert\Gamma^{k-1}\lvert=\infty$ assertion v) 
of Lemma~ \ref{lemma4.17} applies in rank $k-1$ and yields  ${\rm alt}G_{k-1}\leqslant gp_G(g)$.
\end{proof}

We use the Bottleneck Theorem to recover the rank 
of elements $g\in G=pei(S)$ as a group theoretic 
property. For this we introduce the 
\emph{translation-rank} of the elements $g\in G$,  by putting 
$$trk(g):= min_{0\le k\leq rk(g)}\{k\mid G_k\cap gp(g)\neq 1\}.$$
Note that $trk(g)\leq rk(g)$, and $trk(g)=0 $ if and only if $g$ is a torsion element. And if $g$ is torsion-free then $trk(g)$ is the maximal $t\in\mathbb N$ with the property that supp$(g)$ contains a rank-$t$ orthant $L$ on which the restriction $g^p\mid_L:L\rightarrow S$ of some $g^p\in gp(g)$ is also induced by a non-trivial 
translation $\vartheta:\langle L\rangle\rightarrow\langle L\rangle$. 

Now we put $t=trk(g)$, choose a generator $g^q$ of $G_k\cap gp(g)$ and consider the sequence 
of normal subgroups $N_p:=gp_G(g^p)$, $p\in \mathbb N$. As $t\leq rk(g^p)$ for all $p\in\mathbb N$ we 
know from the  Bottleneck Theorem that ${\rm alt}G_{t-1}\leqslant N_p$ for all $p\in\mathbb N$. On the other hand, as $G_t/C^{ord}(\Gamma^t)$ is a 
torsion group and $g^q\in G_t$ we know some power $g^p$ of $g$ is contained in $C^{ord}(\Gamma^t)$. Since $C^{ord}(\Gamma^t)$ is a characteristic subgroup of $G$ it follows that $N_p=gp_G(g^p)$ is also contained in $C^{ord}(\Gamma^t)$. 

Putting things together we thus fond that the intersection of all normal subgroups $N_p$ -- a quantity which depends only upon the group structure 
of $G$ -- satisfies

Now we can use that $C^{ord}(\Gamma^t)/G_{t-1}$ is free-Abelian and hence 
residually finit: the subgroup $N_p/G_{t-1}$ is generated by the $G$-translates of $g^	p/G_{t-1}$, hence for each $i\in\mathbb N$, we have $N_{pi}/G_{t-1}=gp_G(g^{pi}/G_{t-1})=(N_{p}/G_{t-1})^i$, whence $\bigcap_{i\in \mathbb N}N_{pi}\leqslant G_{t-1}$.

Putting things together we find that the intersection of all normal subgroups $N_p$ -- a quantity that depends only on the group structure of $G$ -- satisfies, for each $g\in G$,
$${\rm alt}G_{trk(g)-1}\leqslant~\bigcap_{p\in\mathbb N}gp_G(g^p)~\leqslant~G_{trk(g)-1}.$$
Since ${\rm alt}G_{t-1}$ and $G_{t-1}$ uniquely determine one another this shows that they are characterized in terms of the group structure of $G=pei(S)$.\\
 Summarizing we have

\begin{corollary}\label{corollary4.19} $G=pei(S)$ satisfies the maximal 
condition for normal subgroups. \\
The rank-groups $G_k$ are uniquely determined by the group structure of $pei(S)$. In particular, the poly-(Abelian-by-locally-finite) length of $G$ which is equal to rkS+1, is an invariant of the group structure.\hfill$\square$
\end{corollary}

\textit{Exercise}: In \cite{os02} and \cite{ww15} Osin and Wesolek-Williams define fine-meshed (ordinal valued) ranks which measure the complexity of elementary amenable groups $G$. Use the Bottleneck Theorem to compute compute these ranks for $G=$pei$(S)$.

\chapter{The Euclidean case II: The finiteness length}\label{The Finiteness Length}
\section{A lower bound for the finiteness length of pei(S)}
In this section we will define a certain ``diagonal'' subgroup, ${{\rm pei}}_{{\rm {\rm dia}}}(S) \leq {{\rm pei}}(S)$, and prove
\begin{theorem}\label{theorem5.1}  For every orthohedral set $S$ we have 
$$ fl \big( {\rm pei}  (S) \big) \ge 
fl \big( {\rm pei}_{\emph{dia}} (S) \big) = h(S) - 1.$$
\end{theorem}
We follow the strategy of Brown's proof in the influental paper [Br87] which covers the case when S is a stack of rays; and we also take full advantage of the technical results and the insight provided by the second named author's Diploma Thesis \cite{sa92}.

\subsection{The height of a pei-injection \texorpdfstring{$\mathbf{f:S\to 
S}$}{f:S\rightarrow S}.}
We start with a general observation on the set of germs, when an orthohedral set $S$ comes with
a decomposition of a disjoint union  $S = A\cup B$ of two orthohedral subsets.
In that case every orthant $L\subseteq S$ inherits the decomposition  $L = 
(A\cap L)\cup (B\cap L)$, which shows that one of the orthants of either $A$ or $B$ is commensurable to $L$. 
This shows that the germs of S have an induced disjoint decomposition $\Gamma^k(S)  =  \Gamma^k(A)\cup \Gamma^k(B)$ for
each  $k.$

Now let $S$ be an orthohedral set of rank ${{\rm rk}}S = n.$ We can represent the
rank-$n$ germs of $S$ by pairwise disjoint orthants $L_1, \ldots , L_h, h 
=
h(S)$, with the property that the restriction of $f$ to each $L_i$ is an
isometric embedding into $S$. $f(L_i)$ is then commensurable to some $L_j$, and
since $f$ is injective it follows:
\\ $f$ permutes the germs  $\gamma (L_1),\ldots , \gamma (L_p)$, and $ {{\rm rk}}(S - f(S)) < {{\rm rk}}S.$
\\
As $S - f(S)$ is an orthohedral set, we now obtain  that the number of
rank-$(n-1)$ germs in $S - f(S)$ is finite. We call this number \emph{the
height of} $f$, denoted $h(f) = h\big( S - f(S)\big) = h(S - Sf).$

\begin{lemma}\label{lemma5.2}
\begin{enumerate}[i)]
\item If $g, f: S\to S$ are two pei-injections, then $h(gf) = h(g) + h(f).$
\item If  $A\subseteq S$ is an orthohedral subset whose complement $A^c = 
S - A
$ has rank ${\emph{rk}}A^c < n = {\emph{rk}}S$, then the height of any pei-injection $f: S\to S$
is given by  $h(f) = $ \mbox{$h\big( A\cap f(A)^c \big)$} $-$ $h\big( A^c\cap
f(A)\big).$
\end{enumerate}
\end{lemma}

\begin{proof}
i) Consider the disjoint union $S = (S - Sg)\cup Sg$. As $f$ is injective $Sf=$ \mbox{$(Sf - Sgf)$} $\cup~Sgf$ is also a disjoint union. Hence 
so is

$$S = (S - Sf)\cup Sf = (S - Sf)\cup(Sf -Sgf)\cup Sgf,$$ 
and we find \[	S - Sgf = (S - Sf)\cup (Sf - Sgf). \] Now, $f$ is a pei-bijection
between $(S - Sg)$ and $(S - Sg)f =  (Sf - Sgf)$; and a pei-bijection of a an orthohedral set induces a pei-bijection on its germs. Thus the number of rank-$(n-1)$ germs of $(S - Sg)$ and $(Sf - Sgf)$ are the same. This proves i).

ii) Each pei-injection $f: S\to S$ induces an injection $f^\ast:
\Gamma^{n-1}(S)\to \Gamma^{n-1}(S)$. We abbreviate $B = A^c$ and know from ${{\rm rk}}B < n$  that $\Gamma^{n-1}(B)$ is finite. Hence $f^\ast$ restricts to a bijection  $f^\ast: \Gamma^{n-1}(B)\to \Gamma^{n-1}\big(f(B)\big)$. On the complement we find the induced injection $f^\ast: \Gamma^{n-1}(A) \to  \Gamma^{n-1}\big(f(A)\big).$
\\We use the abbreviation $P^\ast:=\Gamma^{n-1}(P) $ for $P = S, A, B$, and
consider the disjoint union
$$\begin{aligned}
	S^\ast - f^\ast(S^\ast) 	=& \big(A^\ast - A^\ast\cap f^\ast(S^\ast)\big) \cup  \big(B^\ast - B^\ast\cap f^\ast(S^\ast)\big)\\
			=  & \big( A^\ast - A^\ast\cap f^\ast(A^\ast) - A^\ast\cap f^\ast(B^\ast) \big) \\ & \cup \big( B^\ast - B^\ast\cap f^\ast(A^\ast) - B^\ast\cap f^\ast(B^\ast) \big).
\end{aligned}$$
By definition,  $h(f) = $ \mbox{$h\big( S^\ast - f^\ast(S^\ast)\big)$}. 
Using the fact that $B^\ast$ as well as $A^\ast - f^\ast(A^\ast)$ are finite, we find
$$\begin{aligned}  h(f)  = &  h \big(
A^\ast- A^\ast\cap f^\ast(A^\ast)\big) - h \big( A^\ast\cap f^\ast(B^\ast) \big) \\ & + h(B^\ast) - h \big( B^\ast\cap f^\ast(A^\ast)\big) - h \big( B^\ast\cap f^\ast(B^\ast)\big).\end{aligned}$$
Now we apply that $h(B^\ast) = h\big( f(B^\ast) \big)$ and observe that
\[ 
\begin{array}{lcl}
- h \big( A^\ast \cap f^\ast(B^\ast)\big) &+& h(B^\ast) - h \big( B^\ast\cap f^\ast(B^\ast)\big)\\
 &=& - h\big( A^\ast\cap f^\ast(B^\ast)\big) + h\big( f(B^\ast)\big) - h \big( B^\ast\cap f^\ast(B^\ast)\big)\\
 &=& h\big( f(B^\ast) - h\big((A^\ast\cap f^\ast(B^\ast) \big)\cup\big( 
B^\ast\cap f^\ast(B^\ast)\big)\big) = 0.
\end{array}
\]

Hence our expression for $h(f)$ simplifies to
$$\begin{aligned}
h(f)  = &  h \big( A^\ast- A^\ast\cap f^\ast(A^\ast) \big)  - h \big( B^\ast\cap f^\ast(A^\ast) \big)\\ = &  h \big( A^\ast\cap f^\ast(A^\ast)^c\big)  - h \big(B^\ast\cap f^\ast(A^\ast)\big)\end{aligned}$$
as asserted.
\end{proof}

\subsection{Monoids of pei-injections.}\label{Monoids of pei-injections}
Let $S$ be an orthohedral set of rank $n = {\text{rk}}S$ in pet-normal form.  In
particular $S$ is the pairwise disjoint union of finitely many specified stacks of orthants. By Lemma \ref{lemma3.2} the set of all maximal germs of S, max$\Gamma^\ast(S)$, is finite. We write  $M(S)$ for the monoid of all pei-injections $S \to S$. It is endowed with the height function  $h: 
M(S)\to
\mathbb N$ of Section 3.1. Let $M_0(S)$ be the submonoid of all pei-endoinjections of $S$, which fix all maximal germs of $S.~ M_0(S)$
is of  finite index in $M(S)$ since max$\Gamma ^\ast(S)$ is finite.  Just 
as we have observed for pei-permutations, each $f\in M_0$ induces an isometry $\tau_{(f,\gamma)}: \langle \gamma \rangle \to \langle \gamma \rangle$ on the tangent coset of each germ $\gamma\in \max\Gamma^\ast(S).$ Thus 
we  have a homomorphism

\begin{itemize}
\item[(5.1)] ~~~~ $\kappa : M_0(S)\twoheadrightarrow\bigoplus_{\gamma\in \max\Gamma^\ast(S)} {\rm Iso} \left(\langle \gamma\rangle \right)$,       
       ~~~~~   given by
$$\kappa (f) =
\bigoplus\limits_{\gamma\in \max\Gamma^\ast(S)} \tau_{(f,\gamma)}.$$
\end{itemize}

The \emph{translation submonoid} $M_{tr}(S)\subseteq M_0(S)$ consists of all
$f\in M_0(S)$ with the property that the induced maps $\tau_{(f,\gamma)}: 
\langle \gamma \rangle \to \langle \gamma \rangle$
are translations for each $\gamma\in \max\Gamma^\ast(S)$.
Since the translation subgroup of  Isom$(\langle \gamma \rangle)$ is of finite index, $M_{tr}(S)$ has finite index in
$M_0(S)$. And restricting (5.1) yields a surjective homomorphism
 
\begin{itemize}
\item[(5.2)]\label{item(5.2)}
\begin{equation}\kappa : M_{\rm tr}(S)\twoheadrightarrow \bigoplus_{\gamma\in\max\Gamma^\ast(S)} \mathbb{Z}^{{\rm rk}(\gamma)} = \mathbb{Z}^N,
\end{equation} with $N = \Sigma_{\gamma\in \max\Gamma^\ast(S)} {\rm rk}\gamma.$
\end{itemize}

Every orthant $L$ contains a characteristic \emph{diagonal element} $u_L\in
L$: the sum of the canonical basis of $L$. We write $t_L: L\to L$ for the
translation given by addition of $u_L$ and call this the \emph{diagonal
unit-translation} of $L$. The general diagonal translations on $L$ (i.e., 
on $\langle L \rangle$) are given by addition of
an integral multiple  of $u_L$. By the \emph{diagonal submonoid}  $M_{{\rm dia}}(S)\subseteq M_{\rm tr}(S)$ we mean the set of all
elements $f\in M_{\rm tr}(S)$ with the property, that for each $\gamma \in \max\Gamma^\ast(S)$ the induced isometry $\tau_{(f,\gamma)}:
\langle \gamma \rangle \to \langle \gamma \rangle$ is a diagonal translation. Restricting \ref{item(5.2)} yields the homomorphism

\begin{itemize}
\item[(5.3)]\label{item(5.3)} $\kappa : M_{{\rm dia}}(S)\twoheadrightarrow\bigoplus_{\gamma\in \max\Gamma^\ast(S)} \mathbb Z
 = \mathbb Z^{\mid \max\Gamma^\ast(S) \mid}$.
 \end{itemize}
 
 We write max$\Omega^*(S)$ for the set of all maximal orthants of the stacks of S, and consider the set $T = \{t_L\mid
 L\in \max\Omega^\ast(S)\}$ of all diagonal unit-translations of these orthants. Each $t_L \in T$ extends canonically to a pei-injection on \mbox{$t_L : S
 \to S$}, which is the identity on $S - L$.
 We denote it by the same symbol $t_L $, and with this interpretation $T$ 
generates a free-Abelian submonoid mon$(T)\le M_{{\rm dia}} (S)$. 

Following the strategy of \cite{br87} we put
\begin{definition}
Given $f, f '\in M_{{\rm dia}} (S)$ we define $f\le f '$ if there is some 
$t\in \rm {mon}(T)$ with $tf = f'$.
\end{definition}
\textbf{Observation.} $M_{{\rm dia}} (S)$ is a \emph{directed partially ordered set}.

It is an important fact that the height function $h: M(S) \to  \mathbb N$ 
 is
order preserving and its restrictions to totally ordered subsets of $M(S)$ are
injective.  We will also have to consider slices of  $M_{{\rm dia}} (S)$. 
For
given $r,s \in \mathbb N_0$, $r \le s$ we put 
\begin{equation*}
\begin{aligned}
M^{[r,s]} &:= \{f\in M_{{\rm dia}} (S)\mid r \le h(f) \le s\} \text{ and}\nonumber\\
M^{[r,\infty]} &:= \{f\in M_{{\rm dia}} (S)\mid r \le h(f)\}.\nonumber
\end{aligned}
\end{equation*}
$M^{[r,\infty]}$ inherits the partial ordering from $M_{{\rm dia}} (S)$ and is
also a directed set.

\subsection{Maximal elements \texorpdfstring{$\mathbf{< f}$}{< f} in \texorpdfstring{$\mathbf{M_{{\rm dia}} (S)}$}{TEXT}.}\label{maximal elements}~
From now on we assume that all maximal orthants of the stacks of $S$ have the same finite rank
$n = {\rm rk}S$. We put $\Lambda := \max\Omega^\ast(S)$. We write
\[	M_{< f} = \{a\in M_{{\rm dia}} (S)\mid a< f\}, \quad M_{\le f} = \{a\in M_{{\rm dia}} (S)\mid a\le f\}  \] 
for the \emph{``open resp. closed cones below $f$''} and aim to understand the set of all maximal elements
of $M_{< f}$ .
For this it will be convenient to introduce an abbreviation for the points on the (finite)
boundary of the maximal orthants $L$; so we set \mbox{$\partial L := L - Lt_L$}.

\begin{lemma}\label{lemma5.4}
Let $b$ be a maximal element of  $M_{< f}$. Then there is a unique maximal
orthant $L \in \Lambda$ with the property that $f = t_L b$ and $h(f) = 
h(b) + n$.
Furthermore $b$ is given as the union $b = b'\cup b''$, where $b': \partial L \to (S - Sf)$ is a pei-injection, and $b'':
(S - \partial L) \to Sf$ is the restriction $(t_L^{ -1}f)\mid_{ (S - \partial L)}$.
Converseley, if $c': \partial L \to (S -  Sf)$ is an arbitrary pei-injection distinct from
$b'$, then the union $c = c'\cup b''$ is a maximal element of $M_{< f}$
distinct from $b$.
\end{lemma}

\begin{proof}
For each element $b\in M_{< f}$ there is some $t\in {\rm mon}(T)$ with $f 
= tb.~ t$ has a unique reduced expansion as a product of elements 
of $T$; let $l(t)$ denote the length of this expansion. As $h(t_L ) = L 
= n$ for each $L\in \Lambda$  we have $h(t) = n~l(t)$. It follows that if 
$b$ is maximal, then $h(t) = n$ and $t = t_L \in T$ for some $L\in \Lambda$.
The maximal orthant $L$ is uniquely determined by the fact that the restriction of $f$ and $b$ coincide on $S - L.$ 
The restriction $b''$ of $b$ to $(S - \partial L)$ coincides with $(t_L^{
-1}f)\mid_{ (S - \partial L)}$, and has its image in Sf. The restriction $b'$ of $b$ to  $\partial L$ is not determined by $f$ and $L.$ As $b$ and 
$f$ are injective we know that \[ \begin{array}{rcl} \emptyset  = (\partial L)b\cap (S - \partial L)b &=& (\partial L)b\cap \big( (S - L)b\cup 
Lt_L b\big)\\
				& =& (\partial L)b\cap \big( (S - L)f\cup Lf\big)\\ 
			&	 = &(\partial L)b\cap Sf.
\end{array} \]

Hence $b'$ can be viewed as a pei-injection $\partial L\to (S - Sf)$.
If we replace $b'$ by another pei-map $c': \partial L\to (S - Sf)$, the union
$c = c'\cup b''$ will still satisfy  $f = t_L c$ and $h(f) = h(c) + 
n$. This shows that $c$ will also be maximal in $M_{< f} .$
\end{proof}

\begin{lemma}\label{lemma5.5}
Let $B\subseteq M_{< f}$ be a finite set of maximal elements of $M_{< f}$. Then the following conditions are equivalent:
\begin{enumerate}[i)]
\item   The elements of $B$ have a common lower bound $\delta$ in $M_{< f}$
\item For every pair $(b,b')\in B\times B$ with $b\not= b'$ and $tb=f 
=t'b'$ for diagonal unit-translations $t$, $t'$, we
have
\begin{enumerate}[a)]
\item  $t\not= t'$  and 
\item 	  $b(\partial L)\cap b'(\partial L') = \emptyset$, where $L$ resp. $L'$ are the maximal orthants of $S$ on which
$t$ resp. $t'$ acts non-trivially.
\end{enumerate}
\end{enumerate}
\end{lemma}

\begin{proof}
i) $\Rightarrow$ ii). Let $\delta$ be a common lower bound of the elements
of $B$. Then for every pair $(b,b')\in B\times B$ there are diagonal translations $d, d'\in {{\rm mon}}(T)$ with $d\delta
= b$ and $d'\delta = b'$.
From $tb = f = t'b'$ we obtain  $td\delta =  t'd'\delta$ and  conclude $td = t'd'$.
The assumption $t = t'$ would now imply $d = d'$ and hence $b = b'$.
\\
Let $L$ resp. $L'$ denote the maximal orthants of $S$ on which $t$ resp. $t'$ acts non-trivially. As $d, d'$ are diagonal translations, we have 
$d(L)\subseteq L$ and $d'(L')\subseteq L'$. From $t \not=  t'$ we know $L\cap L' = \emptyset$ , and hence $(\partial L)d\cap (\partial L')d' = 
\emptyset$. 
Since $\delta$ is injective, this implies  $\emptyset = (\partial L)d\delta\cap (\partial L')d'\delta = (\partial L)b\cap (\partial L')b'$, as 
asserted.

ii)$\Leftarrow$ i). For each $b\in B$ we have some diagonal unit-translation $t_b\in T$ with $t_b b = f$, and we put

\begin{itemize}
\item[(5.4)] $t_B := \prod_{ b\in B}t_b$.
\end{itemize}

By assumption (i) the maximal orthants $L_b $ on which $t_b $ is a diagonal unit-translation are pairwise disjoint. 
Thus $|B|$ $\le h(S)$, and $S$ decomposes in the disjoint union  $S  =
\left(\bigcup_{b\in B}L_b \right) \cup  S'$.
We define the pei-injection $\delta_B: S\to S$ as follows:
$$ \delta_B := \begin{cases}
{t_b}^{-1} f & \text{ on each } L_b t_b \\ 
b & \text{ on the complements } \partial L_b = L_b - L_b t_b \\ 
f & \text{ on } S'. 
\end{cases} $$

Assumption (ii) guarantees that the restriction of $\delta_B $ to the union 
$$\bigcup\limits_{ b\in B}\partial L_b  = \bigcup_{ b\in B} (L_b  - L_b 
t_b ) = (S - St_B)$$ is injective.  And since the image of each $\partial L_b$ is disjoint to
$f(S)$, we also find that the image of $(S - St_B)$ is disjoint to $f(S)$, and also to $f(S')\subseteq f(S)$. This
shows that $\delta$  is a pei-injection. It remains to prove that $\delta_B $ is a common lower bound for the elements of $B$.
By commutativity we find elements $s_b\in ~{\rm mon}~(T)$ with $t_B \delta_B  = t_b s_b\delta_B $, where 

\begin{itemize}
\item[(5.5)]\label{(5.5)} $s_b = \prod_ { x\in (B - \{b\})}t_x$.
\end{itemize}

One observes that $s_b\delta_B $ and $b$ agree on $(S- St_b) = (L_b  - L_b t_b)$, and that $t_B\delta_B = f = t_b b$. 
Hence $b$ and $s_b\delta_B $ agree on $S$. 
\end{proof}

\begin{lemma}\label{lemma5.6}
In the situation of Lemma \ref{lemma5.5} we have for the lower bound $\delta_B$ defined in the proof:
\begin{enumerate}[i)]
\item $\delta_B$  is, in fact, a largest common lower bound of the elements of $B$
\item $h(\delta_B ) \ge h(f) - h(S)n$.
\end{enumerate}
\end{lemma}
\begin{proof}
i) We compare an arbitrary common lower bound $\gamma$ with
$\delta_B$, the lower bound constructed in the proof above. Thus for each 
$b\in
B$ we are given $u_b\in {\rm mon} (T)$ with $u_b \gamma  = b.$ We fix a 
base element $b'\in B$ and define the diagonal translation 
$t'\in {\rm mon} (T)$ by its action on $S$ as $$ t' := \begin{cases} u_{b'} & \text{ on } S' \\ 
u_b & \text{ on each } L_b.
\end{cases} $$
We use the elements $s_b$ of \ref{(5.5)} and observe that

\begin{tabular}{r l l}
	& $xt'\gamma  = xu_{b'} \gamma  = xb' = xs_{b'}\delta_B  = x\delta_B$ & for
	$x\in S'$\\
	& $xt'\gamma  = xu_b \gamma  = xb = xs_b\delta_B  = x\delta_B$ & 
for $x\in L_b$.
\end{tabular}

This shows that  $t'\gamma  = \delta_B$ , hence  $\delta_B  \ge \gamma$.
\\
ii) For the translation $t$ of (3.4) we have $t\delta_B  = f$  and can deduce
that $h(\delta_B ) =$ \mbox{$h(f) -  h(t)$} $=$ \mbox{$h(f) -  |B| n$} $\geq$ $h(f) - h(S)n.$
\end{proof}

\subsection{The simplicial complex of \texorpdfstring{$\mathbf{M_{\mathbf{{\rm dia}}(S)}}$}{TEXT}.} \label{subsection5.4}
We consider the simplicial complex $|M_{{\rm dia}}(S)|$,  whose vertices are the elements of $M_{{\rm dia}} (S)$ and whose chains of length $k,  
a_0 < a_1 < \ldots < a_k$, are the $k$-simplices. As the partial ordering 
on
$M_{{\rm dia}} (S)$ is directed, $|M_{{\rm dia}}(S)|$ is contractible.

In this section we aim to prove

\begin{lemma}\label{lemma5.7}
If $h(f)\ge 2\cdot \emph{rk}S\cdot h(S)$, then $|M_{< f}|$  has the homotopy type of a bouquet of $\big( h(S) - 1
\big)$-spheres.
\end{lemma}

The first step towards proving Lemma \ref{lemma5.7} is to consider the covering of  $|M_{< f}|$  by the subcomplexes $|M_{\le b}|$ , where $b$ runs 
through the maximal elements of  $M_{< f}$. We write $N(f)$ for the nerve 
of this covering. Lemma \ref{lemma5.5} and
\ref{lemma5.6} show that all finite intersections of such subcomplexes $|M_{\le b}|$ are again cones and hence
contractible. It is a well known fact that in this situation the space is 
homotopy equivalent to the nerve of the covering. Hence we have
$$|M_{<f}|  ~\text{is homotopy equivalent to the nerve}~ N(f)$$
and it remains
to compute the homotopy type of $N(f)$.

The next step is to use the results of Section \ref{maximal elements} to find a combinatorial model
for the nerve $N(f)$. The vertices of $N(f)$ are the maximal elements of $M_{<f}
$, and hence, by Lemma \ref{lemma5.4}, in 1-1-correspondence to the disjoint
union  $A = \bigcup_{L\in \Lambda} A_L $, where $A_L $ stands for the set of all pei-injections  $a: \partial L\to (S - Sf ).$ 
Lemma \ref{lemma5.4} allows to translate the simplicial structure of $N(f)$ into a simplicial complex $\Sigma(f)$ on $A$:
\\
the $p$-simplices of $N(f)$ are the $p$-element sets of maximal elements $B\subseteq M_{<f} $ with a common lower bound, and the corresponding
p-simplices of $\Sigma(f)$ are the sequences $(a_L)_{L\in \Lambda '}$, where $\Lambda '$ is a $p$-element subset of $\Lambda$  and $a_L\in A_L $ with the property
$$(\ast)\text{ \textit{The intersections of the images} }a_L(\partial L), 
L\in \Lambda '\text{\textit{, are pairwise disjoint.}}$$

The next Lemma \ref{lemma5.8} below on coloured graphs will 
enable us to determine the homotopy type of $
\Sigma(f)$. This is a natural generalization of the 
second author's (Heike Sach's) Lemma 4.7 in
\cite{sa92}, where it served as a major technical 
key in extending computation of $fl$(pet($S$)) from the 
case when $S$ is a stack of rays (the Houghton 
Group result of \cite{br87}, to the case when $S$ 
is a stack of quadrants. Sach's Lemma and its proof 
were based on but are in parts rather different 
from Brown's Lemmas 5.2/5.3. Ken Brown, in turn, 
remarks that the inductive proof of his Lemma 5.3 
uses ''a method due to K. Vogtmann [private 
communication]''.

Let $\Gamma  = (V,E)$ be a combinatorial graph, given by a set $V$ of \emph{vertices} and set $E$ of \emph{edges},
where an edge is a set consisting of two non-equal vertices. A \emph{clique} of $\Gamma$  is any subset $C\subseteq V$ with the property, that any
two vertices of $C$ are joined by an edge of $\Gamma$. The \emph{flag-complex} $K(\Gamma )$ is the simplicial complex on $V$ whose
$p$-simplices are the cliques consisting of $p+1$ vertices of $V$.  Our main example here is the complex $\Sigma(f)$, which is easily seen to be the 
flag-complex of its $1$-skeleton $\Gamma (f)$.

Let $h$ be a natural number. We say that the graph $\Gamma_h=(V,E)$ is \emph{$h$-colored} if its vertex set $V$ is the pairwise disjoint union of $h$ subsets $V_1, 
\ldots V_h$ (where the index $i$ is the colour of the vertices in $V_i$), 
and no edge has endpoints with the same color. 
\begin{lemma}\label{lemma5.8}
If all colours $i\in \{1, 2, \ldots , h\}$ of an $h$-coloured graph $\Gamma_h = (V,E)$ satisfy the two properties
\begin{enumerate}[i)]
\item[\emph{(1)}] $V_i$ contains at least two distinct elements, and
\item[\emph{(2)}]  For any choice of $2(h - 1)$ vertices $u^1, \ldots u^{2(h - 1)}$ in $V - V_i $ there are two vertices $v, w$ in 
$V_i $ which are adjacent to each $u^j$; in other words, for each $j\in \{1, 2, \ldots 2(h - 1)\}$ there is an edge path of length
$2$ in $\Gamma_h$ joining $v$ and $w$ via $u^j$.
\end{enumerate}
Then the flag-complex $K(\Gamma_h)$ has the homotopy type of a bouquet of 
$(h - 1)$-spheres.
\end{lemma}
\begin{remark}
Note that (2) holds vacously if h = 1; and if $h>1$ then (1) actually follows from (2).
\end{remark}
\begin{proof}
We use induction on $h$, starting with the observation that the statement 
is trivial when $h=1$. For $h\ge 2$ we assume that $K(\Gamma_{ h-1})$ is homotopy equivalent to a bouquet of  $(h - 2)$-spheres, if $\Gamma_{h-1}$ is an $(h - 1)$-coloured graph which satisfies the properties (1) and (2). We construct $K(\Gamma_h)$  in several steps, similar to the method applied in Brown's proof for Hougthon's groups \cite{br87}. We start with 
choosing a base vertex $v_1\in V_1$ and consider its star in $K(\Gamma_h)$, 
\[ K_0  :=  {\rm st}_{K(\Gamma h)}(v_1)\]
Then we proceed with $i = 1, 2, \ldots , h$ by taking the union of $K_{i-1} \cup V_i '$, where $V_i '$ is the set of all vertices of $V_i$  which are not joined with the base vertex $v_1$ by an edge. And we put  
\[ K_i := ~\text{ full subcomplex of}~ K(\Gamma_h) ~\text{ generated by}~ K_{i-1} \cup V_i '.\]

One observes that $(V_1 \cup  \ldots \cup V_i )\subseteq K_i $ and  $(V_{i +1}\cup  \ldots \cup V_h)\cap K_i  = K_0$. In particular, $K_h = K(\Gamma_h).~ K_i $ is obtained from $K_{i -1}$ by adjoining vertices $v\in 
V_i$   that are not connected to the base vertex $v_1$ by an edge; then taking the full subcomplex of $K(\Gamma_h)$. Thus $K_i $ is obtained from $K_{i -1}$ by adjoining for these vertices $v$ the cone over
\[ {\rm lk}(K_{i -1} ,v) := ~\text{ the link of}~ v ~\text{ in}~ K_{i -1}. \]

The $1$-skeleton of ${\rm lk} (K_{i -1} ,v)$ has vertex set
\[ \begin{array}{rcl}
		W &=& W_1\cup  \ldots \cup W_{i-1}\cup W_{i+1}\cup  \ldots \cup W_h ~\text{ with }\\ 
		W_j &:=& \text{set of vertices of}~ V_j  ~\text{which are joined with}~ v~\text{ by an edge},\\
&& \text{ for}~  j =  1, \ldots , i-1\\
		W_j &:=& \text{ set of vertices of}~ V_j  ~\text{ which are joined with}~ v~\text{  and}~ v_1~\text{ by an}\\
&& \text{ edge, for}~j =  i+1, \ldots , h.
\end{array}\]

Thus, the $1$-skeleton of ${\rm lk}(K_{i-1} ,v)$  is an $(h - 1)$-coloured subgraph $\Gamma_{h-1}$ of $\Gamma_h$ with vertex set $W$ and colours 
$\{ 1, 2, \ldots , h\} -\{ i\} $, and ${\rm lk}(K_{i-1} ,v)$  is the flag-complex $K(\Gamma_{h-1})$. Now we consider any $2(h -2)$ vertices 
$u^1 , \ldots, u^{2(h - 2)}$  of  $W - W_j$   with colours in $\{ 1, 2, \ldots , h\} -\{ i, j\} $ for some $j \in \{ 1, 2, \ldots , h\} -\{ i\} $. 

Together with the vertices $v_1$ and $v$, we obtain $2(h - 1)$ vertices
$$u^1 , \ldots , u^{2(h - 2)}, v_1, v \text{ of }  V - V_j.$$
By property (2) of 
$\Gamma_h$, there exists two vertices $w, w'$ in $V_j$ , which can be joined by an edge path of length $2$ via $u^k$ for each $k \in \{ 1, 2, \ldots , 2(h - 2)\} $, and additionally via $v_1, v$. In particular, $w$ and $w'$ can be joined by an edge with $v_1$ and $v$, and so they are vertices of  $W_j$.  
Hence $\Gamma_{h-1}$ satisfies the two properties of the lemma, and in view of the inductive hypothesis, ${\rm lk}(K_{i-1} ,v)$  is homotopy equivalent to a 
bouquet of  $(h - 2)$-spheres.

From here we can use the same arguments as in the proof of Lemma 5.3 in 
\cite{br87}: Starting with the contractible complex $K_0,  K_1$  is obtained
from $K_0$ by adjoining for each vertex  $v\in V_1 '$ a cone over ${\rm lk}(K_0 ,v)$.  Using the homotopy type of ${\rm lk}(K_0 ,v)$, we can deduce that $K_1$ is homotopy equivalent to a bouquet of $(h - 1)$-spheres. For the next steps in the construction of $K_h$, we know that $K_i$
is obtained from $K_{i-1}$ by adjoining for each vertex $v\in V_i '$ a cone over ${\rm lk}(K_{i-1} ,v)$.  In view of the homotopy type of ${\rm lk}(K_{i-1} ,v)$,  we see that, up
to homotopy, the passage from $K_{i-1}$ to $K_i$   consists of the adjunction of $(h - 1)$-cells to a bouquet of $(h - 1)$-spheres.
\end{proof}

We will now apply Lemma \ref{lemma5.8} to the $1$-skeleton  $\Gamma (f)$ of $\Sigma(f)$.
By definition its vertex set is the disjoint union $A = \bigcup_{ L\in \Lambda}
A_L$, and we regard the various $A_L$  as the colouring of  $\Gamma (f)$. 
The edges of $\Gamma$  are the pairs of such pei-injections $\{ a_L, a'_{L'}\}$  with disjoint images.
Thus $\Gamma (f)$ is an $h(S)$-colored graph $\Gamma (f)_{h(S)}$ in the sense above, and in order to establish Lemma \ref{lemma5.7} it remains to prove

\begin{lemma}\label{lemma5.9}
If  $h(f) \ge 2\cdot \emph{rk}S \cdot h(S)$, then $\Gamma (f)_{h(S)}$ satisfies the assumptions of Lemma \ref{lemma5.8}.
\end{lemma}
\begin{proof}
Let $n := S$ and $h := h(S)$. By the remark following Lemma 3.8 we can assume $h>1$ and have to prove (2). For this we fix $L\in \Lambda$  and 
consider a set of $2(h - 1)$ elements $F\subseteq A - A_L$ . We have to show that
there are two elements $a, b\in A_L$  with the property that for each $c\in F$ the image im$(c) = c\left( \partial L(c)\right)$ is disjoint to  both $a(\partial L)$ and
$b(\partial L)$. In other words: there are two pei-injections 
\[   a, b:  \partial L  \to  \big(S - Sf \big) - \Big(\bigcup_{ c\in F}{\rm im}(c)\Big). \] 
To show this it suffices to compare the height function - i.e., 
the number of rank-$(n-1)$ germs -  of domain and target. Clearly, $h\left(a(\partial L))\right) = h(\partial L) = n$, and the same applies to 
every 
vertex of $A$. Hence $h\left(\bigcup_{ c\in F}{\rm im}(c)\right) \le 2(h - 1)n.$ By assumption $h\big(S - Sf \big) \ge 2hn$, and so the target orthohedral 
set has height at least $2hn - 2(h - 1)n = 2 n$, which is more than the 
height $h(\partial L) = n$ of the domain. In this  situation one observes easily
that there are arbitrarily many different pei-injections in $A_L$ whose image is
disjoint to  $\bigcup_{ c\in F}{\rm im}(c)$. This proves the lemma.
\end{proof}
\begin{remark}
If we replace $M_{<f}$ by the subset $M_{r,f} :=\{a \in M_{{\rm dia}}(S) \mid h(a) \geq r ~ \text{ and} ~ a < f \}$, the
assertion of Lemma \ref{lemma5.7} true, provided $f$ satisfies the additional condition $h(f) \geq r + h(S)$.
In this case we know by Lemma \ref{lemma5.6} that $h(\delta_B) \geq r$, where $\delta_B$ stands for the largest lower bound of a
finite set $B$ of maximal elements of $M_{r,f}$. Thus $\delta_B$ is an element of $M_{r,f}$ and the proof of Lemma~ \ref{lemma5.6}
works the same way for the reduced simplicial complex $|M_{r,f}|$.
\end{remark}

\subsection{Stabilizers and cocompact skeletons of M(S)}
 The group $G(S)$ of all pei-permutations acts on $M(S)$ from the right, and as $h(g) = 0$ for all $g\in G(S)$ the height 
 function $h: M(S) \to  \mathbb N$  is invariant under this action. Correspondingly, $G_{\#}(S) := G(S)\cap M_{\#} (S)$ acts on 
 $M_{\#} (S)$, where $\# $ stands for $0,$ tr, or dia.  We will also restrict attention to the various $G_{\#} (S)$-invariant 
 subsets  $M^{[r,s]}_{\#}(S)= \{ f\in M_{\#} (S)\mid r\le h(f)\le s\} $ 
for prescribed numbers $r \le s$ in $\mathbb
 N_0$. And also, mutatis mutandis, for the corresponding pet-groups ${\rm 
pet}_{\#}(S)$ - note that ${\rm pet}_0(S) = {\rm pet}_{{\rm tr}}(S)$.

We start with the following easy observation:
\begin{lemma}\label{lemma5.10}
  Two elements $f, f '\in  M(S)$ are in the same \emph{pei}$(S)$-orbit if 
and only if  $(S - Sf )$ and $(S -
Sf ')$ are pei-isomorphic.
\end{lemma}
\begin{proof}
As both $Sf$ and $Sf '$ are pei-isomorphic to $S$ there is always a pei-isomorphism \mbox{$g': Sf \to  Sf '$}.
Assuming there is also a pei-isomorphism $g'': (S - Sf ) \to  (S - Sf ')$ 
implies that the union $g = g'\cup g''$ is a pei-permutation of 
$S$ with $fg = f '$. Conversely, $fg = f '$ implies $(S - Sf ') = (Sg - Sfg) = (S - Sf)g$, hence $(S - Sf ')$ is pei-isomorphic 
to $(S - Sf)$.
\end{proof}

Since orthohedral sets of the same rank and height are pei-isomorphic by Corollary \ref{corollary3.6}, it follows that ${\rm pei}(S)$ acts
transitively on the set of all pei-injections of a given rk$(S - Sf)$ and 
height $k$. The very same can be said for the 
action of  $G_{\#} (S)$ on $M_{\#} (S).$

Let  $\Delta  = (a_0 < a_1 < \ldots < a_{k-1} < a_k)$ be a  $k$-simplex 
of $ |M(S)|$. By definition there are elements 
$t_1, t_2, \ldots , t_k \in {\rm mon} (T)$, with $a_i = t_i a_0$ for all $i$; they are uniquely defined and form a $k$-simplex   
$ \Delta ' = (id<t_1< \ldots <t_{k-1}<t_k) \in |{\rm mon}(T)|$ . Moreover, putting $\sigma (\Delta ) := (\Delta , a_0)$
defines a bijection 
			\[\sigma :~   |M(S)| ~ \longrightarrow ~  |{\rm mon}(T)| \times ~ M(S).\]

The action of ${\rm pei}(S)$ on $ |M(S)|$  is given by $(a_0 <a_1 < \ldots <a_k )g = (a_0 g< a_1 g< \ldots < a_{k-1} g< a_k g)$.
We can leave it to the reader to observe that this action induces, via $\sigma$, on $ |{\rm mon} (T)| \times M(S)$ the
$G(S)$-action given by simple right action on $M(S).$

The simple structure of the $G_{\#} (S)$-action on $ |M_{\#} (S)|$  has two immediate consequences:

\begin{corollary}\label{corollary5.11} \begin{enumerate}[i)]
\item The stabilizer of a $k$-simplex of $ |M_{\#} (S)|$  coincides with the
stabilizer of its minimal vertex $f$ and is isomorphic to $G_{\#} (S - Sf).$
\item  For every numbers $r \le s$ in $\mathbb N\cup\{0\}$ the simplicial 
complex of
$$M^{[r,s]}_{\#}(S)= \{ f\in
M_{\#} (S)\mid r\le h(f)\le s\}$$
is cocompact under the $G_{\#} (S)$-action.
\end{enumerate} \end{corollary}

\begin{proof}
i) One observes that right action of $g\in G_{\#} (S)$ on $M_{\#} (S)$ fixes an element $f\in M_{\#} (S)$
if and only if $g$ restricted to $Sf$ is the identity. In other words, the stabilizer of the vertex $f\in M_{\#} (S)$ is isomorphic to $G_{\#} (S - Sf).$
\\
ii) We use the interpretation of a simplex $\Delta  = (a_0 < a_1 < \ldots < a_{k-1} < a_k) \in  |M(S)| $ in $|{\rm mon} (T)| \times~ M(S)$. 
Since $G_{\#} (S)$ acts transitively on the set of all pei-injections in $M_{\#} (S)$ of a given
rk$(S-Sf)$ and height $k$, the bound on $h(a_0)$ allows only finitely many $G_{\#} (S)$-orbits on the second component
$M(S)$. The bound on $h(a_i)$ for $i=1,\ldots,k$ allows only finitely many simplices in the first component \mbox{$| {\rm mon} (T)|$}.
\end{proof}
\subsection{The conclusion.}
Here we put things together to prove Theorem \ref{theorem5.1}, i.e
$fl \big(G(S)\big) \ge fl\big(G_{{\rm dia}}(S)\big) = h(S) - 1.$

\begin{proof}
We will first show, by induction on $n = S$, that $fl\big(G_{{\rm dia}}(S)\big) = h(S) - 1.$ If  $n = 1$,
then the group $G_0(S)$ is the Houghton group on $h(S)$ rays and has finite index in $G(S).$  In that case the assertion is due to Brown \cite{br87}.
Now we assume $n > 1$. Here we use $M^{[r,s]} = \{ f\in M_{{\rm dia}} (S)\mid r\le~ h(f)\le~ s\},\\ r,s~\in~ \mathbb N$. 
Since $f\in M^{[r,s]}$ is a diagonal pei-injection, the height of $f$ is a multiple of $n$.
So we fix the lower bound $r = nk_0$, $k_0 \in \mathbb N$, and consider 
the filtration of  $M := M^{[r,\infty]}$ in terms
of $M^k := M^{[r,nk]}$, with $k\to \infty$. Then we follow the argument 
of Brown \cite{br87}.
\begin{itemize}  
\item First we note that $M$ is a directed partially ordered set and hence $| M|$  is contractible. 
\item $| M^{k+1} |$  is obtained from $| M^k |$  by adjoining cones over the subcomplexes $| M_{<f} |$ for each $f$ with
$h(f)=k+1$. By Lemma \ref{lemma5.7} and the Remark at the end of Section \ref{subsection5.4} we know that the subcomplexes $| M_{<f} |$  have the homotopy type of a
bouquet of $\big(h(S)- 1\big)$-spheres for $k$ sufficiently large.
This shows that the embedding $| M^k | \subseteq | M^{k+1} |$  is homotopically trivial in all dimensions $< h(S).$
\item By Corollary \ref{corollary5.11} we know that the $| M^k |$ have cocompact skeleta.
\item The stabilizers, stab$_{G(S)}(f)$, of the vertices $f\in M$ -  in fact of all simplices - are of the form $G(S -
Sf)$. As rk$(S - Sf) < {\rm rk}S$ the inductive hypothesis applies. The assumption that $M$ contains only injections $f$
with $h(f) \ge r$ implies now, that $fl \big({\rm stab}_{G(S)}(f)\big) \ge r - 1$ for each $f\in M$.
\end{itemize}

We can choose $r$ arbitrarily; if we choose $r \ge h(S) + 1$ the main results of \cite{br87} apply and it follows that $
fl\big(G_{{\rm dia}}(S)\big) = h(S) - 1.$  This completes the inductive 
step.

In order to prove that $fl \big(G(S)\big) \ge fl\big(G_{{\rm dia}}(S)\big)$ we note that $fl \big(G(S)\big) = fl\big(G_{\rm
tr}(S)\big)$, since $G_{\rm tr}(S)$ is of finite index in $G(S).$ Then we 
observe that $G_{{\rm dia}}(S)$ is a normal subgroup of $G_{\rm
tr}(S)$ with $Q = G_{\rm tr}(S)/G_{{\rm dia}}(S)$ finitely generated Abelian. As  $fl(Q) = \infty$ this implies
$fl\big(G_{\rm tr}(S)\big) \ge fl\big(G_{{\rm dia}}(S)\big).$
\end{proof}

\section{A lower bound for the finiteness length of pet(S) for a stack of 
orthants}
In this section we will show
 \begin{theorem} If $S$ is a stack  of orthants then $fl\big({\emph{pet}}(S)\big)\ge h(S) - 1.$
\end{theorem}
The steps to prove this lower bound of  $fl\big({{\rm pet}}(S)\big)$ are similar to those in Section 5 for
the corresponding pei-result. We will use a certain poset of injective pet-maps $f: S \to  S$ to form a simplicial
complex, and we will choose a diagonal subgroup of ${{\rm pet}}(S)$ for the action on the complex.
However, the part concerning the finiteness length of the stabilizers of $f$ is more difficult here, because the set $(S - Sf)$ is generally not pet-isomorphic to a stack of orthants with lower rank
 (there are different parallelism classes of rank-$(n-1)$ germs in $S$ if 
rk$S = n)$.
So even if the stabilizers are isomorphic to ${\rm pet}(S - Sf)$, there is no base for an induction argument.

In order to set up an inductive proof we need a version of Theorem 5.1, which makes the assertion not only
for stacks of orthants but also for stacks $S$ of parallel copies of a ``rank-n-skeleton'' of an orthant.  In combination with special
injective pet-maps $f$ (the ``super-diagonal'' maps), such a stack $S$ leads to a set $(S - Sf)$, which has the structure
of a stack of rank-$(n-1)$-skeletons.

\subsection{Stack of skeletons of an orthant.}\label{stack of skeletons}
Let $X$ be the canonical basis of the standard orthant $\mathbb N^N$.
Every orthant $L$ is of the form $a+\oplus_{y\in Y}\mathbb N y$, where $Y$ is a subset of $X$. $L$ carries the structure of a
simplex whose faces, indexed by the subsets $Z\subseteq Y$, are the suborthants $L_Z = a+\oplus_{z\in Z}\mathbb N z
\subseteq L.$ We refer to $L_Z$ as a \emph{rank-$k$ face} of $L$ if $| Z| 
=k.$ By the \emph{rank-$k$ skeleton} of $L$,
denoted $L^{(k)}$, we mean the union of all rank-$k$-faces of $L$. Thus the skeleta of $L$ form an ascending chain of
orthohedral set \[ \{ a\}  = L^{(0)} \subseteq  L^{(1)} \subseteq  \ldots \subseteq  L^{(k)} \subseteq  \ldots \subseteq  L^{({{\rm rk}}L)}
 = L.\] Let $L^{(n)}$ be the rank-$n$ skeleton of a rank-$r$ orthant $L 
= a+\bigoplus_{ y\in Y}\mathbb N y$. Then
 $L^{(n)}$ is the union of $h(L^{(n)}) = \binom{r}{n}$ pairwise non-parallel rank-$n$ orthants.

Now we consider a stack $S$ of parallel copies of the rank-$n$-skeleton $L^{(n)}$ of an rank-$r$ orthant - in other words,
$S = R^{(n)}$ is the rank-$n$-skeleton of a stack $R$ of rank-$r$ orthants. We call each copy of $L^{(n)}$
in such a stack S a \emph{component of} $S$, and we write $c(S)$ for the number of components of $S$. Note that
$h(S)= c(S)\binom{r}{n}$.
The next proposition shows a lower bound for $fl\big({\rm pet}(S)\big)$, and the case $n=r$ yields the assertion of
Theorem 5.1.

\begin{proposition}\label{proposition6.2} If $S$ is a stack of rank-$n$-skeletons of an orthant then 
$fl\big(\emph{pet}(S)\big) \ge c(S) - 1$.
\end{proposition}

For later purpose in this section we consider the subset $\mathring{S}\subseteq S$ of all regular
points of $S$, which is defined as follows: If $S$ is an orthant, then $\mathring{S}$ is the image $t_S(S$) of $S$ under
the diagonal unit-translation; and if $S$ is a stack of rank-$n$ skeletons of an orthant, a point $p\in S$ is regular if $S$ contains a
maximal suborthant of rank equal to $S$, which contains $p$ as a regular point. The complement, denoted sing$(S) = S -
\mathring{S}$, is the set of all \emph{singular points} of $S$. 
 In the case when $S$ is a stack of orthants, we will also use the geometrically more suggestive notation $\partial S$ for 
 sing$(S)$. If $S = R^{(n)}$ is the $n$-skeleton of a stack of rank-$r$ 
orthants $L$, then sing$(S) = R^{(n-1)}$ and 
 $\mathring{S}$  has the canonical decomposition as the disjoint union of 
the regular points of the maximal orthants of $S$.
 By a \emph{component} of  $\mathring{S}$ we mean \mbox{$C \cap \mathring{S}$}, the intersection of $\mathring{S}$ with
 a component $C$ of $S$. Note that $c(S) = c(\mathring{S})$.
\begin{lemma} For the sets $S$ and $\mathring{S}$ the following holds
\begin{enumerate}[i)]
\item $S$ and $\mathring{S}$ are pet-isomorphic. Hence  $\emph{pet}(S$) is isomorphic to
$\emph{pet}(\mathring{S})$.
\item $h\big(sing (\mathring{S})\big) = h\big(sing(S)\big)(r - n + 1)$, 
where $r$ is the rank of the stack $R$ with $S
= R^{(n)}$.
\end{enumerate}
\end{lemma}

\begin{proof}
i) $S$ is the disjoint union of $\mathring{S}$ and $(S- \mathring{S})$. As each maximal orthant of $(S-
\mathring{S})$ is parallel to a subortant of $\mathring{S}$, the assertion follows from the pet-normal form.
\\
ii) Since $\mathring{S} = R^{(n)} - R^{(n-1)} $, sing$(\mathring{S})$ is the disjoint union of $h(R^{(n)}) \cdot n$
rank-$(n-1)$ orthants. So $h({\rm sing}(\mathring{S})) = h(R^{(n)})n$. For the height of $S$ and sing$(S)$ we have
$h(S)=h(R^{(n)}) = c(S)\binom{r}{n}$ and $h({\rm sing}(S))=h(R^{(n-1)})= c(S)\binom{r}{n-1}$. As $\binom{r}{n} n =
\binom{r}{n-1} (r-n+1)$, we get $h(R^{(n)}) n = h(R^{(n-1)})(r-n+1)$.
\end{proof}

\subsection{Reduction to the diagonal subgroup.}\label{reduction to the diagonal subgroup}
From now on we assume that $S$ is a stack of rank-$n$ skeletons of an
orthant. Since $S$ and $\mathring{S}$ are pet-isomorphic, it suffices to establish Proposition \ref{proposition6.2} for the set
$\mathring{S}$, which is more suitable for some parts of the proof.
As noted above $\mathring{S}$ is canonically in pet-normal form. 
In particular, every maximal germ of S (or $\mathring{S}$) is represented 
by a unique
maximal orthant of $\mathring{S}$. Thus we can conceptually simplify matters by replacing the set of all maximal germs,
max$\Gamma^\ast(S)$ = max$\Gamma^\ast(\mathring{S})$, by the set of the 
canonical representatives
max$\Omega^\ast(\mathring{S})$, the set of all maximal orthants of $\mathring{S}$.
\par
Let $M_{{\rm dia}} (\mathring{S})$ denote the monoid of all diagonal
pei-injections of $\mathring{S}$ introduced in Section \ref{Monoids of pei-injections}. $M_{{\rm dia}} (\mathring{S})$ is a submonoid of
$M_{{\rm tr}} (\mathring{S})$, the translation submonoid of $M(\mathring{S})$.
Its elements f have the property that they induce, for each
$L \in {\rm max}\Omega^\ast(\mathring{S})$, a diagonal translation $\tau_{(f,L)}: \langle L \rangle \to \langle L
\rangle$. 
 Now we consider the submonoid $M^{{\rm pet}}_{{\rm sdia}}(\mathring{S})\subseteq M_{{\rm dia}} (\mathring{S})$
 consisting of all diagonal pet-injections $f: \mathring{S}\to \mathring{S}$ which satisfy the additional \emph{super-diagonality} condition:

\begin{itemize}
\item[(6.1)] \emph{When two maximal orthants $L, L' of  \mathring{S}$ are 
contained in the same component of~$S$, then the diagonal translations $\tau_{(f,L)}$ and $\tau_{(f,L')}$ have the same translation length.}
\end{itemize}
 
The restriction of the homomorphism \ref{item(5.3)} of Section \ref{Monoids of pei-injections} to $M^{{\rm pet}}_{{\rm sdia}}(\mathring{S})$ can thus be interpreted as a map
\begin{itemize}\item[(6.2)]
\begin{equation}\label{eq15} 
\lambda : M^{{\rm pet}}_{{\rm sdia}}(\mathring{S})\twoheadrightarrow \bigoplus_{ C\in {\rm Comp}(\mathring{S})}\mathbb Z = \mathbb Z^{c(S)},
\end{equation}
\end{itemize}
which associates to each super-diagonal pet-injection $f$ the translation 
length $\lambda 
(f,C)$ on each component $C$ of $\mathring{S}$.

The group of all invertible elements of $M^{{\rm pet}}_{{\rm sdia}} (\mathring{S})$ is the \emph{super-diagonal pet-group} 
${\rm pet}_{{\rm sdia}} (\mathring{S})$. 

Let be ${\rm pet}_{{\rm tr}} (\mathring{S})$ the group of all invertible elements of $M_{{\rm tr}} (\mathring{S})$.
It is a subgroup of ${\rm pet}(\mathring{S})$, which has finite index in ${\rm pet}(\mathring{S})$.
Analogous to (5.2) in Section \ref{Monoids of pei-injections} is a homomorphism
\begin{itemize}
 \item[(6.3)]
\begin{equation}
\kappa : {\rm pet}_{{\rm tr}}(\mathring{S})\twoheadrightarrow \bigoplus_{ 
L\in \max
\Omega^\ast(\mathring{S})} {\rm Tran}(\langle L \rangle )\text{, given by 
}\kappa (g) = \bigoplus\limits_{ L\in \max\Omega^\ast(\mathring{S})}\tau (g,L)
\end{equation}
\end{itemize}
\noindent which associates to each pet-injection $g\in {\rm pet}_{{\rm tr}}(\mathring{S})$ the translation length $\tau 
(g,L)$ on each maximal orthant $L$ of max$\Omega^\ast(\mathring{S})$. We observe that a permutation $g\in{\rm pet}_{{\rm tr}}(\mathring{S})$ is in 
${\rm pet}_{{\rm sdia}} (\mathring{S})$, if and only if the translations 
$\tau_{ (g,L)}$ are diagonal for each $L$ and its translation length constant as $L$ runs through the maximal orthants of
a component $C$ of $\mathring{S}.$ 
\par
Given a component $C$ of $\mathring{S}$, we consider the set $\Lambda(C):={\rm max}\Omega^\ast({C})$ of all $h(C) =
\binom{r}{n}$ rank-$n$ orthants of $C$. For each orthant $L\in \Lambda(C)$, we
write $Y(L)$ for its canonical basis. The translation $\tau_{ (g,L)} :\langle L \rangle \to\langle L\rangle$ has the
canonical decomposition into the direct sum of translations $\tau^y_{(g,L)}$ in the directions $y\in Y(L)$, and we write
$l^y(g, L)\in \mathbb Z$ for the corresponding translation lengths. 
\\
Therefore, for $g\in {\rm pet}(\mathring{S})$ to be super-diagonal, means 
that the numbers  $l^y(g, L)\in\mathbb Z$  coincide 
for all pairs in $P(C) := \{ (y, L)\mid y\in L\in \Lambda(C) \}$ 
- and this is so for all components $C$. Hence, associating to $g$ the sequence $$\big( l^y(g, L) - l
^{y'}(g,L')\big)_{( i(C),C)},$$
with $i(C)$ running through all pairs $\big((y, L), (y', L')\big) \in  P(C)$, and
$C$ through the components of $\mathring{S}$, exhibits the super-diagonal 
pet-group ${\rm pet}_{{\rm sdia}} (\mathring{S})$
as the kernel of a homomorphism of ${\rm pet}_{{\rm tr}}(\mathring{S})$ into
a finitely generated Abelian group. It is well known that in this situation 
$fl\big({\rm pet}_{{\rm tr}}(\mathring{S})\big) \ge fl\big({\rm pet}_{{\rm sdia}} (\mathring{S})\big)$. Since
${\rm pet}_{{\rm tr}}(\mathring{S})$ has finite index in ${\rm pet}(\mathring{S})$, we have
$fl\big({\rm pet}(\mathring{S})\big) = fl\big({\rm pet}_{{\rm tr}}(\mathring{S})\big)$, hence
$$fl\big({\rm pet}(\mathring{S})\big) \ge fl\big({{\rm pet}}_{{\rm sdia}} 
(\mathring{S})\big).$$ 
The proof of Proposition \ref{proposition6.2} is thus reduced to a proof of $fl\big({{\rm pet}}_{{\rm sdia}} (\mathring{S})\big) =
c(S) - 1.$ To show this, we follow the arguments in the proof of the corresponding pei-result: 
$fl\big({{\rm pei}}_{{\rm dia}}(S)\big) = h(S) - 1$, where $S$ was a stack of $h(S)$ orthants of rank $n$. 
In the present situation, where $\mathring{S}$ is the set of regular points of the $n$-skeleton of the stack $R$ of $h(R)$ 
orthants, the components $C$ of $\mathring{S}$ have to take over the role 
previously played by the orthants $L$ of the stack $S$.
Correspondingly we now have to work with the multiplicative submonoid ${\rm mon} (T)\subseteq M^{{\rm pet}}_{{\rm sdia}} (\mathring{S})$ 
freely generated by the set $T$ of all super-diagonal unit-translations $t_C: C\to C$ as $C$ runs through the
components of $\mathring{S}$, where each $t_C=\prod_{L\in \Lambda(C)}t_L$ is the composition of the diagonal unit-translations
$t_L$ defined in Section \ref{Monoids of pei-injections}. As at the end of Section \ref{Monoids of pei-injections} we use the action of ${\rm mon} 
(T)$ by left multiplication to
endow $M^{{\rm pet}}_{{\rm sdia}} (\mathring{S})$ with a partial ordering; and we observe that this partial ordering is directed.

\subsection{Maximal elements \texorpdfstring{$\mathbf{< f}$}{TEXT} in \texorpdfstring{$\mathbf{M^{{\rm pet}}_{{\rm sdia}} (\mathring{S})}$}{TEXT}.} 
To adapt notation to the one used in the corresponding pei-situation in Section 5, we write $\Lambda$  for the set of all
components $C$ of $\mathring{S}$, and $\partial C := C - Ct_C$  for each component $C\in \Lambda$ .
Note that $C$ is the disjoint union of $h(C) = \binom{r}{n}$ rank-$n$ orthants
-- using the notation of Section \ref{stack of skeletons} one for each $n$-element set $Z\subseteq Y$.
Hence $h(\partial C) = n \binom{r}{n}.$ We are still in the situation that all maximal orthants of $\mathring{S}$ have the same
 finite rank $n = rk S.$  And given $f\in M^{{\rm pet}}_{{\rm sdia}} (\mathring{S})$ we write
\[	M_{ <f}   = \{ a\in M^{{\rm pet}}_{{\rm sdia}} (\mathring{S})\mid a <f  \} , ~ 	M_{\le f} = \{ a\in M^{{\rm pet}}_{{\rm sdia}} (\mathring{S})\mid a\le f\} ,\]  
for the \emph{``open resp. closed cones below $f$''}, aiming to understand the set of all maximal
elements of $M_{<f}$.

\begin{lemma} Let $b$ be a maximal element of  $M_{<f}$. Then there is a unique component $C$ of $\mathring{S}$ with the 
property that $f = t_Cb$, and $h(f) = h(b) + n$. Furthermore, $b$ is given as the union $b = b'\cup b''$, where b': 
$\partial C\to (\mathring{S} - \mathring{S}f)$ is a pet-injection, and $b'': (\mathring{S} - \partial C) \to 
\mathring{S}f$ is the restriction $t_C^{-1}f|_{(\mathring{S} - \partial C)}$.
Converseley, if $c': \partial C\to (\mathring{S} - \mathring{S}f)$ is an arbitrary pet-injection distinct to
$b'$, then the union $c = c'\cup b''$ is a maximal element of $M_{< f}$ 
distinct to $b$.
\end{lemma}

\begin{proof}See argument in Lemma \ref{lemma5.4}.\end{proof}

\begin{lemma}\label{lemma6.5} Let $B\subseteq M_{<f}$  be a finite set of 
maximal elements of $M_{<f}$  . Then the following conditions are equivalent:
\begin{enumerate}[i)]
\item The elements of $B$ have a common lower bound $\delta $ in $ M_{<f}.$
\item For every pair $(b,b')\in B\times B$, with $b\not= b'$ and $tb =f =t'b'$ for super-diagonal unit-translations $t$,
$t'$, we have
\begin{enumerate}[a)]
\item		 $t\not= t'$,  and 
\item		  $b(\partial C )\cap b'(\partial C') = \emptyset$,
 where $C$ resp. $C'$ are the components of $\mathring{S}$ on which $t$ resp. $t'$ acts non-trivially.
\end{enumerate}
\end{enumerate}
\end{lemma}
\begin{proof}
For each $b\in B$ we have a super-diagonal unit-translation $t_b\in T$ with $t_b b = f$, and we put

\begin{itemize}
\item[(6.4)] \begin{equation}\label{eq17}
t_B := \prod_{ b\in B}t_b.
\end{equation}\
\end{itemize}

By assumption (a) the components $C_b$, on which $t_b $ acts non-trivially, are pairwise disjoint. 
Thus $|B|$ $\le c(S)$, and $S$ decomposes in the disjoint union  $S  =
\left(\bigcup_{b\in B}C_b \right) \cup  S'$.
We define the pet-injection $\delta_B: S\to S$ as follows:
$$ \delta_B := \begin{cases}
{t_b}^{-1} f & \text{ on each } C_b t_b \\ 
b & \text{ on the complements } \partial C_b = C_b - C_b t_b \\ 
f & \text{ on } S'. 
\end{cases} $$
To show that $\delta_B$ is a common lower bound, see arguments in Lemma \ref{lemma5.5}.  \end{proof}

\begin{lemma}\label{lemma6.6} In the situation of Lemma \ref{lemma6.5} we 
have for the lower bound
$\delta_B$ defined in the proof:
\begin{enumerate}[i)]
\item $\delta_B$ is, in fact, a largest
common lower bound of the elements of $B$
\item $h(\delta_B ) \ge h(f) - h(S)n$.
\end{enumerate}
\end{lemma} 	
\begin{proof}
For (i) see argument in Lemma 5.6. For (ii) we use the translation  $t=
\Pi_{ b\in B}t_b$ of \eqref{eq17} which satisfies $t\delta_B  = f$  and 
yields:
\[ 
\begin{array}{lclr}
h(\delta_B ) = h(f) - h(t ) &=& h(f) - |B| \cdot h(t_C) &\\
&=&  h(f) - |B| \cdot h(C)n &\\
&\ge& h(f) - c(S)h(C)n = h(f) - h(S)n.&
\end{array}\]
\end{proof}

\subsection{The simplicial complex of \texorpdfstring{$\mathbf{M^{{\rm pet}}_{{\rm sdia}} (\mathring{S})}$}{TEXT}.}
We consider the simplicial complex
$|M^{{\rm pet}}_{{\rm sdia}} (\mathring{S})|$, whose vertices are the elements of $M^{{\rm pet}}_{{\rm sdia}} (\mathring{S})$ and whose chains of
length $k,  a_0  < a_1  < \ldots < a_k$, are the $k$-simplices. As the partial ordering on $M^{{\rm pet}}_{{\rm sdia}} (\mathring{S})$
is directed, $|M^{{\rm pet}}_{{\rm sdia}} (\mathring{S})|$  is contractible.

In this section we aim to prove    
\begin{lemma}\label{lemma6.7} If $h(f)\ge2\cdot \emph{rk}S\cdot h(S)$ then $|M_{<f}|$  has the homotopy type of a bouquet of
$\big(c(S) - 1\big)$-spheres.
\end{lemma}
 The first step towards proving Lemma \ref{lemma6.7} is to consider the covering of  $|M_{<f}|$  by the subcomplexes $|M_{\le b}|$,
 where $b$ runs through the maximal elements of  $M_{<f}$. We write $N(f)$ for the nerve of this covering. Lemma~ \ref{lemma6.6}~ i) asserts that all finite intersections of such subcomplexes $|M_{\le b}|$  are again cones and hence contractible. It is a well known fact that in this situation the space is homotopy equivalent to the nerve of the covering. Hence we 
have \[  |M_{<f}| ~\text{ is homotopy equivalent to the nerve}~ N(f),\]
and it remains to compute the homotopy type of~ $N(f).$

The next step - replacing the nerve $N(f)$ by the combinatorial complex $\Sigma(f)$ - follows the arguments in Section 3:
We find that the set of vertices of $\Sigma(f)$ is the disjoint union  $A 
= \bigcup_{ C\in \Lambda} A_C $, where $A_C $
stands for the set of all pet-injections  $a: \partial C\to (\mathring{S} 
- \mathring{S}f )$; 
and the $p$-simplices of $\Sigma(f)$ are the sequences $(a_C)_{C\in \Lambda'}$, where $\Lambda ' $ is a $p$-element subset
of $\Lambda$ whose entries $a_C\in A_C $ satisfy the condition
\begin{itemize}
 \item [(6.5)] \quad \emph{The intersections of the images}~ $ a_C(\partial C), C\in \Lambda'$, ~\emph{are pairwise disjoint.}
\end{itemize}
The homotopy type of $\Sigma(f)$ can again be computed by Lemma \ref{lemma5.7}, which we apply to the $1$-skeleton  $\Gamma (f)$ of 
$\Sigma(f)$, viewed as a $c(S)$-colored graph $\Gamma (f)_{c(S)}$.
At the end it remains to prove 
\begin{lemma} If  $h(f)
\ge 2\cdot\emph{rk}S\cdot h({S})$  then $\Gamma (f)_{c(S)}$ satisfies the 
assumptions of Lemma \ref{lemma5.8}.
\end{lemma}
\begin{proof}
Let $n := {\rm rk}S$ and $h := c(S)$. Assumption (1) is a consequence 
of assumption (2) except in the trivial case $h = 1.$ 

To prove (2) we fix $C\in \Lambda$  and consider a set of $2(h - 1)$ elements $F\subseteq A - A_C$ . We have to show that
there are two elements  $a, b\in A_C$  with the property that for each $d\in F, d: \partial C_d \to  (\mathring{S} -
\mathring{S}f), im(d) = d(\partial C_d)$ is disjoint to both $a(\partial C)$ and $b(\partial C).$ In other words: there
are two pet-injections

(6.6)	\quad	$a, b:  \partial C  \to  (\mathring{S} - \mathring{S}f) - \big(\bigcup_{ d\in F} im(d)\big).$
 
For this it suffices to compare the height function - i.e., the number of 
rank-$(n-1)$ germs - of domain and target. 
Clearly, $h\big(a(\partial C)\big) = h(\partial C) = nh(C)$, and the same applies to every vertex of $A.$ Hence 
$h\big(\bigcup_{ d\in F} im(d)\big) \le 2(h - 1)nh(C)$. By assumption $h(\mathring{S} - \mathring{S}f) \ge
2nh({S})$, and so the target orthohedral set has height at least $2nh({S}) - 2(h - 1)nh(C) = 2n h(C)$, which is more than at 
least twice the height $h(\partial C) = nh(C)$ of the domain when $h(C)$ is positive. Moreover, by Lemma 4.10 below, the
set $(\mathring{S} - \mathring{S}f)$ is pet-isomorphic to a stack of copies of $\partial C$. In this situation one observes that the two different 
pet-injections required in (4.6) certainly do exist. 
This proves the Lemma 4.8 and hence Lemma 4.7.
\end{proof}

\begin{remark}
By the same argument as in Remark 3.10, the assertion of Lemma 4.7 holds true if $M_{<f}$ is replaced with the subset
$M_{r,f} :=\{a \in M^{{\rm pet}}_{{\rm sdia}} (\mathring{S}) \mid h(a) \geq r ~ and ~ a < f \}$ and $f$ satisfies the
additional condition $h(f) \geq r + h(S)$.
\end{remark}

\subsection{Stabilizers and cocompact skeletons of \texorpdfstring{$\mathbf{|M^{{\rm pet}}_{{\rm sdia}} (\mathring{S})|}$}{TEXT}.}
Here we consider the monoid
$M^{{\rm pet}}(S)$ of all pet-injections endowed with the height function 
$h: M^{{\rm pet}}(S)\to \mathbb Z$ inherited from $M(S)$ and the
${\rm pet}(S)$-action induced by right multiplication. Our main interest, 
however, is the
super-diagonal submonoid $M^{{\rm pet}}_{{\rm sdia}} (\mathring{S})\subseteq M^{{\rm pet}}(S)$ acted on by the super-diagonal pet-group
${{\rm pet}}_{{\rm sdia}} (\mathring{S}).$ 
\begin{lemma}\label{lemma6.9} \begin{enumerate}[i)]
\item $f,f' \in M^{\emph{pet}}(S)$ are in the same
\emph{pet}$(S)$-orbit if and only if $(S - Sf )$ and $(S - Sf ')$ are 	pet-isomorphic.
\item Let $S$ be a stack of copies of $L^{(n)}$, where $L$ is a rank-$r$ orthant. If $f\in M^{\emph{pet}}_{\emph{sdia}}
(\mathring{S})$ with $h(f) > 0$, then $h(f)$ is a multiple of $(r-n+1)\binom{r}{n-1}$ and $(\mathring{S} - \mathring{S}f)$
is pet-isomorphic to a stack of $h(f) / \binom{r}{n-1}$ copies of  $L^{(n-1)}$.
\item Two elements $f,f '\in M^{\emph{pet}}_{\emph{sdia}} (\mathring{S})$ 
with $h(f) = h(f ') > 0$ are in the same
$\emph{pet}_{\emph{sdia}} (\mathring{S})$-orbit.
\end{enumerate}
\end{lemma}
\begin{proof}
The proof of i) is analogous to the pei-version in Section 3.\\  
ii) The key here is a pet-version of Lemma 3.2i). The set of germs $\Gamma^{ n-1}~(\mathring{S})$ decomposes into its
parallelism classes, and as these are ${\rm pet}(\mathring{S})$-invariant, the height function 
$h: M^{{\rm pet}}_{{\rm sdia}} (\mathring{S})\to \mathbb Z$ can be written as the sum of functions 
$h_Y: M^{{\rm pet}}_{{\rm sdia}} (\mathring{S})\to \mathbb Z$, with $Y$ running through all $(n-1)$-element subsets of the canonical basis 
of $L$, that count the number of germs in $\Gamma^{ n-1}(\mathring{S} - \mathring{S}f)$ parallel to $\langle Y\rangle .$

We need the $h_Y $-version of Lemma 3.2i), asserting that we have for all 
rank-$(n-1)$ faces $Y$ of $L$
  
\begin{itemize}
\item[(6.7)] $h_Y (f)  =  h_Y \big(A\cap (Af)^c\big) - h_Y (A^c \cap Af)$, \emph{for every orthohedral subset $A\subseteq S$
      with} ${\rm rk}A^c  < n$.
 \end{itemize}
  
The proof is a straightforward adaptation of the one in Section 3.1 and can be left to the reader.

We can refine (6.7) by exhibiting $A$ as the disjoint union of rank-$n$ orthants $K_i$  on which $f$ acts by
(super-diagonal) translations. Since $f$ is super-diagonal, the corresponding translation lenghts $\lambda_{C(i)}$ depend
only on the component $C(i)$ of $\mathring{S}$ containing $K_i$. One observes that $f(K_i )$ is
contained in the uniquely defined maximal orthant of $\mathring{S}$ containing $K_i$. This has the consequence that for
$i \not=  j,  K_i  \cap (K_jf)^c  = K_i$   and $K_i^c\cap K_jf = K_jf,$ from which one finds

\begin{itemize}
\item[(6.8)]   $h_Y (f)$ = $h_Y  \big(\bigcup_i (K_i \cap (K_i f)^c\big)- h_Y  \big(\bigcup_i ( K_i^c \cap K_i f)\big) \\ = \sum_i h_Y(K_i \cap (K_i f)^c -  h_Y( K_i^c \cap K_i f) = \sum_i \lambda_{C(i)}$.
\end{itemize}

Clearly, for each
component $C$ of $S$, $sing(C)$ contains exactly one orthant parallel to $\langle Y\rangle$. By Lemma 4.3ii) this orthant
is parallel to a face of exactly $(r - n + 1)$ maximal orthants $K_i$  in 
$C$, and each of them gives rise to a summand
$\lambda_{C(i)}$ .
Hence summation over all $K_i$  contained in a single component $C$ of $S$ yields  $\lambda_C(r - n + 1)$. And summation over all I, finally,

 $h_Y (f)  =  \lambda (r - n + 1)$ , where $\lambda$  is the sum of $\lambda_C$,  with $C$ running through all components of $S$.

This shows, in particular, that $h_Y (f)$ is independent of $Y$. As we are assuming that $h(f) > 0$ it follows that 
$\lambda  > 0$  and  $h(f) = \lambda \cdot (r - n + 1)\cdot \binom{r}{n-1}$.

It follows that $(\mathring{S} - \mathring{S}f)$  is pet-isomorphic to disjoint union $S'\cup S''$, where $S'$ is a stack of 
$h_Y (f) = \lambda (r - n + 1)$ copies of $L^{(n-1)}$ and $S''$ a subset of rank $< n - 1$. As $\lambda  > 0$, $S'$
contains at least one copy of $L^{(n-1)}$. In this situation $S'$ contains orthants parallel to any given maximal orthant of
$S''$. In view of the the pet-normal form of  $S'\cup S''$  it follows that  $S'\cup  S''$ is pet-isomorphic to $S'$, i.e., 
to a stack of    $\lambda (r - n + 1)$ copies of $L^{(n-1)}$.

iii) Part ii) shows that $h(f) = h(f ') > 0$ implies that $\mathring{S} 
- \mathring{S}f$ and $\mathring{S} - \mathring{S}f
'$ are pet-isomorphic. Hence, by assertion i), there is a pet-permutation 
$g\in {{\rm pet}}(\mathring{S})$ with $f ' =
fg$. The assumption that $f$ and $f '$ are in $M^{{\rm pet}}_{{\rm sdia}} 
(\mathring{S})$ implies that $g\in
{{\rm pet}}_{{\rm sdia}} (\mathring{S}).$
\end{proof}

 \begin{corollary} \label{corollary6.10}
 \begin{enumerate}[i)]
 \item The stabilizer of  $f\in M^{\emph{pet}}_{\emph{sdia}} (\mathring{S})$ in
 ${\emph {pet}}_{{\rm sdia}} (\mathring{S})$ is isomorphic to ${\emph{pet}}(\mathring{S} - \mathring{S}f)$
\item For every number $r, s\in \mathbb N\cup\{0\}$, the simplicial complex of $M^{[r,s]}:= \{ f\in
M^{\emph{pet}}_{\emph{sdia}} (\mathring{S})\mid r\le h(f)\le s\}$  is cocompact under the \emph{pet}$_{\emph{sdia}}
(\mathring{S})$-action.
\end{enumerate}
\end{corollary}

\begin{proof}
i) Same reasoning as the proof of Corollary 3.12. 
\\ii) Lemma 4.10iii) asserts that ${{\rm pet}}_{{\rm sdia}} (S)$ acts
cocompactly on the vertices of a given height in the  simplicial complex $| M^{{\rm pet}}_{{\rm sdia}} (S)|$. Just as in the proof 
of Corollary 3.12ii) this yields the claimed assertion.
\end{proof}

\subsection{The conclusion.}
It is an elementary observation that right action of $g\in {{\rm pet}}(S)$ on
$M^{{\rm pet}}(S)$ fixes an element $f\in M^{{\rm pet}}(S)$ if and only if $g$ restricted to $f(S)$ is the identity. In other words, the stabilizer of a vertex
$f\in M^{{\rm pet}}(S)$ is isomorphic to \mbox{${\rm pet}(S - Sf)$}. This 
will be crucial for the inductive step in the
following inductive

\begin{proof}[Proof of Proposition 6.2]
In Section \ref{reduction to the diagonal subgroup} we roved already, $fl\big({\rm pet}(\mathring{S})\big) \geq\\
fl\big({\rm pet}_{{\rm sdia}}(\mathring{S})\big)$; hence it suffices to show
$fl\big({{\rm pet}}_{{\rm sdia}}(\mathring{S})\big) \ge c(S) - 1$. We will argue by induction on $n = {\rm rk}S$.
If $n = 1$, then $h(S)=c(S)\cdot r$, and the group ${\rm pet}_{{\rm sdia}}(\mathring{S})$ is the Houghton group on $h(S)$ rays.
By Brown \cite{br87} this implies that $fl\big({{\rm pet}_{{\rm sdia}}}(\mathring{S})\big) \ge h(S) - 1 \ge c(S) - 1$.
This establishes the case $n = 1$ of the induction.
      
Now we assume $n > 1$. By induction we can assume that $fl\big({{\rm pet}}(S')\big) \ge c(S') - 1$ holds for
every stack S' of copies of a rank-$(n-1)$ skeleton of an orthant $L$. To 
prove the inductive step we start with restricting
attention to the subgroup ${\rm pet}_{{\rm sdia}} (\mathring{S})$ acting on the super-diagonal monoid
$M^{[u,v]} = \{ f\in  M^{{\rm pet}}_{{\rm sdia}} (\mathring{S})\mid u\le h(f)\le v\} , u,v\in  \mathbb N$. By
Lemma 4.10ii) the $h(f)$  is a multiple of $s := (r - n + 1)\binom{r}{n-1}$. So we fix a lower bound $u = sk_0, k_0\in 
\mathbb N$, and consider the filtration of  $M:= M^{[u, \infty]}$ in terms of $M^k = : M^{[u,sk]},$ with $k\to \infty.$
Then we argue as follows
\begin{itemize}
\item First we note that $M$ is a directed partially ordered set and hence $|M|$  is contractible. 
\item $| M^{k+1} |$  is obtained from $| M^k |$  by adjoining cones over the subcomplexes $| M_{<f} |$ for each $f$
with $h(f)=k+1$. By Lemma 4.7 and Remark 4.9 we know, that the subcomplexes $| M_{<f} |$  have the homotopy type of a
bouquet of $\big(c(S)- 1\big)$-spheres for k sufficiently large. This shows that the embedding $|M_k| \subseteq |M_{k+1}|$
is homotopically trivial in all dimensions $< c(S)$.
 \end{itemize}
\begin{itemize}
\item By Corollary \ref{corollary6.10} ii) we know that each $| M^k|$  has cocompact skeleta.
\item The stabilizer of the $f\in M$ under the action of ${\rm pet}_{{\rm 
sdia}} (\mathring{S})$ on $M$ coincides with
${\rm pet}(\mathring{S} - \mathring{S}f)$ by Corollary \ref{corollary6.10} i). Lemma \ref{lemma6.9}ii) asserts that if $h(f) > 0$ then $(\mathring{S} -
\mathring{S}f)$ is pet-isomorphic to a stack of copies of the rank-$(n-1)$ skeleton of an orthant.  The stack height here is
 $c(\mathring{S} - \mathring{S}f) = h(f)/\binom{r}{n-1}$. We can choose 
u arbitrarily; if we choose $u = \big(c(S) +  1)
 \binom{r}{n-1}$ the inductive hypothesis together with the assumption that $h(f)\ge u$ yields
 $${\rm fl}\big({{\rm pet}}(\mathring{S} - \mathring{S}f)\big) \ge c(\mathring{S} - \mathring{S}f) - 1 = h(f) \left/
 \binom{r}{n-1} - 1 \right. \ge c(S),$$
 for all~$f\in M$.
\item The main results of \cite{br87} now establishes $fl\big({\rm pet}_{{\rm sdia}}(\mathring{S})\big) \ge c(S) - 1.$
This completes the inductive step.
\end{itemize}
\end{proof}

\section{The upper bounds of \texorpdfstring{$fl\left({\rm pet}(S)\right) 
$}{TEXT} when \texorpdfstring{$S \subseteq  \mathbb{N}^N$}{TEXT}}

\subsection{More structure at infinity.}
Here we assume, for simplicity, that our orthohedral sets~$S$ are contained in
$\mathbb N^N$. By Corollary 1.5 this is not a restriction for the pei-group ${\rm pei}(S)$, and it is a basic
special case for the  pet-group ${\rm pet}(S)$:

Given an element $x\in X$ (i.e. a coordinate axis), we write $\Gamma_{x}^1(S)$ for the set of all germs of rank-$1$
orthants of $S$ parallel to $\mathbb N x$. We have a canonical embedding $\kappa : \Gamma_{x}^1(S)\to  \mathbb N^{ N-1}$
defined as follows: Each $\gamma \in \Gamma_{x}^1(S)$ is represented by a 
unique maximal orthant $L\in \Omega^1(S)$; we
delete the $x$-coordinate of the base point of $L$ and put $\kappa (\gamma )$ to the remaining coordinate vector. We write
$\partial_xS$ for the image $\kappa \big(\Gamma_{x}^1(S)\big)$, and we will often identify $\Gamma_{x}^1(S)$ with
$\partial_xS$ via $\kappa .~ \partial_xS$ can be viewed as the boundary of $S$ \emph{at infinity in direction}~$x$.
\begin{lemma}For $\partial_xS$ the following hold.
\begin{enumerate}[i)]
\item $\partial_xS$ is an orthohedral subset of $\mathbb N^{N-1}.$
\item For very rank-$(k-1)$ orthant $L\subseteq \partial_xS$ there is a unique rank-$k$ orthant $L'\subseteq S$, which
is maximal with respect to the property that for each point of $p\in L~ \kappa^{ -1}(p)$ is  represented by a suborthant of $L.$
\end{enumerate}
\end{lemma}

\begin{proof}Easy.\end{proof}

\subsection{Short exact sequences of pet-groups.}
From now on we assume that $S = \bigcup_j S_j \subseteq  \mathbb N^m$
where $m$ is minimal and $S$ is in pet-normal form as defined after Proposition 1.5. Given $x\in X$ arbitrary we note that
$S$ is the disjoint union $S = S(x)\cup S^{\bot}(x)$, where $S(x)$ collects the stacks $S_j$ which contain a rank-$1$ orthant parallel to $\mathbb N
x$, and $S^{\bot}(x)$ the stacks $S_j$ which are perpendicular to $x$. We 
note that $\partial_x S = \partial_x S(x)$, and
we have an obvious projection $\pi_x: S(x) \twoheadrightarrow \partial_x S$. Moreover, there is a canonical injection 
$\sigma_x: \partial_x S \to  S(x)$ which maps each germ $\gamma \in \partial_x S$ to the base point of the unique maximal 
rank-$1$ orthant representing $\gamma$, and is right-inverse to  $\pi_x: S(x) \twoheadrightarrow \partial_x S.$

As the action of ${\rm pet}(S)$ on the $\Omega^1(S)$ preserves directions 
it induces, for each coordinate axis $x\in X$, an
action on $\Gamma_{x}^1 (S) = \partial_x S$, and one observes that this 
is an action by pet-permutations. This yields an
induced homomorphism \mbox{$\vartheta_x : {\rm pet}(S)\to {\rm pet}(\partial_x S)$}. The kernel of $\vartheta_x$  is the set of
all pet-permutations fixing all rank-$1$ germs parallel to $x$. And we note the following:

\begin{enumerate}[i)]
\item
\mbox{$\sigma_x: \partial_x S \to  S(x)$} induces an embedding of ${\rm pet}(\partial_x S)$ as a subgroup of
 ${\rm pet}\big(S(x)\big)$, which splits the surjective homomorphism
 $$\vartheta_x : {{\rm pet}}(S(x)) \twoheadrightarrow 
 {{\rm pet}}(\partial_x S)$$
 induced by $\pi_x$.
\item  Every pet-permutation $g\in {{\rm pet}}\big(S(x)\big)$ extends to a pet-permutation of $S$ by the identity on $S^{\bot}(x)$. 
This exhibits ${{\rm pet}}\big(S(x)\big)$ as a canonical subgroup of ${\rm pet}(S)$. Even though we do not have ${\rm pet}(S)$ acting on
$S(x)$, we do have that the surjective homomorphism
$$\vartheta_x : {{\rm pet}}(S)\twoheadrightarrow {\rm pet}(\partial_x
S)$$
splits by the embedding ${{\rm pet}}(\partial_x S)\le {{\rm pet}}\big(S(x)\big)\le {{\rm pet}}(S).$
\end{enumerate}
Summarizing we have
\begin{proposition}
${\emph{pet}}(\partial_x S)$ is a retract both of $
{\emph{pet}}(S)$ and of ${\emph{pet}}\big(S(x)\big)
$. In other words, we have split exact sequences 
$$1 \to  K \to  {\emph{pet}}(S) \to {\emph{pet}}(\partial_x S)\to ~ 1$$ 
and 
$$1 \to  K^\dag
\to  {\emph{pet}}\big(S(x)\big) \to {\emph{pet}}(\partial_x S) \to  1.$$\hfill$\square$
\end{proposition}

\subsection{An upper bound of the finiteness length of pet(S)}
To deduce an upper bound for the finiteness
lengths of the pet-groups we need the following elementary lemma which was overlooked in [BE75]; the first occurrence in
the  literature we are aware of is [Bu04].
\begin{lemma}
Let $G$ be a group. If a subgroup $H\le G$ is a retract of $G$ then  $fl(H) \ge fl(G).$
\end{lemma}
\begin{proof}
The assertion  $fl(G)\ge s$ is equivalent to saying that on the category of $G$-modules the homology functors  
$H_k(G; - )$ commute with direct products for all $k<s.$ That this is inherited by retracts follows from the fact that 
$H_k( - ; - )$ is a functor on the appropriate category of pairs $(G,A)$, 
with $G$ a group and $A$ a $G$-module. 
\end{proof}

If $S \subseteq \mathbb N^N$ is an orthohedral set of rank ${\rm rk}S = 
n$ in pet-normal form, then
$\Omega_0^\ast(\mathbb N^m)$ is canonically bijective to the set $P(X)$ of all subsets of $X$. Hence we can view the
height function (1.1) of section 1.4 as a map
$$h_S:P(X) \to \mathbb N\cup\{0\},$$
and organize stacks of maximal orthants of $S$ as follows: For every subset $Y\subseteq X$ we have
the (possibly empty) stack $S(Y) \subseteq S$ of $h_S(Y)$ orthants parallel to the orthant $\langle Y \rangle$ defined by
$Y$.
\\For each $(n-1)$-element subset $Y\subseteq X$ we consider the link ${{\rm Lk}}(Y)$ of $Y$ in $S_\tau$, by this we mean the set
of all $n$-element sets $Y'\subseteq Y$ with the property, that $\langle Y' \rangle\subseteq S_\tau$, noting that $\langle
Y' \rangle \in {\rm max} \Omega_0^\ast(S_\tau)$. Then we put $S({{\rm Lk}}(Y)) \subseteq S$ to be the union of the stacks $S(Y')$
with $Y'$ running through ${{\rm Lk}}(Y)$. The height, $h\big(S\big({{\rm 
Lk}}(Y)\big)\big)$, is the sum of all
stack-heights $h_S(\langle Y' \rangle)$ as $Y'$ runs through ${{\rm Lk}}(Y)$.

\begin{theorem}
If $S \subseteq  \mathbb N^m$ is orthohedral of rank $n$ in pet-normal form, then each indicator $(n-1)$-subset $Y
\subseteq X$ with non-empty link ${\emph{Lk}}(Y)$ imposes an upper bound $fl\big(\emph{pet}(S)\big)\leq
h\big(S\big({\emph{Lk}}(Y)\big)\big)-1$.
\end{theorem}

\begin{proof}
We choose any $y\in Y$ and consider the projection \mbox{$\pi_y: S \twoheadrightarrow
\partial_y S\cup\{\emptyset\}$}, where the symbol $\{\emptyset\}$ is the image of $S-S(y)$. Proposition 5.2 asserts, that
${\rm pet}(\partial_y S)$ is a retract of ${\rm pet}(S)$. We have ${\rm rk}\partial_yS={\rm rk}S-1$; in fact, $\partial_yS$ is the
disjoint union of stacks $S(Z)$, with $Z$ running through all subsets of $X$ avoiding y and satisfying $\langle Z\cup \{y\}\rangle\in
{\rm max} \Omega_0^\ast(S_\tau)$. Thus note that $S(Z)$ is a stack of rank-$(n-1)$ orthants with unchanged stack height
$h(S(Z))=h_S(Z\cup \{y\})$.
\\We can choose the next element $y'\in~ Y - \{y\}$, consider the projection $\pi_{y'}: \partial_y S \twoheadrightarrow
\partial_{y'}\partial_y S\cup\{\emptyset\}$. Upon putting $\pi_y(\emptyset)=\emptyset$ for all $y\in Y $, we can iterate
the argument with all elements of $Y=\{y,\ldots,z\}$, noting that only the stacks in $S({{\rm Lk}}(Y))$ survive all  these projections. The composition
$$
\pi_y = \pi_z \ldots \pi_y: S \twoheadrightarrow \partial_z \ldots\partial_y \cup \{\emptyset\}
$$
projects the stacks of $S({{\rm Lk}}(Y))$ onto stacks of rank-1 orthants with the original stack heights. This shows that ${\rm pet}(S)$
admits a retract isomorphic to the pet-group ${\rm pet}(S')$ of a disjoint union of $h(S({{\rm Lk}}(Y))$ rank-1 orthants. But ${\rm pet}(S')$
contains Houghton's group on $h\big(S\big({{\rm Lk}}(Y)\big)\big)$ copies 
of $\mathbb N$ as a subgroup of finite index.
Hence Lemma 5.3 together with Ken Brown's result \cite{br87} yields  $fl\big({\rm pet}(S)\big) \leq fl\big({\rm pet}(S')\big) 
= h\big(S({{\rm Lk}}(Y)\big) - 1$, as asserted.
\end{proof}

\subsection{Application to stacks of the n-skeleton of an orthant.}
Let $S $ be a stack of $c(S)$ copies of the rank-$n$
skeleton $K^{(n)}$ of a rank-$r$ orthant $K$. The link of each cardinality-$(n-1)$ subset $Y$ of the cardinality-$r$ set
$X$ contains exactly $(r - n + 1)$ cardinality-$n$ subsets $Y'$. And $S$ contains exactly $c(S)$ orthants parallel to
$\langle Y'\rangle$.
Hence the height of disjoint union of the stacks of $S$ over the link ${{\rm Lk}}(Y)$ is  $h(S({{\rm Lk}}(Y)) = c(S)(r - n + 1)$.
Combining Proposition 4.2 with Theorem 5.4 thus yields

\begin{theorem}\label{stack of n-skeletons} If $S$ is the rank-n skeleton 
of a stack of rank-r orthants then
	\[		c(S) - 1  \le  fl  \big(\emph{pet}(S)\big)  \le  c(S)(r - n + 1) - 1.\]
\end{theorem}
\begin{corollary}
If $S$ is a stack of orthants then $fl  \big(\emph{pet}(S)\big) = c(S) - 1.$
\end{corollary}


\small
\begin{thebibliography}{References}

\bibitem[ABM14]{abm14} \textbf{Y. Atolin, J. Burillo, A. Martino.} \emph{Conjugacy in Houghton's groups.} ArXiv:1305.2044 (Preprint 2013).
\bibitem[BB97]{bb97}\textbf{M. Bestvina, N. Brady.} \emph{Morse theory and finiteness properties of groups.} Inven. Math. 129
(1997), 445-470.
\bibitem[BCK15]{bck157}\textbf{L. Bartholdi, Y. de Cornulier, D. Kochloukova.} \emph{Homolgical finiteness properties of wreath products.} Quarterly J. of Math. 66 (2015), 437-457. (Preprint: arXiv 1. July 2014) 
\bibitem[BG84]{bg84} \textbf{K.S. Brown, R.
Geoghegan.} \emph{An infinite dimensional $FP\infty$-group.} Invent. Math., 77 (1984), 367-381.
\bibitem[BGe03]{bge03} \textbf{R. Bieri, R. Geoghegan.} \emph{Connectivity Properties of Group Actions on Non-Positively Curved Spaces.} Memoirs of
the American Mathematical Society vol. 765 (Jan. 2003).
\bibitem[BGr84]{bgr84} \textbf{R. Bieri, J.R.J. Groves.} \emph{The geometry of the set of characters iduced by valuations.}	Journal für die Reine und
die Angewandte Mathematik vol. 347 (1984), 168-195.
\bibitem[BNS86]{bns86} \textbf{R. Bieri, W.D. Neumann and R. Strebel.} \emph{A geometric invariant of discrete groups.} Invent. Math. 90 (1987),
451-477.
\bibitem[BR88]{br88} \textbf{R. Bieri, B. Renz.} \emph{Valuations on free 
resolutions and higher geometric invariants of groups.} Comm. Math. Helv. 
63
(1988), 464-497.
\bibitem[BiSa16]{bisa16} \textbf{R. Bieri, H. Sach.} \emph{Groups of piecewise isometric permutations of lattice points.} arXiv: 1606:07728[math.GR] (June 2016). 
\bibitem [Bi75]{bi75} \textbf{R. Bieri.} \emph{Mayer-Vietoris sequences for HNN-groups and homological duality.} Math. Z. 143 (1975), 123-130.
\bibitem[BCMR14]{bcmr14} \textbf{J. Burillo, S Cleary, A Martino, C.E. R\"over.} \emph{Commensurations and metric properties of 	Houghton's groups.} ArXiv:1403.0026, (Preprint 2014).

\bibitem[BNR20]{bnr20} \textbf{J. Burillo, B. Nucinkis, L. Reeves.} \emph{An irrational-slope Thompson's group}, Publicacions Matemàtiques, to appear. (Preprint: ArXiv:1806.00108v3 (2020)).

\bibitem[BMN13]{bmn13} \textbf{C. Bleak, F. Matucci, M. Neunhoeffner.} \emph{Embeddings into Thompson's group $V$ and $coCF$ groups.} ArXiv:1312.1855 (2013), (Preprint 2013).
\bibitem[Br87]{br87} \textbf{K. S. Brown.} \emph{Finiteness properties of 
groups.} Journal of Pure and Algebra 44 (1987), 45-75.
\bibitem[BS73]{bs73} \textbf{A. Borel, J.-P. Serre.} \emph{Corners and arithmetic groups.} Comm. Math. Helv. 48 (1973), 436-491.
\bibitem[BS76]{bs76} \textbf{A. Borel, J.-P. Serre.} \emph{Cohomologie d'immeubles et de groupes S-arithmetiques.}	Topology 15 (1976), 211-232.

\bibitem[CV86]{cv86} \textbf{M. Culler, K. Vogtmann.} \emph{Moduli of graphs and automorphisms of free groups.}	Invent. Math. 84 (1986), 91-119.

\bibitem[FH20]{fh20} \textbf{D.S. Farley, B. Hughes} \emph{Finiteness properties of locally defined groups.} arXiv:2010.08035v1 [math.GR], 15. Oct 
2020.

\bibitem[FKS11]{fks11} \textbf{L. Funar, C. Kapoudjian, V. Sergiescou.} \emph{Asymptotically rigid mapping class groups and Thompson's groups.}
Handbook of Teichmüller Theory, vol 3. Edited by A. Papadopoulos. (arxiv:1105.0559v1[math.GR] Preprint 2011).

\bibitem[Gr93]{gr93}\textbf{P. Greenberg.} \emph{Piecewise $SL_2(\mathbb Z)$-geometry.} Trans. Amer. Math. Soc. 335 (1993), 705-720. 
  
\bibitem[Ho78]{ho78} \textbf{C. H. Houghton.} \emph{The first cohomology of a group with permutation module coefficients.} 	Arch Math. 31 (1978), 254-258.
\bibitem[KM16]{km16} \textbf{P. H. Kropholler, C. Martino.} \emph{Graph wreath products and finiteness conditions for elementary amenable groups.} 
J. Pure Appl. Algebra 220, 422-434 (2016). (Preprint: arXiv 1. Juli 2014)
\bibitem[KMN09]{kmn09} \textbf{P. H. Kropholler, C. Martinez-Pérez, B. E. A. Nucinkis.} \emph{Cohomological finiteness conditions for elementary amenable groups.} J. Reine Angew. Math. 2009:637.
\bibitem[Ka03]{ka03} \textbf{C. Kopoudjian.} \emph{From Symmetries of the 
Modular Tower of Genus Zero Real Stable 	Curves to a Euler Class for the Dyadic Circle.} Compositio Mathematica 13, 49-73 (2003).
\bibitem[Kr93]{kr93} \textbf{P. H. Kropholler.} \emph{On groups of type $(FP)_\infty$.}  J. Pure Appl. Algebra 90 (1993), 55-67.
\bibitem[Lee12]{lee12} \textbf{S. R. Lee.} \emph{Geometry of Houghton's groups.} arXiv:1212.0257 (Preprint 2012).
\bibitem[Leh08]{leh08} \textbf{J. Lehnert.} \emph{Gruppen von quasi-Automorphismen.} Goethe-Universitaet Frankfurt am Main 2008.
\bibitem[LS07]{ls07} \textbf{J. Lehnert, P. Schweitzer.} \emph{The co-word problem for the Higman-Thompson group is context-free.} Bulletin of the 
London
Mathematical Society 39(2), 235-241.
\bibitem[NSt15]{nst15} \textbf{B. Nucinkis, S. St. John-Green.} \emph{Quasi automorphisms of the infinite rooted 2-edge	coloured binary tree.} ArXiv:
1412.7715 (January 2015).
\bibitem[Os02]{os02}\textbf{D. V. Osin.}\emph{Elementary classes of groups.} Mat. Zametki 72 (2002) (1), 84-93
\bibitem[R89]{r89} \textbf{B. Renz.} \emph{Geometrische Invarianten und Endlichkeitseigenschaften von Gruppen.}	Dissertation, Frankfurt am Main 1988. 
\bibitem[Sa92]{sa92} \textbf{H. Sach.} \emph{$(FP)_n$-Eigenschaften der verallgemeinerten Houghton-Gruppen.} Diplomarbeit 1992, Goethe-Universit\"at
Frankfurt.
\bibitem[St15]{st15} \textbf{S. St. John-Green.} \emph{Centralisers in Houghton's groups.} Proceedings of the Edinburgh 	Mathematical Society (2015),  arXiv:1207.1597 (Preprint 2012)
\bibitem[WW15]{ww15} \textbf{P. Wesolek, J. Williams.} \emph{Chain conditions, elementary amenable groups, and descriptive set theory.} ArXiv1410.0975v4[math.GR] (2015)  
\bibitem[WZ14]{wz14} \textbf{S. Witzel and M. C. B. Zaremsky}, \emph{Thompson groups for systems of groups, and their finiteness properties.} Preprint, arXiv:1405.5491 (2014). 
 
\bibitem[Za15]{za15} \textbf{M.C.B. Zaremsky.} \emph{On the $\Sigma$-invariants of generalized Thompson groups and Houghton groups.} arXiv:1502.02620 (Preprint 2015).
\end{thebibliography}
\end{document}